\documentclass[11pt]{amsart}
\usepackage{epigraph}
\usepackage{amsthm}
\usepackage{bm}
\usepackage{tikz-cd}
\usepackage{amsmath}
\usepackage{amssymb}
\usepackage{relsize}
\usepackage{amsfonts}
\usepackage{amsxtra}
\usepackage{braket}
\usepackage{lmodern}
\usepackage{mathrsfs}
\usepackage{array}
\usepackage{xcolor}
\definecolor{citecol}{rgb}{0.07,0.07,0.05}
\definecolor{urlcol}{rgb}{0.06,0.04,0.09}
\definecolor{linkcol}{rgb}{0.01,0.03,0.08}
\usepackage[colorlinks,linkcolor=linkcol,urlcolor=urlcol,citecolor=citecol,pagebackref,breaklinks]{hyperref}
\usepackage{appendix} 
\usepackage[lite,abbrev,msc-links,alphabetic]{amsrefs}
\usepackage{indentfirst}
\usepackage{mathtools}
\usepackage[all]{xy}
\allowdisplaybreaks
 \numberwithin{equation}{subsection}
\theoremstyle{plain}
\newtheorem{theorem}{Theorem}[section]
\newtheorem{lemma}[theorem]{Lemma}
\newtheorem{proposition}[theorem]{Proposition}

\newtheorem{conjecture}[theorem]{Conjecture}

\theoremstyle{definition}
\newtheorem{definition}[theorem]{Definition}
\newtheorem{remark}[theorem]{Remark}

\newtheorem*{acknowledgement}{Acknowledgement}

\theoremstyle{remark}

\newcommand{\tb}[1]{\textbf{#1}}

\newcommand{\BB}{{\mathbb B}}
\newcommand{\BC}{{\mathbb C}}

\newcommand{\BF}{{\mathbb F}}

\newcommand{\BL}{{\mathbb L}}

\newcommand{\BP}{{\mathbb P}}
\newcommand{\BQ}{{\mathbb Q}}

\newcommand{\BV}{{\mathbb V}}

\newcommand{\BX}{{\mathbb X}}
\newcommand{\BY}{{\mathbb Y}}
\newcommand{\BZ}{{\mathbb Z}}

\newcommand{\CA}{{\mathcal A}}
\newcommand{\CB}{{\mathcal B}}

\newcommand{\CF}{{\mathcal F}}
\newcommand{\CG}{{\mathcal G}}

\newcommand{\CI}{{\mathcal I}}
\newcommand{\CJ}{{\mathcal J}}

\newcommand{\CN}{{\mathcal N}}

\newcommand{\CR}{{\mathcal R}}
\newcommand{\CS}{{\mathcal S}}

\newcommand{\CY}{{\mathcal Y}}
\newcommand{\CZ}{{\mathcal Z}}

\newcommand{\FA}{{\mathfrak A}}
\newcommand{\FB}{{\mathfrak B}}
\newcommand{\FC}{{\mathfrak C}}

\newcommand{\FF}{{\mathfrak F}}
\newcommand{\FG}{{\mathfrak G}}

\newcommand{\FX}{{\mathfrak X}}

\newcommand{\Fa}{{\mathfrak a}}

\newcommand{\ScF}{{\mathscr F}}

\DeclareMathOperator{\Lie}{Lie}
\DeclareMathOperator{\Charpol}{Charpol}
\DeclareMathOperator{\End}{End}

\DeclareMathOperator{\Spf}{Spf}

\DeclareMathOperator{\Ker}{Ker}

\DeclareMathOperator{\Hom}{Hom}

\DeclareMathOperator{\Nm}{Nm}

\DeclareMathOperator{\val}{val}

\DeclareMathOperator{\Tr}{Tr}
\DeclareMathOperator{\Mod}{mod}
\DeclareMathOperator{\diag}{diag}

\author{Sungyoon Cho}
\address[Sungyoon Cho]{Department of Mathematics, University of Arizona}
\email{sungyooncho@math.arizona.edu}

\title[Special cycles]{Special cycles on unitary Shimura varieties with minuscule parahoric level structure}
\date{\today}

\begin{document}

\begin{abstract}
In this paper, we formulate conjectural formulas for the arithmetic intersection numbers of special cycles on unitary Shimura varieties with minuscule parahoric level structure. Also, we prove that these conjectures are compatible with all known results.
\end{abstract}

\maketitle
\tableofcontents{}

\section{Introduction}\label{section1}
In \cite{KR3}, Kudla and Rapoport made a conjectural formula for the arithmetic intersection numbers of special cycles on Shimura varieties attached to unitary groups of signature $(1,n-1)$ with hyperspecial level structure (which corresponds to the case of good reduction). This conjecture, the so-called Kudla-Rapoport conjecture, relates the arithmetic intersection numbers of special cycles to Fourier coefficients of a certain Siegel-Eisenstein series. In \cite{KR2}, they made a local analogue of the Kudla-Rapoport conjecture which relates the arithmetic intersection numbers of special cycles on unitary Rapoport-Zink spaces to the derivatives of representation densities of hermitian forms and they showed that the local-version of the Kudla-Rapoport conjecture induces the global-version the conjecture (in hyperspecial level structure case). Recently, this conjecture was proved by Li and Zhang in \cite{LZ}.

In the present paper, we make conjectural formulas for the arithmetic intersection numbers of special cycles on unitary Rapoport-Zink spaces which correspond to unitary Shimura varieties (more precisely, the Rapoport-Smithling-Zhang's unitary Shimura varieties in \cite{RSZ2}) with minuscule parahoric level structure via non-archimedean uniformization of \cite{RZ} (more precisely, \cite[Section 4.3]{Cho}). This can be regarded as a part of the Kudla program (\cite{Kud}). Also, this will have applications to the arithmetic inner product formula (\cite{Liu1}, \cite{Liu2}).

We now describe our results in more detail. We fix an odd prime $p$. Let $F_v$ be an unramified extension of $\BQ_p$ with ring of integers $O_{F_v}$ and residue field $\BF_q$. Let $\pi$ be its uniformizer (this is just $p$, but we use the general notation for future use). Let $E_v$ be a quadratic unramified extension of $F_v$ with ring of integers $O_{E_v}$. We write $\breve{E}_v$ for the completion of a maximal unramified extension of $E_v$.

Let $h,n$ be integers such that $0 \leq h \leq n$. We then consider the unitary Rapoport-Zink space $\CN_{E_v/F_v}^h(1,n-1)$ over $\Spf O_{\breve{E}_v}$ (see Section \ref{section2}). Here $h$ is related to the level structure of the corresponding unitary Shimura variety. In particular, $\CN^0_{E_v/F_v}(1,n-1)$ is the unitary Rapoport-Zink space which uniformizes the basic locus of the unitary Shimura variety with hyperspecial level structure.

We have an isomorphism $\theta:\CN^h_{E_v/F_v}(1,n-1) \rightarrow \CN^{n-h}_{E_v/F_v}(1,n-1)$. This isomorphism will play an important role to formulate our conjectural formula.

Let $(\overline{\BY},i_{\overline{\BY}},\lambda_{\overline{\BY}})$ (resp. $(\BX,i_{\BX},\lambda_{\BX})$) be the framing object of $\CN^0_{E_v/F_v}(0,1)$ (resp. $\CN^h_{E_v/F_V}(1,n-1)$). Then consider a $E_v$-vector space
\begin{equation*}
\BV:=\Hom_{O_{E_v}}(\overline{\BY},\BX) \otimes_{\BZ} \BQ.
\end{equation*}
Also, we define a hermitian form $h$ on this space as
\begin{equation*}
h(x,y)=\lambda_{\overline{\BY}}^{-1}\circ y^{\vee} \circ \lambda_{\BX} \circ x \in \End_{O_{E_v}}(\overline{\BY})\otimes \BQ \simeq E_v \quad \forall x,y \in \BV,
\end{equation*}
where $y^{\vee}$ is the dual of $y$.

For each $x \in \BV$, we can attach a certain relative divisor $\CZ(x)$ which is called a special cycle or the Kudla-Rapoport cycle.

Also, for each $y \in \BV$, one can consider a relative divisor $\CZ(\lambda_{\BX} \circ y)$ in $\CN^{n-h}_{E_v/F_v}(1,n-1)$. Then by using the isomorphism $\theta:\CN^h_{E/F}(1,n-1) \rightarrow \CN^{n-h}_{E/F}(1,n-1)$, we can define a relative divisor $\CY(y):=\theta^{-1}(\CZ(\lambda_{\BX}\circ y))$ (see Definition \ref{definition2.2}).

These $\CY$-cycles only appear in the case of $h \neq 0$. Indeed, $\CZ(x)=\CY(y)$ when $h=0$. This new definition of $\CY$-cycles and its properties are important to formulate conjectural formulas for the arithmetic intersection numbers of $\CZ$ and $\CY$-cycles. Before we talk about our conjecture, we need to talk about the Kudla-Rapoport conjecture.

The Kudla-Rapoport conjecture is a formula for the arithmetic intersection numbers of special cycles $\CZ(x_1), \dots, \CZ(x_n)$ in $\CN^0_{E/F}(1,n-1)$. This was formulated in \cite{KR2} as follows.

\begin{conjecture}(Kudla-Rapoport conjecture)
	Let $x_1, \dots, x_n \in \BV$ and let $B$ be the matrix $(h(x_i,x_j))_{1 \leq i,j \leq n}$. Then the arithmetic intersection number of $\CZ(x_1), \dots \CZ(x_n)$ in $\CN^0_{E/F}(1,n-1)$ is
	\begin{equation*}
	\begin{array}{ll}
	\langle \CZ(x_1), \dots, \CZ(x_n) \rangle &:=\chi(O_{\CZ(x_1)} \otimes^{\BL} \dots \otimes^{\BL} O_{\CZ{(x_n)}})\\
	&=\dfrac{\alpha'(1_n,B)}{\alpha(1_n,1_n)}.
	\end{array}
	\end{equation*}
	Here $\chi$ is the Euler-Poincare characteristic and $\otimes^{\BL}$ is the derived tensor product. We denote by $\alpha$ the representation density and $\alpha'$ the derivative of the representation density.	
\end{conjecture}

Recently, this conjecture was proved by Li and Zhang in \cite{LZ}.

Now, let us talk about our conjecture. This is a conjecture on the arithmetic intersection numbers of arbitrary special cycles 
\begin{equation*}
\CZ(x_1), \dots, \CZ(x_{n-t}), \CY(y_1), \dots, \CY(y_{t})
\end{equation*}
in $\CN^h_{E_v/F_v}(1,n-1)$, for $0 \leq t \leq n$.

In \cite{Cho}, we proved the following proposition.

\begin{proposition}\label{proposition}(\cite[Proposition 5.10]{Cho})
	\begin{enumerate}
		\item Let $x \in \BV$. If $h(x,x)$ has valuation $0$, then 
		\begin{equation*}
		\CZ(x) \simeq \CN^h_{E_v/F_v}(1,n-2).
		\end{equation*}
		
		\item Let $y \in \BV$. If $h(y,y)$ has valuation $-1$, then
		\begin{equation*}
		\CY(y) \simeq \CN^{h-1}_{E_v/F_v}(1,n-2).
		\end{equation*}
		
	\end{enumerate}
\end{proposition}

This proposition has two important remarks.

First, we can regard the arithmetic intersection numbers of special cycles in $\CN^h_{E_v/F_v}(1,n-1)$ as those in $\CN^n_{E_v/F_v}(1,2n-1)$. For example, the arithmetic intersection number of $\CZ(x_1), \dots, \CZ(x_n)$ in $\CN^0_{E_v/F_v}(1,n-1)$ is the arithmetic intersection number of $\CZ(x_1),\dots,\CZ(x_n),\CY(y_1),\dots,\CY(y_n)$ in $\CN^n_{E_v/F_v}(1,2n-1)$, where $y_1, \dots, y_n$ have $-1$ valuations.

Second, $\CY(y)$ are not empty when $h(y,y)$ has $-1$ valuation. This means that we need a certain $-1$ valuation on the analytic side. This is one important difference between hyperspecial level case ($h=0$) and minuscule parahoric level structure case ($h \neq 0$).

Therefore, it suffices to formulate conjectural formulas for the arithmetic intersection numbers of $\CZ$ and $\CY$-cycles in $\CN^n_{E_v/F_v}(1,2n-1)$.

First step is the arithmetic intersection number 
\begin{equation}\label{eq1}
\langle \CZ(x_1), \dots,\CZ(x_n),\CY(y_1),\dots,\CY(y_n) \rangle.
\end{equation}

Let $B$ be the following $2n \times 2n$-matrix.
\begin{equation*}
B=\left(\begin{array}{ll}
h(x_i,x_i) & h(x_i,y_l) \\
h(y_k,x_j) & h(y_k, y_l)
\end{array}\right)_{1 \leq i,j,l,k \leq n}.
\end{equation*}

Then, a conjectural formula for \eqref{eq1} should satisfy the following three properties.

\begin{enumerate}
	\item \eqref{eq1} should be compatible with the Kudla-Rapoport conjecture.
	
	When $B=\diag(B_1,\pi^{-1} 1_n)$, $\langle \CZ(x_1),\dots,\CZ(x_n),\CY(y_1),\dots,\CY(y_n) \rangle$ can be reduced to $\langle \CZ(x_1), \dots \CZ(x_n) \rangle$ in $\CN^0_{E_v/F_v}(1,n-1)$ by $n$-times reduction in Proposition \ref{proposition}. Therefore,
	\begin{equation*}
	\langle \CZ(x_1),\dots,\CZ(x_n),\CY(y_1),\dots,\CY(y_n) \rangle=\dfrac{\alpha'(1_n,B_1)}{\alpha(1_n,1_n)}.
	\end{equation*}
	
	\item \eqref{eq1} should be invariant under $GL_n(O_{E_v}) \times GL_n(O_{E_v})$-action.
	
	In $\CN^0_{E_v/F_v}(1,n-1)$, $\langle \CZ(x_1), \dots, \CZ(x_n) \rangle$ is invariant under $GL_n(O_{E_v})$-action on $\lbrace x_1, \dots, x_n \rbrace$. In our case, $\langle \CZ(x),\CY(y) \rangle \neq \langle \CZ(y),\CY(x) \rangle$ since $\CZ$-cycle is empty when $h(x,x)$ has $-1$ valuation, but $\CY$-cycle is not empty. But, it seems that we can change $\lbrace x_1, \dots, x_n \rbrace$ to an element in its $GL_n(O_{E_v})$-orbit and $\lbrace y_1, \dots, y_n \rbrace$ to an element in its $GL_n(O_{E_v})$-orbit. This implies that \eqref{eq1} cannot be a sum of usual representation densities.
	
	\item \eqref{eq1} should be invariant under $\theta:\CN:=\CN^n_{E_v/F_v}(1,2n-1) \simeq \widehat{\CN}:=\CN^n_{E_v/F_v}(1,2n-1)$.
	
	Recall that we have an isomorphism $\theta :\CN \simeq \widehat{\CN}$. This changes $\CZ$-cycles to $\CY$-cycles, $\CY$-cycles to $\CZ$-cycles. More precisely,
	\begin{equation*}
	\theta(\CZ(x_i))=\CY(x_i'),
	\end{equation*}
	for some $x_i'$ such that $\lambda'_{\BX} \circ x_i'=x_i$, where $\lambda'_{\BX}$ is the polarization of the framing object of $\widehat{\CN}$ and
	\begin{equation*}
	\theta(\CY(y_i))=\CZ(\lambda_{\BX}\circ y_i).
	\end{equation*}
	
	Since $\theta$ is an isomorphism, we have
	\begin{equation*}
	\begin{array}{l}
	\langle \CZ(x_1),\dots,\CZ(x_n),\CY(y_1),\dots,\CY(y_n) \rangle\\
	=\langle \theta(\CZ(x_1)),\dots,\theta(\CZ(x_n)),\theta(\CY(y_1)),\dots,\theta(\CY(y_n)) \rangle\\
	=\langle \CZ(\lambda_{\BX}\circ y_1),\dots,\CZ(\lambda_{\BX}\circ y_n),\CY(x_1'),\dots,\CY(x_n') \rangle.
	\end{array}
	\end{equation*}
	
	This changes the matrix $B$. More precisely, if
	\begin{equation*}
	B=\left(\begin{array}{ll}
	E & F\\
	G& H
	\end{array}\right),
	\end{equation*}
	for $n\times n$-matrices $E,F,G,H$, then we have
	\begin{equation*}
	B^{\vee_n}:=\left(\begin{array}{cc}
	h'(\lambda_{\BX}\circ y_i, \lambda_{\BX}\circ y_j) & h'(\lambda_{\BX}\circ y_i,x_l')\\
	h'(x_k',\lambda_{\BX}\circ y_j)& h'(x_k',x_l')
	\end{array}
	\middle)=\middle(\begin{array}{ll}
	\pi H & G\\
	F& \pi^{-1} E
	\end{array}
	\right),
	\end{equation*} where $h'$ is the hermitian form induced by $\lambda'_{\BX}$.
	
	Therefore, our formula should be invariant under the change $B \leftrightarrow B^{\vee_n}$.
\end{enumerate}

We want to find the Fourier coefficients of a certain Siegel-Eisenstein series satisfying all of the above three properties.

In the case of $\CN^0_{E_v/F_v}(1,n-1)$, the representation density $\alpha'(1_n,B)$ is related to the arithmetic intersection number $\langle \CZ(x_1), \dots, \CZ(x_n) \rangle$ because it is associated to the local Whittaker function of $\widetilde{1_{L_0^n}}$. Here, $\widetilde{1_{L_0^n}}$ is the Siegel-Weil section associated to a characteristic function $1_{L_0^n}$, where $L_0$ is an $O_{E_v}$-lattice of rank $n$ with hermitian form $1_n$.

In the case of $\CN^n_{E_v/F_v}(1,2n-1)$, we need to consider a function $W_{n,n}(B,r)$ which is associated to the local Whittaker function of $\widetilde{1_{n,n}}$. Here, $\widetilde{1_{n,n}}$ is the Siegel-Weil section associated to the characteristic function $1_{n,n}=1_{(L_n^{\vee})^n \oplus (L_n)^n}$. Here $L_n$ is an $O_{E_v}$-lattice of rank $2n$ with the hermitian form
\begin{equation*}
\left(\begin{array}{ll}
1_n & 0\\
0 & \pi^{-1} 1_n
\end{array}\right),
\end{equation*}
and $L_n^{\vee}$ is its dual lattice. Note that we have a $-1$ valuation on the matrix. This is because we should allow $-1$ valuation on $\CY$-cycles.

Now, we can naively guess that \eqref{eq1} $\sim W_{n,n}'(B,0)$. But, this is not true. In the case of bad reductions, we need some correction terms for \eqref{eq1} (cf. \cite{KR5} \cite{KRY}, \cite{San}, \cite{RSZ3}, \cite{RSZ1}). In many previous papers (\cite{KR5} \cite{KRY}, \cite{San}), these are some linear combination of local Whittaker functions at $r=0$.

We have one clue to find the correction terms. If $B=\diag(B_1, \pi^{-1} 1_n)$, then $W_{n,n}'(B,0)$ is already equal to $\alpha'(1_n,B_1)$ up to a constant. Therefore, $W_{n,n}'(B,0)$ already gives the compatibility with the Kudla-Rapoport conjecture. This means that we need a correction term such that it vanishes when $B=\diag(B_1, \pi^{-1} 1_n)$.

For this reason, we chose a linear combination of functions $W_{n,t}(B,0)$ which are associated to the local Whittaker functions of $\widetilde{1_{n,t}}$, $t \leq n-1$, where $\widetilde{1_{n,t}}$ is the Siegel-Weil section associated to the characteristic function $1_{n,t}=1_{(L_t^{\vee})^n \oplus (L_t)^n}$. Here $L_t$ is an $O_{E_v}$-lattice of rank $2n$ with the hermitian form
\begin{equation*}
\left(\begin{array}{ll}
1_{2n-t} & 0\\
0 & \pi^{-1} 1_t
\end{array}\right),
\end{equation*}
These terms vanishes when $B=\diag(B_1, \pi^{-1} 1_n)$.

Then, one can guess \eqref{eq1}$\sim$ $W_{n,n}'(B,0)-\mathlarger{\sum}_{0 \leq i \leq n-1} \beta_i^nW_{n,i}(B,0)$ for some constants $\beta_i^n$. Now, how can we determine these constants $\beta_i^n$? Here, we use the third property of \eqref{eq1}.

\begin{theorem}(Theorem \ref{theorem3.14})
There are the unique constants 
\begin{equation*}
\beta_{0}^n, \dots, \beta_{n-1}^n, \delta_h,
\end{equation*}
in $E_v$ such that
\begin{equation*}
\begin{array}{l}
W_{n,n}'(B,0)-W_{n,n}'(B^{\vee_n},0)\\
=\mathlarger{\sum}_{0 \leq i \leq n-1} \beta_i^nW_{n,i}(B,0)
-\mathlarger{\sum}_{0 \leq j \leq n-1}\beta_j^nW_{n,j}(B^{\vee_n},0)
+\delta_hW_{n,n}(B,0).
\end{array}
\end{equation*}
\end{theorem}
Indeed, we can compute all constants $\beta_i^n$ explicitly. Now, our conjecture is as follows.

\begin{conjecture}\label{conjecture0.1}(Conjecture \ref{conjecture1})
	\begin{equation*}
	\begin{array}{l}
	\langle \CZ(x_1), \dots,\CZ(x_n),\CY(y_1),\dots,\CY(y_n) \rangle\\
	=\dfrac{1}{W_{n,n}(A_n,0)}\lbrace W'_{n,n}(B,0)-\mathlarger{\sum}_{0 \leq i \leq n-1} \beta_i^nW_{n,i}(B,0)\rbrace.
	\end{array}
	\end{equation*}
	Here, $A_n=\diag(1_n,\pi^{-1}1_n)$.
\end{conjecture}

We prove that this holds when $n=1$ (Theorem \ref{theorem4.8}).

Now, in the case of the arithmetic intersection number of arbitrary special cycles
\begin{equation*}
\langle \CZ(x_1), \dots,\CZ(x_{2n-h}),\CY(y_1),\dots,\CY(y_h) \rangle
\end{equation*} in $\CN^n(1,2n-1)$ ($ 0 \leq h \leq 2n$), we can formulate a similar conjecture. The only differences are characteristic functions and the dual relation under $\theta$.

Let $\widetilde{1_{h,t}}$ be the Siegel-Weil section associated to the characteristic function $1_{(L_t^{\vee})^{2n-h} \oplus (L_t)^{h}}$. We consider the functions $W_{h,t}(B,0)$ which are associated to the local Whittaker functions of $\widetilde{1_{h,t}}$. Also, we define $B^{\vee_h}$ as follows. If
\begin{equation*}
B=\left(\begin{array}{ll}
h(x_i,x_j) & h(x_i,y_l)\\
h(y_k,x_j)& h(y_k,y_l)
\end{array}\middle)_{1 \leq i,j \leq 2n-h, 1 \leq k,l \leq h}=\middle(\begin{array}{ll}
E & F\\
G& H
\end{array}
\right),
\end{equation*}
for $(2n-h)\times (2n-h)$-matrix $E$, $(2n-h) \times h$-matrix $F$, $h \times (2n-h)$-matrix $G$, and $h \times h$-matrix $H$, then we define
\begin{equation*}
B^{\vee_h}:=\left(\begin{array}{cc}
h(\lambda_{\BX}\circ y_i, \lambda_{\BX}\circ y_j) & h(\lambda_{\BX}\circ y_i,x_l')\\
h(x_k',\lambda_{\BX}\circ y_j)& h(x_k',x_l')
\end{array}
\middle)=\middle(\begin{array}{ll}
\pi H & G\\
F& \pi^{-1}E
\end{array}
\right)
\end{equation*}

Then, we have the following theorem.

\begin{theorem}(Theorem \ref{theorem3.14})
There are the unique constants 
\begin{equation*}
\beta_{0}^h, \dots, \beta_{n-1}^h, \beta_0^{2n-h},\dots,\beta_{n-1}^{2n-h},\delta_h,
\end{equation*}
in $E_v$ such that
\begin{equation*}
\begin{array}{l}
W_{h,n}'(B,0)-W_{2n-h,n}'(B^{\vee_h},0)\\
=\mathlarger{\sum}_{0 \leq i \leq n-1} \beta_i^hW_{h,i}(B,0)
-\mathlarger{\sum}_{0 \leq j \leq n-1}\beta_j^{2n-h}W_{2n-h,j}(B^{\vee_h},0)
+\delta_hW_{h,n}(B,0).
\end{array}
\end{equation*}
\end{theorem}
All of these constants can be computed explicitly.

Now, our conjecture on the arithmetic intersection number \begin{equation*}
\langle \CZ(x_1), \dots,\CZ(x_{2n-h}),\CY(y_1),\dots,\CY(y_{h}) \rangle
\end{equation*}
 is as follows.

\begin{conjecture}\label{conjecture0.2}(Conjecture \ref{conjecture2}) For $0 \leq h \leq 2n$, we have
	\begin{equation*}
	\begin{array}{l}
	\langle \CZ(x_1), \dots,\CZ(x_{2n-h}),\CY(y_1),\dots,\CY(y_{h}) \rangle\\
	=\dfrac{1}{W_{n,n}(A_n,0)}\lbrace W'_{h,n}(B,0)-\mathlarger{\sum}_{0 \leq i \leq n-1} \beta_i^hW_{h,i}(B,0)\rbrace.
	\end{array}
	\end{equation*}
\end{conjecture}

By the construction of the conjectures, these are compatible with the Kudla-Rapoport conjecture. Also, we prove that these conjectures hold when $n=1$, $0 \leq h \leq 2$ and we prove that these are compatible with the result of Li and Zhang in \cite{LZ} (almost-self dual case). Therefore, Conjecture \ref{conjecture0.1} and Conjecture \ref{conjecture0.2} are compatible with all known results. 

Now, we explain the outline of the paper.

In Section \ref{section2}, we review some geometric background on special cycles on unitary Rapoport-Zink spaces. Also, we state some properties of the arithmetic intersection numbers of special cycles. In Section \ref{section3}, we define the functions $W_{h,t}(B,r)$ and formulate our conjectures. Here, we use the result of \cite{Hir} with some modification. In Section \ref{section4}, we prove that the conjectures hold when $n=1$. Here, we use the result of \cite{San}. In Appendix \ref{appendix}, we prove that our conjecture is compatible with the formula of the arithmetic intersection number of $\CZ(x_1),\dots,\CZ(x_n)$ in $\CN^1_{E_v/F_v}(1,n-1)$ which is proved by Li and Zhang in \cite{LZ} (this is the almost self-dual case in loc. cit.).

\bigskip
\begin{acknowledgement}
I would like to thank Yifeng Liu for many helpful discussions. In particular, the discussions with him were very helpful to formulate our conjectures.
\end{acknowledgement}

\bigskip
\section{Unitary Rapoport-Zink spaces and special cycles}\label{section2}
In this section, we recall some facts and definitions of unitary (relative) Rapoport-Zink spaces $\CN^h_{E_v/F_v}(1,n-1)$ and special cycles $\CZ, \CY$ from \cite{Cho}.
\subsection{The moduli space $\CN^h_{E_v/F_v}(1,n-1)$}
We fix a prime $p>2$. Let $F_v$ be an unramified finite extension of $\BQ_p$ with ring of integers $O_{F_v}$, and residue field $\BF_q$. We fix a uniformizer $\pi$ of $O_{F_v}$ (this is just $p$, but we will use the general notation for future use). Let $E_v$ be a quadratic unramified extension of $F_v$ with ring of integers $O_{E_v}$. We write $\breve{E_v}$ for the completion of a maximal unramified extension of $E_v$ and $O_{\breve{E_v}}$ for its ring of integers. We denote by $\BF_{q^2}$ (resp. $k$) the residue field of $O_{E_v}$ (resp. $O_{\breve{E_v}}$). Denote by $^*$ the nontrivial Galois automorphism of $E_v$ over $F_v$.

Now, we fix an integer $0 \leq h \leq n$. To define our moduli space $\CN^h_{E_v/F_v}(1,n-1)$, we need to fix a triple ($\BX$, $i_{\BX}$, $\lambda_{\BX}$) consisting of the following data:

(1) $\mathbb{X}$ is a supersingular strict formal $O_{F_v}$-module of $F_v$-height $2n$ over $\BF_{q^2}$ (we refer to \cite[Definition 2.1]{Cho} for the precise definition of the supersingular strict formal $O_{F_v}$-modules);

(2) $i_{\BX}:O_{E_v} \rightarrow \End \BX$ is an $O_{E_v}$-action on $\BX$ that extends the $O_{F_v}$-action on $\BX$;

(3) $\lambda_{\BX}$ is a polarization
\begin{equation*}
\lambda_{\BX} :\BX \rightarrow \BX^{\vee},
\end{equation*}
such that the corresponding Rosati involution induces the involution $^*$ on $O_{E_v}$.

We also impose the following conditions:

(a) For all $a \in O_{E_v}$, the action $i_{\BX}$ satisfies
\begin{equation*}
\Charpol(i_{\BX}(a) \vert \Lie \BX) = (T-a)(T-a^*)^{n-1}.
\end{equation*}

Here, we view $(T-a)(T-a^*)^{n-1}$ as an element of $O_S[T]$ via the structure morphism;

(b) We assume that $\Ker \lambda_{\BX} \subset \BX[\pi]$ and its order is $q^{2h}$.

Now, let (Nilp) be the category of $O_{E_v}$-schemes $S$ such that $\pi$ is locally nilpotent on $S$. Then the functor $N^h_{E_v/F_v}(1,n-1)$ associates to a scheme $S \in$ (Nilp) the set of isomorphism classes of tuples $(X,i_X,\lambda_X,\rho_X)$.

Here $X$ is a supersingular strict formal $O_{F_v}$-module of $F_v$-height $2n$ over S and $i_X$ is an $O_{E_v}$-action on $X$ satisfying the following determinant condition:
For all $a \in E_v$,
\begin{equation*}
\Charpol(i_X(a)\vert \Lie X)=(T-a)(T-a^*)^{n-1} \in O_S[T].
\end{equation*}

$\rho_X$ is an $O_{E_v}$-linear quasi-isogeny \begin{equation*}
\rho_X:X_{\overline{S}} \rightarrow \mathbb{X}\times_{\BF_{q^2}}\overline{S},
\end{equation*}
of height 0, where $\overline{S}=S\times_{O_{E_v}}\BF_{q^2}$ and $X_{\overline{S}}$ is the base change $X\times_S\overline{S}$.

Finally, $\lambda_X:X \rightarrow X^{\vee}$ is a polarization
such that its Rosati involution induces the involution $^*$ on $O_{E_v}$, and the following diagram commutes up to a constant in $O_{F_v}^{\times}$

\begin{center}
	\begin{tikzcd}
		X_{\overline{S}} \arrow{r}{\lambda_{X_{\overline{S}}}}
		\arrow{d}{\rho_X}
		&X^{\vee}_{\overline{S}} \\
		\mathbb{X}_{\overline{S}}
		\arrow{r}{\lambda_{\BX_{\overline{S}}}}  &\mathbb{X}_{\overline{S}}^{\vee}
		\arrow{u}{\rho_X^{\vee}}.
	\end{tikzcd}
\end{center}

Two quadruples $(X,i_X,\lambda_X,\rho_X)$ and $(X',i_{X'},\lambda_{X'},\rho_{X'})$ are isomorphic if there exists an $O_{E_v}$-linear isomorphism $\alpha:X \rightarrow X'$ such that $\rho_{X'} \circ (\alpha \times_S \overline{S})=\rho_X$ and $\alpha^{\vee} \circ \lambda_{X'} \circ \alpha$ differs locally on $S$ from $\lambda_X$ by a scalar in $O_{F_v}^{\times}$. \\

Then the functor $\CN^h_{E_v/F_v}(1,n-1) \otimes O_{\breve{E_v}}$ is representable by a formal scheme over $\Spf O_{\breve{E_v}}$ which is locally formally of finite type. Also, this formal scheme is regular (see \cite[Proposition 3.33]{Cho}). For simplicity, denote by $\CN$ the moduli space $\CN^h_{E_v/F_v}(1,n-1) \otimes O_{\breve{E_v}}$.

\subsection{Special cycles}\label{subsec:section2.2}

In this subsection, we recall the definitions of special cycles from \cite[Section 5]{Cho}. We denote by $val$ the valuation of $E_v$. We fix a triple $(\overline{\BY},i_{\overline{\BY}},\lambda_{\overline{\BY}})$. Here $\overline{\BY}$ is a strict formal $O_{F_v}$-module of $F_v$-height 2 over $k$, and $i_{\overline{\BY}}$ is an $O_{E_v}$-action satisfying the determinant condition of signature $(0,1)$. Also, $\lambda_{\overline{\BY}}$ is a principal polarization. Then, we can consider the corresponding moduli space $\CN^0(0,1)$. For simplicity, we write $\CN^0$ for $\CN^0(0,1) \otimes O_{\breve{E_v}}$, $\CN$ for $\CN^h(1,n-1) \otimes O_{\breve{E_v}}$ and $\widehat{\CN}$ for $\CN^{n-h}(1,n-1) \otimes O_{\breve{E_v}}$.

\begin{definition}$\quad$
	\begin{enumerate}
		\item (\cite[Definition 3.1]{KR2})	The space of special homomorphisms is the $E_v$-vector space
		\begin{equation*}
		\BV:=\Hom_{O_{E_v}}(\overline{\BY},\BX)\otimes_{\BZ}\BQ.
		\end{equation*}
		For $x,y \in \BV$, we define a hermitian form $h$ on $\BV$ as
		\begin{equation*}
		h(x,y)=\lambda_{\overline{\BY}}^{-1} \circ y^{\vee}  \circ \lambda_{\BX} \circ x \in \End_{O_{E_v}}(\overline{\BY}) \otimes \BQ \overset{i_{\overline{\BY}}^{-1}}{\simeq} E_v.
		\end{equation*}
		We often omit $i_{\overline{\BY}}^{-1}$ via the identification $\End_{O_{E_v}}(\overline{\BY}) \otimes \BQ {\simeq} E_v$.
	
		\item We denote by $\theta: \CN \rightarrow \widehat{\CN}$ the isomorphism which is defined in \cite[Remark 5.2]{Cho}. More precisely, For each $O_{\breve{E_v}}$-scheme $S$ such that $\pi$ is locally nilpotent, the isomorphism $\theta$ sends $(X,i_X,\lambda_X,\rho_X) \in \CN(S)$ to \begin{equation*}
		(X^{\vee}, \overline{i}_X^{\vee},\lambda'_X,(\rho_X^{\vee})^{-1}) \in \widehat{\CN}(S).
		\end{equation*}
		Here $\lambda'_X : X^{\vee} \rightarrow X$ is the unique polarization such that $\lambda'_X \circ \lambda_X =i_X(\pi)$, and for $a \in O_{E_v}$, the action $\overline{i}_X^{\vee}$ is defined as $\overline{i}_X^{\vee}(a):=i_X(a^*)^{\vee}$.
	\end{enumerate}
\end{definition}

\begin{definition}\label{definition2.2}(\cite[Definition 3.2]{KR2}, \cite[Definition 5.4]{Cho})
		$\quad$
	\begin{enumerate}
		\item For $x \in \BV$, we define the special cycle $\CZ(x)$ to be the closed formal subscheme of $\CN^0 \times \CN$ with the following property: For each $O_{\breve{E_v}}$-scheme $S$ such that $\pi$ is locally nilpotent, $\CZ(x)(S)$ is the set of all points $\xi=(\overline{Y},i_{\overline{Y}},\lambda_{\overline{Y}},\rho_{\overline{Y}},X,i_X,\lambda_X,\rho_X)$ in $(\CN^0 \times \CN)(S)$ such that the quasi-homomorphism
		\begin{equation*}
		\rho_X^{-1}\circ x \circ \rho_{\overline{Y}}:\overline{Y} \times_S \overline{S} \rightarrow X \times_S \overline{S}
		\end{equation*}
		extends to a homomorphism from $\overline{Y}$ to $X$.
		
		\item For $y \in \BV$, we define the special cycle $\CY(y)$ in $\CN^0 \times \CN$ as follows. First, consider the special cycle $\CZ(\lambda_{\BX}\circ y)$ in $\CN^0 \times \widehat{\CN}$. This is the closed formal subscheme of $\CN^0 \times \widehat{\CN}$. We define $\CY(y)$ as $(id \times \theta^{-1})(\CZ(\lambda_{\BX}\circ y))$ in $\CN^0 \times \CN$.
	\end{enumerate}
	
	Note that $\CN^0$ can be identified with $\Spf O_{\breve{E_v}}$, and hence $\CZ(x), \CY(y)$ can be identified with closed formal subschemes of $\CN$. 
	
\end{definition}

\begin{proposition}\label{proposition2.3}(\cite[Proposition 5.9]{Cho})
		The functors $\CZ(x)$ and $\CY(y)$ are represented by closed formal subschemes of $\CN^0 \times \CN$. In fact, $\CZ(x)$ and $\CY(y)$ are relative divisors in $\CN^0 \times \CN$ (or empty) for any $x, y \in \BV \backslash \lbrace 0 \rbrace$.	
\end{proposition}

In the next subsection, we will use the following propositions.

\begin{proposition}\label{proposition2.4}(\cite[Proposition 5.10]{Cho}) $\quad$
	\begin{enumerate}
		\item If $val(h(x,x))=0$, then $\CZ(x) \simeq \CN^{h}_{E_v/F_v}(1,n-2)_{O_{\breve{E_v}}}$.
		\item If $val(h(y,y))=-1$, then $\CY(y) \simeq \CN^{h-1}_{E_v/F_v}(1,n-2)_{O_{\breve{E_v}}}$.
	\end{enumerate}
\end{proposition}

\begin{proposition}\label{proposition2.5}(\cite[Proposition 5.11]{Cho})
	Assume that $val(h(x,x))=0$ and $val(h(y,y))=-1$. Assume further that by rescaling as in the proof of \cite[Proposition 5.10]{Cho}, $x^*\circ x=1$, and $ (\lambda_{\BX} \circ y)^*\circ (\lambda_{\BX} \circ y)=1$. We define $e_x:=x \circ x^*$ and $e_y:=(\lambda_{\BX} \circ y) \circ (\lambda_{\BX} \circ y)^*$. Fix isomorphisms
	\begin{equation*}
	\begin{split}
	\Phi:\CZ(x) \simeq \CN^h_{E_v/F_v}(1,n-2)_{O_{\breve{E_v}}},\\
	\Psi:\CY(y) \simeq \CN^{h-1}_{E_v/F_v}(1,n-2)_{O_{\breve{E_v}}},
	\end{split}
	\end{equation*}
	as in Proposition \ref{proposition2.4}. Then the following statements hold.
	\begin{enumerate}
		\item For $z\in \BV$ such that $h(x,z)=0$, let $z':=(1-e_x)\circ z$. Then, we have $\Phi(\CZ(x) \cap \CZ(z))=\CZ(z')$ in $\CN^h_{E_v/F_v}(1,n-2)$ and $h(z',z')=h(z,z)$.
		
		\item For $w\in \BV$ such that $h(x,w)=0$, let $w':=(1-e_x)\circ w$. Then, we have $\Phi(\CZ(x) \cap \CY(w))=\CY(w')$ in $\CN^h_{E_v/F_v}(1,n-2)$ and $h(w',w')=h(w,w)$.
		
		\item For $z\in \BV$ such that $h(y,z)=0$, let $z':=(1-e_y^{\vee})\circ z$. Then, we have $\Psi(\CY(y) \cap \CZ(z))=\CZ(z')$ in $\CN^{h-1}_{E_v/F_v}(1,n-2)$ and $h(z',z')=h(z,z)$.
		
		\item For $w\in \BV$ such that $h(y,w)=0$, let $w':=(1-e_y^{\vee})\circ w$. Then, we have $\Psi(\CY(y) \cap \CY(w))=\CY(w')$ in $\CN^{h-1}_{E_v/F_v}(1,n-2)$ and $h(w',w')=h(w,w)$.
	\end{enumerate}
\end{proposition}

\bigskip
\subsection{Arithmetic intersection numbers of special cycles}
In this subsection, we use the notation in Section \ref{subsec:section2.2}. For each $x,y \in \BV$, the special cycles $\CZ(x), \CY(y)$ can be identified with relative divisors in the regular formal scheme $\CN$ by Definition \ref{definition2.2} and Proposition \ref{proposition2.3}. In this subsection, we assume that $\CN=\CN^n_{E_v/F_v}(1,2n-1)$.

\begin{definition}
	Let $[\textbf{x},\textbf{y}]=[x_1, \dots, x_{2n-t}, y_1, \dots, y_{t}]$ be a basis of $\BV$. For special cycles $\CZ(x_i), \CY(y_j)$, $1 \leq i \leq 2n-t,  1\leq j \leq t$, we denote by $\CI_t^n(\textbf{x},\textbf{y})$ their arithmetic intersection number in $\CN$. More precisely,
	\begin{equation*}
	\CI_t^n(\textbf{x},\textbf{y}):=\chi(O_{\CZ(x_1)} \otimes^{\BL}_{O_{\CN}} \dots \otimes^{\BL}_{O_{\CN}} O_{\CZ(x_{2n-t})} \otimes^{\BL}_{O_{\CN}}O_{\CY(y_1)} \otimes^{\BL}_{O_{\CN}} \dots \otimes^{\BL}_{O_{\CN}} O_{\CY(y_{t})}).
	\end{equation*}
	Here, we write $\chi$ for the Euler-Poincare characteristic and $\otimes^{\BL}$ for the derived tensor product of $O_{\CN}$-modules.
\end{definition}

In the next section, we will formulate a conjectural formula for arithmetic intersection numbers $\CI_t^n(\tb{x},\tb{y})$. First, let us consider the intersection number $\CI_n^n(\tb{x},\tb{y})$. Let $B$ and $B^{\vee}$ be the matrices
	\begin{displaymath}
B=\left(\begin{array}{cc} 
h(x_i,x_j) & h(x_i,y_l)\\
h(y_k,x_j)& h(y_k,y_l)
\end{array} 
\right)_{1\leq i,j,k,l \leq n},
\end{displaymath}

	\begin{displaymath}
B^{\vee}=\left(\begin{array}{cc} 
\pi h(y_k,y_l) & h(y_k,x_j)\\
h(x_i,y_l)& \pi^{-1} h(x_i,x_j)
\end{array}
\right)_{1\leq i,j,k,l \leq n}.
\end{displaymath}

We expect that $\CI_n^n(\tb{x},\tb{y})$ satisfies the following four properties

\begin{enumerate}
	\item (Compatibility with $\CN^0(1,n-1)$)\\
	Assume that 
	\begin{displaymath}
	B=\left(\begin{array}{cc} 
	B_1 & 0\\
	0& \pi^{-1}1_n
	\end{array}
	\right).
	\end{displaymath}
	 Then by Proposition \ref{proposition2.4} and Proposition \ref{proposition2.5}, $\CI_n^n(\tb{x},\tb{y})$ is the intersection number of special cycles in $\CN^0(1,n-1)$.  Therefore, our formula should be compatible with \cite[Conjecture 1.3]{KR2}. More precisely, we should have
\begin{equation*}
\CI_n^n(\tb{x},\tb{y})=\dfrac{\alpha'(1_n,B_1)}{\alpha(1_n,1_n)},
\end{equation*}
where $\alpha$ is the representation density (see Section \ref{section3.2} for its precise definition) and $\alpha'$ is the derivative of the representation density.\\

	\item (Linear invariance under $GL_n(O_{E_v}) \times GL_n(O_{E_v})$-action)\\
	 Note that in the case of $\CN^0(1,n-1)$, the intersection number of special cycles $\CZ(x_1), \dots, \CZ(x_n)$ is invariant under $GL_n(O_{E_v})$-action (see \cite{How}). Therefore, it is reasonable to expect that $\CI_n^n(\tb{x},\tb{y})$ has similar linear invariance. Let $g=g_1 \times g_2 \in GL_n(O_{E_v}) \times GL_n(O_{E_v})$ and let $\tilde{\tb{x}}=\tb{x}g_1$, $\tilde{\tb{y}}=\tb{y}g_2$. We conjecture that $\CI_n^n(\tilde{\tb{x}},\tilde{\tb{y}})=\CI_n^n(\tb{x},\tb{y})$.\\
	
	\item (Invariance under the isomorphism $\theta: \CN=\CN^n(1,2n-1) \rightarrow \widehat{\CN}=\CN^n(1,2n-1)$)\\
	Since the intersection number of $\CZ(x_1), \dots, \CZ(x_n)$ in $\CN^0(1,n-1)$ depends only on the matrix $(h(x_i,x_j))$, it is reasonable to expect that $\CI_n^n(\tb{x},\tb{y})$ depends only on $B$. Now, assume that $x_i=\lambda'_{\BX}(y'_i)$ for some $y'_i \in \Hom_{O_{E_v}}(\overline{\BY},\BX^{\vee})\otimes_{\BZ}\BQ$, $1 \leq i \leq n$ (this is the space of special homomorphisms associated to $\widehat{\CN}$). Also, we write $x_i'$ for $\lambda_{\BX}(y_i)$. Then the intersection number of $\CZ(x_i)$'s and $\CY(y_j)$'s in $\CN$ are the same as the intersection number of $\CY(y'_i)$'s and $\CZ(x_j')$'s in $\widehat{\CN}$. Note that
		\begin{displaymath}
\left(\begin{array}{cc} 
	h'(x'_i,x'_j)) & h'(x'_i,y'_l)\\
	h'(y'_k,x'_j)& h'(y'_k,y'_l)
	\end{array}
	\right)_{1\leq i,j,k,l \leq n}=B^{\vee},
	\end{displaymath}
	where $h'(x,y)=\lambda^{-1}_{\overline{\BY}} \circ y^{\vee} \circ \lambda_{\BX}' \circ x$ (this is the hermitian form on $\Hom_{O_{E_v}}(\overline{\BY},\BX^{\vee})\otimes_{\BZ}\BQ$).
	
	Since $\CN=\widehat{\CN}$ and we expect the $\CI_n^n(\tb{x},\tb{y})$ depends only on $B$, our conjectural formula should be invariant under the change $B \leftrightarrow B^{\vee}$.\\
	
	\item (Compatibility with $\CN^n(1,n-1)$)\\
		Assume that 
	\begin{displaymath}
	B=\left(\begin{array}{cc} 
	1_n & 0\\
	0& B_2
	\end{array}
	\right).
	\end{displaymath}
	Then by Proposition \ref{proposition2.4} and Proposition \ref{proposition2.5}, $\CI_n^n(\tb{x},\tb{y})$ is the intersection number of special cycles in $\CN^n(1,n-1)$. By \cite[Conjecture 1.3]{KR2} and the isomorphism between $\CN^0(1,n-1)$ and $\CN^n(1,n-1)$, we should have
	\begin{equation*}
	\CI_n^n(\tb{x},\tb{y})=\dfrac{\alpha'(1_n,\pi B_2)}{\alpha(1_n,1_n)}.
	\end{equation*}
	
	Indeed, this follows from (1) and (3).\\

\end{enumerate}

\begin{remark}
\begin{enumerate}
	\item Assume that 
	\begin{displaymath}
	B=
	\\\left(\begin{array}{cccc} 
	B_1 & 0 & 0& 0\\
	0&1_{h}&0&0\\
	0&0&B_2&0\\
	0&0&0& \pi^{-1}1_{n-h}
	\end{array}
	\right),
	\end{displaymath}
	where $B_1$ (resp. $B_2$) is a $(n-h)$ $\times$ $(n-h)$ (resp. $h$ $\times$ $h$) matrix.
	Then by Proposition \ref{proposition2.4} and Proposition \ref{proposition2.5}, $\CI_n^n(\tb{x},\tb{y})$ is the intersection number of $\CZ(x_1), \dots, \CZ(x_{n-h}), \CY(y_1),\dots, \CY(y_h)$ in $\CN^h(1,n-1)$. Here, we have abused notation slightly by writing $x_1, \dots, x_{n-h}, y_1, \dots, y_h$ for $x_1',\dots x_{n-h}', y_1', \dots, y_h'$ after $n$-times reductions in Proposition \ref{proposition2.5}. Therefore, our conjectural formula for $\CI_n^n(\tb{x},\tb{y})$ (which will be formulated in the next section) also gives a conjectural formula for the intersection numbers of the special cycles 
	\begin{equation*}
	\CZ(x_1), \dots, \CZ(x_{n-h}), \CY(y_1),\dots, \CY(y_h)
	\end{equation*}
	in $\CN^h(1,n-1)$.
	
	\item More generally, $\forall0 \leq t \leq n$, we consider the intersection number $\CI_{2n-t-h}^n(\tb{x},\tb{y})$, i.e., the intersection number of 
	\begin{equation*}
	\CZ(x_1), \dots, \CZ(x_{t+h}),\CY(y_1), \dots, \CY(y_{2n-t-h}).
	\end{equation*}. Assume that
		\begin{displaymath}
	B=\left(\begin{array}{cc} 
	h(x_i,x_j) & h(x_i,y_l)\\
	h(y_k,x_j)& h(y_k,y_l)
	\end{array} 
	\right)_{1\leq i,j \leq t+h, 1\leq k,l \leq 2n-t-h}
	\end{displaymath}
	\begin{displaymath}
	=\left(\begin{array}{cccc} 
	B_1 & 0 & 0& 0\\
	0&1_{h}&0&0\\
	0&0&B_2&0\\
	0&0&0& \pi^{-1}1_{n-h}
	\end{array}
	\right),
	\end{displaymath}
	where $B_1$ (resp. $B_2$) is a $t$ $\times$ $t$ (resp. $n-t$ $\times$ $n-t$) matrix. Then by Proposition \ref{proposition2.4} and Proposition \ref{proposition2.5},  $\CI_{2n-t-h}^n(\tb{x},\tb{y})$ is the intersection number of $\CZ(x_1), \dots, \CZ(x_{t}), \CY(y_1),\dots, \CY(y_{n-t})$ in $\CN^h(1,n-1)$. Therefore, the conjectural formula for $\CI_{2n-t-h}^n(\tb{x},\tb{y})$ (which will be formulated in the next section) also gives a conjectural formula for the intersection numbers of $\CZ(x_1), \dots, \CZ(x_{t}), \CY(y_1),\dots, \CY(y_{n-t})$ in $\CN^h(1,n-1)$.
\end{enumerate}	
\end{remark}

\bigskip

\section{Weighted representation densities and local conjectures}\label{section3}
In this section, we will study weighted representation densities and relate them to the intersection numbers of special cycles.

\subsection{Weighted representation densities}
In this subsection, we will study some weighted representation densities and their explicit formulas. Here we follow the notation in \cite[Section 10]{KR3} and \cite{Hir}.

We fix the standard additive character $\psi_v:F_v \rightarrow \BC^{\times}$ that is trivial on $O_{F_v}$. Let $V_v^{+}$ (resp. $V_v^{-}$) be a split (resp. nonsplit) $2n$-dimensional hermitian space over $E_v$ and let $\CS((V_v^{\pm})^{2n})$ be the space of Schwartz functions on $(V_v^{\pm})^{2n}$. Let $V_{r,r}$ be the split hermitian space of signature $(r,r)$, and let $L_{r,r}$ be a self-dual lattice in $V_{r,r}$. We write $\varphi_{r,r}$ for the characteristic function of $(L_{r,r})^n$ and $(V_v^{\pm})^{[r]}$ for $V_v^{\pm} \oplus V_{r,r}$. Given $\varphi$ in $\CS((V_v^{\pm})^{2n})$, we define a function $\varphi^{[r]}$ by $\varphi \otimes \varphi_{r,r} \in \CS(((V_v^{\pm})^{[r]})^{2n})$. 

We define the sets
\begin{equation*}
V_{2n}(E_v)=\lbrace Y \in M_{2n,2n}(E_v) \vert ^tY^*=Y \rbrace,
\end{equation*}
\begin{equation*}
X_{2n}(E_v)=\lbrace X \in GL_{2n}(E_v) \vert ^tX^*=X \rbrace.
\end{equation*}

The group $GL_{2n}(E_v)$ acts on $X_{2n}(E_v)$ by $g \cdot X=gX^tg^*$, where $g \in GL_{2n}(E_v)$, $X \in X_{2n}(E_v)$.

We define the Iwahori subgroup
\begin{equation*}
\Gamma_{2n}=\lbrace \gamma=(\gamma_{ij}) \in GL_{2n}(O_{E_v}) \vert \gamma_{ij} \in \pi O_{E_v} \text{ if } i>j \rbrace.
\end{equation*}

For $X,Y \in V_{2n}(E_v)$, we define $\langle X,Y \rangle=\Tr(XY)$. For $X \in M_{m,n}(E_v)$ and $A \in V_m(E_v)$, we define $A[X]=^tX^*AX \in V_n(E_v)$.

\begin{definition}\label{definition3.1}
	Let $0 \leq h,t \leq n$. If $t$ is even (resp. $t$ is odd) we define $L_t$ as a lattice of rank $2n$ in $V_v^+$ (resp. $V_v^-$) with hermitian form defined by 
	\begin{displaymath}
	A_t=\left(\begin{array}{cc} 
	1_{2n-t} & \\
	& \pi^{-1} 1_{t}
	\end{array} 
	\right).
	\end{displaymath}
	
	Let $1_{h,t}$ be the characteristic function of $(L_t^{\vee})^{2n-h} \times L_t^h$, where $L_t^{\vee}$ is the dual lattice with respect to the hermitian form. This is an element in $\CS((V_v^{+})^{2n})$ (resp. $\CS((V_v^{-})^{2n})$) if $t$ is even (resp. $t$ is odd). We can view $1_{h,t}$ as the characteristic function of the set of matrices of the form
	\begin{equation*}
	\left(\begin{array}{ll}
	A&B\\
	C&D
	\end{array}\right),
	\end{equation*}
	where $A \in M_{2n-t,2n-h}(O_{E_v})$, $B \in M_{2n-t,h}(O_{E_v})$, $C \in \pi M_{t,2n-h}(O_{E_v})$, and $D \in M_{t,h}(O_{E_v})$. Here we use the basis of $L_t$ such that the hermitian form of $L_t$ with respect to this basis is $A_t$.
	
	For an element $B \in X_{2n}(E_v)$, we define
	\begin{equation*}
	W_{h,t}(B,r)=\int_{V_{2n}(E_v)} \int_{M_{2n+2r,2n}(E_v)}\psi_v(\langle Y, A_t^{[r]}[X]-B\rangle)1_{h,t}^{[r]}(X)dXdY.
	\end{equation*}
	
	Here $dY$ (resp. $dX$) is the Haar measure on $V_{2n}(E_v)$ (resp. $M_{2n+2r,2n}(E_v)$) such that
	\begin{equation*}
	\int_{V_{2n}(O_{E_v})}dY=1 \quad (\text{resp. }\int_{M_{2n+2r,2n}(O_{E_v})}dX=1).
	\end{equation*}
	Also, $A_t^{[r]}$ is the $2n+2r$ $\times$ $2n+2r$ matrix
\begin{displaymath}
A_t^{[r]}=\left(\begin{array}{cc} 
A_t & \\
&  1_{2r}
\end{array} 
\right).
\end{displaymath}

\end{definition}

\begin{remark}
	The above definition and notation are modelled on the proof of \cite[Proposition 10.1]{KR3}. Therefore, we will relate the integrals $W_{h,t}$ to certain local Whittaker integrals later. Note that the above Haar measures are different from the measures in \cite[Proposition 10.1]{KR3}. This is because we need to use the results in \cite{Hir}, therefore we used the measures in loc. cit.
\end{remark}

\begin{lemma}(cf. \cite[Proposition 3.2]{Hir})\label{lemma3.3}
	For $B \in X_{2n}(E_v)$, we have
	\begin{equation*}
	W_{h,t}(B,r)=\sum_{Y \in \Gamma_{2n} \backslash X_{2n}(E_v)} \dfrac{\CG(Y,B)\CF_h(Y,A_t^{[r]})}{\alpha(Y;\Gamma_{2n})}.
	\end{equation*}
	Here
	\begin{equation*}
		\CF_h(Y,A_t^{[r]})=\int_{M_{2n+2r,2n}(E_v)} \psi_v(\langle Y, A_t^{[r]}[X]\rangle)1_{h,t}^{[r]}(X)dX,
	\end{equation*}
	and
	\begin{equation*}
	\CG(Y,B)=\int_{\Gamma_{2n}}\psi_v(\langle Y, -B[\gamma]\rangle)d\gamma,
	\end{equation*}
	where $d\gamma$ is the Haar measure on $M_{2n,2n}(O_{E_v})$ such that $\int_{M_{2n, 2n}(O_{E_v})} d\gamma=1$.
	
	Also, 
	\begin{equation*}
	\alpha(Y;\Gamma_{2n})=\lim_{d \rightarrow \infty} q^{-4dn^2}N_d(Y;\Gamma_{2n}),
	\end{equation*}
	where
	\begin{equation*}
	N_d(Y;\Gamma_{2n})=\vert \lbrace \gamma \in \Gamma_{2n} (\Mod \pi^d) \vert \gamma \cdot Y\equiv Y (\Mod \pi^d) \rbrace \vert.
	\end{equation*}
\end{lemma}
\begin{proof} The proof is almost identical to the proof of \cite[Proposition 3.2]{Hir}.
	
\end{proof}

\begin{definition}
	Let X be an element
	\begin{displaymath}
	X=\left(\begin{array}{cc} 
	A & B\\
	C&  D
	\end{array} 
	\right)
	\end{displaymath}
	in $X_{2n}(E_v)$, where 
	\begin{equation*}
	\begin{array}{ll}
	A \in M_{2n-h,2n-h}(E_v),\\
	 B \in M_{2n-h,h}(E_v),\\
	C \in M_{h,2n-h}(E_v),\\
	D \in M_{h,h}(E_v).
	\end{array}
	\end{equation*}
	
	We denote by $X^{\vee_h}$ the element
	\begin{displaymath}
	X^{\vee_h}=\left(\begin{array}{cc} 
	\pi D & C\\
	B&  \pi^{-1}A
	\end{array} 
	\right).
	\end{displaymath}
	
	Also, we define $X^{\wedge_h}$ as the element
		\begin{displaymath}
	X^{\wedge_h}=\left(\begin{array}{cc} 
	\pi^{-1} D & C\\
	B&  \pi A
	\end{array} 
	\right).
	\end{displaymath}
\end{definition}

Let $\CR_{2n}$ be the set
\begin{equation*}
\CR_{2n}=\lbrace Y_{\sigma,e} \vert (\sigma,e) \in \CS_{2n}\times \BZ^{2n}, \sigma^2=1, e_i=e_{\sigma(i)} \forall i \rbrace,
\end{equation*}
where $\CS_{2n}$ is the symmetric group of degree $2n$ and

	\begin{displaymath}
Y_{\sigma,e}=\sigma \left(\begin{array}{ccc} 
\pi^{e_1} &  & 0\\
 & \ddots & \\
0 &   &\pi^{e_{2n}}
\end{array} 
\right).
\end{displaymath}

Then, by \cite[Theorem 1]{Hir}, the set $\CR_{2n}$ forms a complete set of representatives of $\Gamma_{2n} \backslash X_{2n}$.

\begin{lemma}\label{lemma3.5}
		For $B \in X_{2n}(E_v)$, $W_{h,t}(B,r)$ is invariant under the action of $\Gamma_{2n}$ and $(GL_{2n-h} \times GL_{h})(O_{E_v})$ on $B$.
	\end{lemma}
\begin{proof}
	Let $\gamma \in \Gamma_{2n} \cup (GL_{2n-h} \times GL_{h})(O_{E_v})$.
	Then, we have
	\begin{equation*}
\begin{split}
W_{h,t}(B,r)&=\int_{V_{2n}(E_v)} \int_{M_{2n+2r,2n}(E_v)}\psi_v(\langle Y, A_t^{[r]}[X]-B\rangle)1_{h,t}^{[r]}(X)dXdY\\
&=\int_{V_{2n}(E_v)} \int_{M_{2n+2r,2n}(E_v)}\psi_v(\langle Y, A_t^{[r]}[ X\gamma]-B\rangle)1_{h,t}^{[r]}( X\gamma)d( X\gamma)dY.
\end{split}
\end{equation*}
Since $\gamma \in \Gamma_{2n} \cup (GL_{2n-h} \times GL_{h})(O_{E_v})$, we have $1_{h,t}^{[r]}(X)=1_{h,t}^{[r]}(X\gamma)$. Also, $\val(\det(\gamma))=0$ implies that $d(X\gamma)=dX$. Therefore, the above equation is equal to
\begin{equation*}
\int_{V_{2n}(E_v)} \int_{M_{2n+2r,2n}(E_v)}\psi_v(\langle Y, A_t^{[r]}[ X\gamma]-B\rangle)1_{h,t}^{[r]}(X)d(X)dY.
\end{equation*}

Also, we have
\begin{equation*}
\begin{split}
\langle Y, A_t^{[r]}[X\gamma]-B\rangle &=\Tr(Y(^t\gamma^*)(^tX^*) A_t^{[r]}X\gamma-YB)\\
&=\Tr( (^t\gamma^*)Y\gamma(^tX^*) A_t^{[r]}X-(^t\gamma^*)Y\gamma \gamma^{-1}B(^t\gamma^*)^{-1})\\
&=\langle (^t\gamma^*)Y\gamma,A_t^{[r]}[X]-\gamma^{-1}B(^t\gamma^*)^{-1} \rangle.
\end{split}
\end{equation*}

Now, we make a change of variables $Y \rightarrow (^t\gamma^*)^{-1}Y\gamma^{-1}$. Since $\val(\det(\gamma))=0$ the above equation is equal to
\begin{equation*}
\begin{split}
\int_{V_{2n}(E_v)} \int_{M_{2n+2r,2n}(E_v)}\psi_v(\langle Y, A_t^{[r]}[ X]-\gamma^{-1}B(^t\gamma^*)^{-1}\rangle)1_{h,t}^{[r]}(X)d(X)dY\\
=W_{h,t}(\gamma^{-1}B(^t\gamma^*)^{-1},r).
\end{split}
\end{equation*}

This finishes the proof.
\end{proof}

	By the invariance under $\Gamma_{2n}$, it suffices to consider $B \in \CR_{2n}$. Also, by the invariance under $(GL_{2n-h} \times GL_{h})(O_{E_v})$, it suffices to consider $Y_{\tau,\lambda} \in \CR_{2n}$ where $\tau$ satisfies
	\begin{equation}\label{equation3.0.1}
	\tau(i)=\left\lbrace \begin{array}{ll}
	i,  &  1 \leq i \leq 2n-h-s;\\
	i+h, &  2n-h-s+1 \leq i \leq 2n-h;\\
	i, & 2n-h+1 \leq i \leq 2n-s;\\
	i-h, & 2n-s+1 \leq i \leq 2n,
	\end{array}
	\right.
	\end{equation}
	for some $1 \leq s \leq \min(h,2n-h)$.

\begin{definition}\label{definition3.6}
	We define $\CR_{2n}^{h}$ by the set of $Y_{\tau,\lambda} \in \CR_{2n}$ satisfying the above conditions \eqref{equation3.0.1} on $\tau$. These are the matrices of the form
		\begin{displaymath}
	B= \left(\begin{array}{llll}
	B_1 &&&\\
	&&&B_4\\
	&&B_3&\\
	&B_2&&
	\end{array} 
	\right),
	\end{displaymath}
	where $B_1, B_2, B_3, B_4$ are diagonal matrices as follows.
		\begin{displaymath}
	B_1= \left(\begin{array}{lll}
	\pi^{\lambda_1}&&\\
	&\ddots&\\
	&&\pi^{\lambda_{2n-h-s}}\\
	\end{array} 
	\right);
	\end{displaymath}
	
		\begin{displaymath}
	B_2=B_4= \left(\begin{array}{lll}
	\pi^{\lambda_{2n-h-s+1}}&&\\
	&\ddots&\\
	&&\pi^{\lambda_{2n-h}}
	\end{array} 
	\middle)=\middle(\begin{array}{lll}
	\pi^{\lambda_{2n-s+1}}&&\\
	&\ddots&\\
	&&\pi^{\lambda_{2n}}
	\end{array} 
	\right);
	\end{displaymath}
	
	\begin{displaymath}
	B_3= \left(\begin{array}{lll}
	\pi^{\lambda_{2n-h+1}}&&\\
	&\ddots&\\
	&&\pi^{\lambda_{2n-s}}
	\end{array} 
	\right).
	\end{displaymath}
	
\end{definition}

\begin{definition}\label{definition3.7} Let $Y=Y_{\sigma,e} \in \CR_{2n}$.
	\begin{enumerate}
		\item We define the sets $A_1^h(Y)$, $A_2^h(Y)$, $B_1^h(Y)$, $B_2^h(Y)$, $C_1^h(Y)$, $C_2^h(Y)$ as follows.
		\begin{equation*}
		\begin{array}{l}
		A_1^h(Y)=\lbrace \text{ } 1 \leq j \leq 2n-h \text{ } \vert \text{ } \sigma(j)=j \text{ } \rbrace;\\
		A_2^h(Y)=\lbrace \text{ } 1 \leq j \leq 2n-h \text{ } \vert \text{ } \sigma(j)\neq j \text{ } \rbrace;\\
		B_1^h(Y)=\lbrace \text{ } 1 \leq j \leq 2n-h \text{ } \vert \text{ } 2n-h+1 \leq \sigma(j) \leq 2n \text{ } \rbrace;\\
		B_2^h(Y)=\lbrace \text{ } 2n-h+1 \leq j \leq 2n \text{ } \vert \text{ } 1\leq \sigma(j) \leq 2n-h \text{ } \rbrace;\\
		C_1^h(Y)=\lbrace \text{ } 2n-h+1 \leq j \leq 2n \text{ } \vert \text{ }  \sigma(j)=j \text{ } \rbrace;\\
		C_2^h(Y)=\lbrace \text{ } 2n-h+1 \leq j \leq 2n \text{ } \vert \text{ }  \sigma(j)\neq j \text{ } \rbrace.
		\end{array}
		\end{equation*}
		
		\item For $1 \leq k \leq 2n$, we define
		\begin{equation*}
		k^{\vee_h}=\left\lbrace \begin{array}{ll}
		k+h,  &  1 \leq k \leq 2n-h;\\
		k-(2n-h), &  2n-h+1 \leq k \leq 2n.
		\end{array}
		\right.
		\end{equation*}
		
		\item For $1 \leq j \leq 2n$, we define
		\begin{equation*}
		j^{\wedge_h}=\left\lbrace \begin{array}{ll}
		j+h,  &  1 \leq j \leq 2n-h;\\
		j-(2n-h), &  2n-h+1 \leq j \leq 2n.
		\end{array}
		\right.
		\end{equation*}
		Therefore, $j^{\wedge_h}=j^{\vee_h}$. We use this notation in order to avoid confusion in the proof of several lemmas below.
		
		\item For $\tau$ in \eqref{equation3.0.1}, we define
			\begin{equation*}
		\tau^{\vee_h}(i)=\left\lbrace \begin{array}{ll}
		i,  &  1 \leq i \leq h-s;\\
		i+(2n-h), &  h-s+1 \leq i \leq h;\\
		i, & h+1 \leq i \leq 2n-s;\\
		i-(2n-h) & 2n-s+1 \leq i \leq 2n.
		\end{array}
		\right.
		\end{equation*}
		
		\item We define $\sigma^{\wedge_h}$ as follows.
		\begin{equation*}
		\sigma^{\wedge_h}(j^{\wedge_h})=\left\lbrace \begin{array}{ll}
		\sigma(j)+h,  &  j \in A_1^h(Y) \cup A_2^h(Y);\\
		\sigma(j)-(2n-h), & j \in B_1^h(Y);\\
		\sigma(j)+h, & j \in B_2^h(Y);\\
		\sigma(j)-(2n-h) & j \in C_1^h(Y) \cup C_2^h(Y).
		\end{array}
		\right.
		\end{equation*}
	\end{enumerate}
	
\end{definition}

\begin{lemma}\label{lemma3.8}
	For $Y \in \CR_{2n}$ and $B \in \CR_{2n}^h$,
	we have
	\begin{equation*}
	\CG(Y,B)=\CG(Y^{\wedge_h},B^{\vee_h}).
	\end{equation*}
\end{lemma}

\begin{proof}
	Let $Y=Y_{\sigma,e}$, $B=Y_{\tau,\lambda}$. We use the notations in Definition \ref{definition3.6}, Definition \ref{definition3.7} and suppress $h$ from the notations for simplicity. Also, we write $O$ for $O_{E_v}$.
	
	Note that
	\begin{equation*}
	(B[\gamma])_{ji}=\sum_{k=1}^{2n} \gamma^*_{kj}\pi^{\lambda_k}\gamma_{\tau(k)i}.
	\end{equation*}
	
	Therefore, we have
	\begin{equation*}
	\begin{split}
	\langle Y, B[\gamma] \rangle&=\sum_{j=1}^{2n}Y_{\sigma(j)j}(B[\gamma])_{j\sigma(j)}\\
	&=\sum_{j=1}^{2n}Y_{\sigma(j)j}\sum_{k=1}^{2n}\gamma^*_{kj}\pi^{\lambda_k}\gamma_{\tau(k)i}\\
	&\sum_{j=1}^{2n}\sum_{k=1}^{2n}\pi^{e_j+\lambda_k}\gamma_{kj}^*\gamma_{\tau(k)\sigma(j)}.
	\end{split}
	\end{equation*}
	
	Similarly, we can compute $\langle Y^{\wedge}, B^{\vee}[\gamma] \rangle$. Before we do so, we need some notations.
	
	We define $\mu_k$, ($1\leq k \leq 2n)$ as follows.
	\begin{equation*}
	\mu_{k}=\left\lbrace \begin{array}{ll}
	\lambda_{2n-h+k}+1,  &  1 \leq k \leq h-s;\\
	\lambda_{2n-h+k}, &  h-s+1 \leq k \leq h;\\
	\lambda_{k-h}-1, & h+1 \leq k \leq 2n-s;\\
	\lambda_{k-h}, & 2n-s+1 \leq k \leq 2n.
	\end{array}
	\right.
	\end{equation*}

	Also, we define $f_j$, ($1 \leq j \leq 2n)$ as follows.
		\begin{equation*}
	f_{j}=\left\lbrace \begin{array}{ll}
	e_{j+2n-h}-1,  &  \text{ if }1 \leq j \leq h,\text{ } j+(2n-h) \in C_1^h(Y) \cup C_2^h(Y);\\
	e_{j}, & \text{ if }1 \leq j \leq h, \text{ }j+(2n-h) \in B_2^h(Y);\\
	e_{j}, & \text{ if }h+1 \leq j \leq 2n, \text{ }j-h \in B_1^h(Y);\\
	
	e_{j-h}+1, & \text{ if } h+1 \leq j \leq 2n, \text{ }j-h \in A_1^h(Y) \cup A_2^h(Y).
	\end{array}
	\right.
	\end{equation*}
	
	Then, we can show that $Y^{\wedge}=Y_{\sigma^{\wedge},f}$ and $B^{\vee}=Y_{\tau^{\vee},\mu}$. Therefore, we have
		\begin{equation*}
	\begin{split}
	\langle Y^{\wedge}, B^{\vee}[\gamma] \rangle&
	=\sum_{j=1}^{2n}\sum_{k=1}^{2n}\pi^{f_j+\mu_k}\gamma_{kj}^*\gamma_{\tau^{\vee}(k)\sigma^{\wedge}(j)}.\\
	&
	=\sum_{j^{\wedge}=1}^{2n}\sum_{k^{\vee}=1}^{2n}\pi^{f_{j^{\wedge}}+\mu_{k^{\vee}}}\gamma_{k^{\vee}j^{\wedge}}^*\gamma_{\tau^{\vee}(k^{\vee})\sigma^{\wedge}(j^{\wedge})}.
	\end{split}
	\end{equation*}
	
These two formulas give explicit formulas for $\CG(Y,B)$ and $\CG(Y^{\wedge_h},B^{\vee_h})$ as in the proof of \cite[Proposition 4.2]{Hir}. For example,

\begin{equation*}
\begin{array}{ll}
\CG(Y,B)&=\mathlarger{\prod}_{k=\tau(k)< j=\sigma(j)} \int_{O} \psi_v(\pi^{e_j+\lambda_k}\Nm(x))dx\\

&\times \mathlarger{\prod}_{ k=\tau(k)= j=\sigma(j)} \int_{O^{\times}} \psi_v(\pi^{e_j+\lambda_k}\Nm(x))dx\\
&\times \mathlarger{\prod}_{k=\tau(k)> j=\sigma(j)} \int_{\pi O} \psi_v(\pi^{e_j+\lambda_k}\Nm(x))dx\\

&\times \mathlarger{\prod}_{k<j, \tau(k)<\sigma(j), k\neq \tau(k) \text{ or } j\neq \sigma(j)} \int_{O \times O} \psi_v(\pi^{e_j+\lambda_k}\Tr(xy))dxdy\\

&\times \mathlarger{\prod}_{k=j, \tau(k)<\sigma(j), k\neq \tau(k) \text{ or } j\neq \sigma(j)} \int_{O^{\times} \times O} \psi_v(\pi^{e_j+\lambda_k}\Tr(xy))dxdy\\

&\times \mathlarger{\prod}_{k>j, \tau(k)<\sigma(j), k\neq \tau(k) \text{ or } j\neq \sigma(j)} \int_{\pi O \times O} \psi_v(\pi^{e_j+\lambda_k}\Tr(xy))dxdy\\

&\times \mathlarger{\prod}_{k<j, \tau(k)=\sigma(j), k\neq \tau(k) \text{ or } j\neq \sigma(j)} \int_{O \times O^{\times}} \psi_v(\pi^{e_j+\lambda_k}\Tr(xy))dxdy\\

&\times \mathlarger{\prod}_{k=j, \tau(k)=\sigma(j), k\neq \tau(k) \text{ or } j\neq \sigma(j)} \int_{O^{\times} \times O^{\times}} \psi_v(\pi^{e_j+\lambda_k}\Tr(xy))dxdy\\

&\times \mathlarger{\prod}_{k>j, \tau(k)=\sigma(j), k\neq \tau(k) \text{ or } j\neq \sigma(j)} \int_{\pi O \times O^{\times}} \psi_v(\pi^{e_j+\lambda_k}\Tr(xy))dxdy\\

&\times \mathlarger{\prod}_{k<j, \tau(k)>\sigma(j), k\neq \tau(k) \text{ or } j\neq \sigma(j)} \int_{O \times \pi O} \psi_v(\pi^{e_j+\lambda_k}\Tr(xy))dxdy\\

&\times \mathlarger{\prod}_{k=j, \tau(k)>\sigma(j), k\neq \tau(k) \text{ or } j\neq \sigma(j)} \int_{O^{\times} \times \pi O} \psi_v(\pi^{e_j+\lambda_k}\Tr(xy))dxdy\\

&\times \mathlarger{\prod}_{k>j, \tau(k)>\sigma(j), k\neq \tau(k) \text{ or } j\neq \sigma(j)} \int_{\pi O \times \pi O} \psi_v(\pi^{e_j+\lambda_k}\Tr(xy))dxdy\\

\end{array}
\end{equation*}

Similarly, we have a formula for $\CG(Y^{\wedge_h},B^{\vee_h})$.

Now, we will prove the lemma via term by term comparison. We write $\CG_{k,j}(Y,B)$ for $(k,j)$-term of $\CG(Y,B)$ and $\CG_{k^{\vee},j^{\wedge}}(Y^{\wedge_h},B^{\vee_h})$ for $(k^{\vee},j^{\wedge})$-term of $\CG(Y^{\wedge_h},B^{\vee_h})$.

Note that we have the following rules.

\begin{enumerate}
	\item If $k-j < 0$, then $\gamma_{kj}$ runs over $O$. If $k-j=0$, then $\gamma_{kj}$ runs over $O^{\times}$. If $k-j <0$, then $\gamma_{kj}$ runs over $\pi O$. The cases of ($\tau(k),\sigma(j)$), ($k^{\vee},j^{\wedge}$) and ($\tau^{\vee}(k^{\vee}),\sigma^{\wedge}(j^{\wedge})$) are similar.
	
	\item If $k=\tau(k)$ and $j=\sigma(j)$, then the integrand in $\CG_{k,j}(Y,B)$ is $\psi_v(\pi^{e_j+\lambda_k}\Nm(x))$. Otherwise, it is $\psi_v(\pi^{e_j+\lambda_k}\Tr(xy))$. 
	The case of $(k^{\vee},j^{\wedge},\tau^{\vee},\sigma^{\wedge})$, $(f_{j^{\wedge}},\mu_{k^{\vee}})$ is similar.
\end{enumerate}

	We write\\ 
	\textbf{Case 1} for $1 \leq k \leq 2n-h-s$,\\
	\textbf{Case 2} for $2n-h-s+1 \leq k \leq 2n-h$,\\
	\textbf{Case 3} for $2n-h+1 \leq k \leq 2n-s$, and\\
	\textbf{Case 4} for $2n-s+1 \leq k \leq 2n$.

Also, we write\\
 \textbf{Case A} for $j \in A_1^h(Y) \cup A_2^h(Y)$,\\
 \textbf{Case $B_1$} for $j \in B_1^h(Y)$,\\
 \textbf{Case $B_2$} for $j \in B_2^h(Y)$ and \\
 \textbf{Case C} for $j \in C_1^h(Y) \cup C_2^h(Y)$.

Then, we have the following 16 cases.

\begin{enumerate}
	\item \textbf{Case 1-A}
In this case, we have
\begin{equation*}
\begin{array}{ll}
\bullet\quad 1 \leq k \leq 2n-h-s,&h+1 \leq k^{\vee}=k+h \leq 2n-s;\\
\bullet\quad \tau(k)=k,&\tau^{\vee}(k^{\vee})=k^{\vee};\\
\bullet\quad 1 \leq j, \sigma(j) \leq 2n-h,&h+1 \leq j^{\wedge}, \sigma^{\wedge}(j^{\wedge}) \leq 2n;\\
\bullet\quad j^{\wedge}=j+h,& \sigma^{\wedge}(j^{\wedge})=\sigma(j)+h.
\end{array}
\end{equation*}

Therefore,
\begin{equation*}
\begin{array}{lll}
\bullet\quad k-j & = & k^{\vee}-j^{\wedge};\\
\bullet\quad \tau(k)-\sigma(j) & =& \tau^{\vee}(k^{\vee})-\sigma^{\wedge}(j^{\wedge});\\
\bullet\quad \lambda_k+e_j &= &\mu_{k^{\vee}}+f_{j^{\wedge}}=\lambda_{k}-1+e_{j}+1.
\end{array}
\end{equation*}

This implies that $\CG_{k,j}(Y,B)=\CG_{k^{\vee},j^{\wedge}}(Y^{\wedge_h},B^{\vee_h})$. 
\\
\item \textbf{Case 2-A}
In this case, we have
\begin{equation*}
\begin{array}{ll}
\bullet\quad  2n-h-s+1 \leq k \leq 2n-h, & h+1 \leq k^{\vee}=k+h \leq 2n-s;\\
\bullet\quad  2n-s+1 \leq \tau(k)=k+h \leq 2n,& \tau^{\vee}(k^{\vee})=k^{\vee}-(2n-h)\\
&=k+2h-2n;\\
\bullet\quad  1 \leq j, \sigma(j) \leq 2n-h,&h+1 \leq j^{\wedge}, \sigma^{\wedge}(j^{\wedge}) \leq 2n;\\
\bullet\quad  j^{\wedge}=j+h,& \sigma^{\wedge}(j^{\wedge})=\sigma(j)+h.

\end{array}
\end{equation*}

Therefore,
\begin{equation*}
\begin{array}{lllll}
\bullet\quad  &k-j & = & k^{\vee}-j^{\wedge};&\\
\bullet\quad  &\tau(k)-\sigma(j) &>0>& \tau^{\vee}(k^{\vee})-\sigma^{\wedge}(j^{\wedge})&\\
  &=k+h-\sigma(j)&&=k-\sigma(j)-(2n-h);&\\
\bullet\quad &\lambda_k&= &\mu_{k^{\vee}};&\\
\bullet\quad &e_j+1 &=& f_{j^{\wedge}}.&
\end{array}
\end{equation*}

Therefore, $\gamma_{\tau(k)\sigma(j)}$ runs over $\pi O$ and $\gamma_{\tau^{\vee}(k^{\vee})\sigma^{\wedge}(j^{\wedge})}$ runs over $O$. More precisely,
\begin{equation*}
\begin{array}{l}
\CG_{k,j}(Y,B)=\int_{\triangle \times \pi O} \psi_v(\pi^{e_j+\lambda_k}\Tr(xy))dxdy,\\
\CG_{k^{\vee},j^{\wedge}}(Y^{\wedge_h},B^{\vee_h})=\int_{\triangle \times  O} \psi_v(\pi^{e_j+\lambda_k+1}\Tr(xy))dxdy.
\end{array}
\end{equation*}
Here $\triangle$ is $O$, $O^{\times}$, and $\pi O$ if $k-j<0$, $k-j=0$, and $k-j>0$ respectively.

By using a change of variables, we get
\begin{equation*}
\CG_{k,j}(Y,B)=q^{-2}\CG_{k^{\vee},j^{\wedge}}(Y^{\wedge_h},B^{\vee_h}).
\end{equation*}
\\
\textbf{Case 3-A}
In this case, we have
\begin{equation*}
\begin{array}{ll}
\bullet\quad 2n-h+1 \leq k \leq 2n-s, & 1 \leq k^{\vee}=k-(2n-h) \leq h-s;\\
\bullet\quad \tau(k)=k,& \tau^{\vee}(k^{\vee})=k^{\vee};\\
\bullet\quad 1 \leq j, \sigma(j) \leq 2n-h,&h+1 \leq j^{\wedge}, \sigma^{\wedge}(j^{\wedge}) \leq 2n;\\
\bullet\quad j^{\wedge}=j+h,& \sigma^{\wedge}(j^{\wedge})=\sigma(j)+h.

\end{array}
\end{equation*}

Therefore, we have
\begin{equation*}
\begin{array}{lllll}
\bullet\quad &k-j & >0> & k^{\vee}-j^{\wedge}&\\
&&&=k-(2n-h)-(j+h)&\\
&&&=k-j-2n;&\\
\bullet\quad &\tau(k)-\sigma(j) &>0>& \tau^{\vee}(k^{\vee})-\sigma^{\wedge}(j^{\wedge})&\\
&=k+h-\sigma(j)&&=k-\sigma(j)-(2n-h);&\\
\bullet\quad &\lambda_k+1&= &\mu_{k^{\vee}};&\\
\bullet\quad &e_j+1 &=& f_{j^{\wedge}}.&
\end{array}
\end{equation*}

This implies that $\gamma_{kj}$ and $\gamma_{\tau(k)\sigma(j)}$ runs over $\pi O$. Also, $\gamma_{k^{\vee}j^{\wedge}}$ and $\gamma_{\tau^{\vee}(k^{\vee})\sigma^{\wedge}(j^{\wedge})}$ runs over $O$.

More precisely,
\begin{equation*}
\begin{array}{l}
\CG_{k,j}(Y,B)=\int_{\pi O \times \pi O} \psi_v(\pi^{e_j+\lambda_k}\Tr(xy))dxdy,\\
\CG_{k^{\vee},j^{\wedge}}(Y^{\wedge_h},B^{\vee_h})=\int_{O \times  O} \psi_v(\pi^{e_j+\lambda_k+2}\Tr(xy))dxdy.
\end{array}
\end{equation*}

This implies that
\begin{equation*}
\CG_{k,j}(Y,B)=q^{-4}\CG_{k^{\vee},j^{\wedge}}(Y^{\wedge_h},B^{\vee_h}).
\end{equation*}
\\
\textbf{Case 4-A}
In this case, we have
\begin{equation*}
\begin{array}{ll}
\bullet\quad 2n-s+1 \leq k \leq 2n, & h-s+1 \leq k^{\vee}=k-(2n-h) \leq h;\\
\bullet\quad \tau(k)=k-h,& \tau^{\vee}(k^{\vee})=k^{\vee}+(2n-h)=k;\\
\bullet\quad 1 \leq j, \sigma(j) \leq 2n-h,&h+1 \leq j^{\wedge}, \sigma^{\wedge}(j^{\wedge}) \leq 2n;\\
\bullet\quad j^{\wedge}=j+h,& \sigma^{\wedge}(j^{\wedge})=\sigma(j)+h.

\end{array}
\end{equation*}

Therefore, we have
\begin{equation*}
\begin{array}{lllll}
\bullet\quad &k-j & >0> & k^{\vee}-j^{\wedge}&\\
&&&=k-(2n-h)-(j+h)&\\
&&&=k-j-2n;&\\
\bullet\quad &\tau(k)-\sigma(j) &=& \tau^{\vee}(k^{\vee})-\sigma^{\wedge}(j^{\wedge})&\\
&=k-h-\sigma(j)&&=k-\sigma(j)-h;&\\
\bullet\quad &\lambda_k&= &\mu_{k^{\vee}};&\\
\bullet\quad &e_j+1 &=& f_{j^{\wedge}}.&
\end{array}
\end{equation*}

As above, this implies that
\begin{equation*}
\CG_{k,j}(Y,B)=q^{-2}\CG_{k^{\vee},j^{\wedge}}(Y^{\wedge_h},B^{\vee_h}).
\end{equation*}
\\
\textbf{Case 1-$B_1$}
In this case, we have
\begin{equation*}
\begin{array}{ll}
\bullet\quad 1 \leq k \leq 2n-h-s, & h+1 \leq k^{\vee}=k+h \leq 2n-s;\\
\bullet\quad \tau(k)=k,& \tau^{\vee}(k^{\vee})=k^{\vee};\\
\bullet\quad 1 \leq j \leq 2n-h,&h+1 \leq j^{\wedge} \leq 2n;\\
\bullet\quad 2n-h+1 \leq \sigma(j) \leq 2n, & 1 \leq \sigma^{\wedge}(j^{\wedge}) \leq h;\\
\bullet\quad j^{\wedge}=j+h,& \sigma^{\wedge}(j^{\wedge})=\sigma(j)-(2n-h).

\end{array}
\end{equation*}

Therefore, we have
\begin{equation*}
\begin{array}{lllll}
\bullet\quad &k-j & = & k^{\vee}-j^{\wedge}&\\
&&&=k-j;&\\
\bullet\quad &\tau(k)-\sigma(j) &<0<& \tau^{\vee}(k^{\vee})-\sigma^{\wedge}(j^{\wedge})&\\
&=k-\sigma(j)&&=(k+h)-(\sigma(j)-(2n-h))&\\
&&&=k-\sigma(j)+2n;&\\
\bullet\quad &\lambda_k-1&= &\mu_{k^{\vee}};&\\
\bullet\quad &e_j &=& f_{j^{\wedge}}.&
\end{array}
\end{equation*}

As above, this implies that
\begin{equation*}
\CG_{k,j}(Y,B)=q^{2}\CG_{k^{\vee},j^{\wedge}}(Y^{\wedge_h},B^{\vee_h}).
\end{equation*}
\\
\textbf{Case 2-$B_1$}
In this case, we have
\begin{equation*}
\begin{array}{ll}
\bullet\quad 2n-h-s+1 \leq k \leq 2n-h, & 2n-s+1 \leq k^{\vee}=k+h \leq 2n;\\
\bullet\quad 2n-s+1 \leq \tau(k)=k+h \leq 2n,& \tau^{\vee}(k^{\vee})=k^{\vee}-(2n-h)\\
&=k+2h-2n;\\
\bullet\quad 1 \leq j \leq 2n-h,&h+1 \leq j^{\wedge} \leq 2n;\\
\bullet\quad 2n-h+1 \leq \sigma(j) \leq 2n, & 1 \leq \sigma^{\wedge}(j^{\wedge}) \leq h;\\
\bullet\quad j^{\wedge}=j+h,& \sigma^{\wedge}(j^{\wedge})=\sigma(j)-(2n-h).

\end{array}
\end{equation*}

Therefore, we have
\begin{equation*}
\begin{array}{lllll}
\bullet\quad &k-j & = & k^{\vee}-j^{\wedge}&\\
&&&=k-j;&\\
\bullet\quad &\tau(k)-\sigma(j) &=& \tau^{\vee}(k^{\vee})-\sigma^{\wedge}(j^{\wedge})&\\
&=k+h-\sigma(j)&&=(k+2h-2n)-(\sigma(j)-(2n-h))&\\
&&&=k-\sigma(j)+h;&\\
\bullet\quad &\lambda_k&= &\mu_{k^{\vee}};&\\
\bullet\quad &e_j &=& f_{j^{\wedge}}.&
\end{array}
\end{equation*}

As above, this implies that
\begin{equation*}
\CG_{k,j}(Y,B)=\CG_{k^{\vee},j^{\wedge}}(Y^{\wedge_h},B^{\vee_h}).
\end{equation*}
\\
\textbf{Case 3-$B_1$}
In this case, we have
\begin{equation*}
\begin{array}{ll}
\bullet\quad 2n-h+1 \leq k \leq 2n-s, & 1\leq k^{\vee}=k-(2n-h) \leq h-s;\\
\bullet\quad \tau(k)=k,& \tau^{\vee}(k^{\vee})=k^{\vee};\\
\bullet\quad 1 \leq j \leq 2n-h,&h+1 \leq j^{\wedge} \leq 2n;\\
\bullet\quad 2n-h+1 \leq \sigma(j) \leq 2n, & 1 \leq \sigma^{\wedge}(j^{\wedge}) \leq h;\\
\bullet\quad j^{\wedge}=j+h,& \sigma^{\wedge}(j^{\wedge})=\sigma(j)-(2n-h).

\end{array}
\end{equation*}

Therefore, we have
\begin{equation*}
\begin{array}{lllll}
\bullet\quad &k-j & >0> & k^{\vee}-j^{\wedge}&\\
&&&=k-(2n-h)-(j+h)&\\
&&&=k-j-2n;&\\
\bullet\quad &\tau(k)-\sigma(j) &=& \tau^{\vee}(k^{\vee})-\sigma^{\wedge}(j^{\wedge})&\\
&=k-\sigma(j)&&=k^{\vee}-(\sigma(j)-(2n-h))&\\
&&&=k-(2n-h)-(\sigma(j)-(2n-h))&\\
&&&=k-\sigma(j);&\\
\bullet\quad &\lambda_k+1&= &\mu_{k^{\vee}};&\\
\bullet\quad &e_j &=& f_{j^{\wedge}}.&
\end{array}
\end{equation*}

As above, this implies that
\begin{equation*}
\CG_{k,j}(Y,B)=q^{-2}\CG_{k^{\vee},j^{\wedge}}(Y^{\wedge_h},B^{\vee_h}).
\end{equation*}
\\
\textbf{Case 4-$B_1$}
In this case, we have
\begin{equation*}
\begin{array}{ll}
\bullet\quad 2n-s+1 \leq k \leq 2n, & h-s+1 \leq k^{\vee}=k-(2n-h) \leq h;\\
\bullet\quad \tau(k)=k-h,& \tau^{\vee}(k^{\vee})=k^{\vee}+(2n-h)=k;\\
\bullet\quad 1 \leq j \leq 2n-h,&h+1 \leq j^{\wedge} \leq 2n;\\
\bullet\quad 2n-h+1 \leq \sigma(j) \leq 2n, & 1 \leq \sigma^{\wedge}(j^{\wedge}) \leq h;\\
\bullet\quad j^{\wedge}=j+h,& \sigma^{\wedge}(j^{\wedge})=\sigma(j)-(2n-h).

\end{array}
\end{equation*}

Therefore, we have
\begin{equation*}
\begin{array}{lllll}
\bullet\quad &k-j & >0> & k^{\vee}-j^{\wedge}&\\
&&&=k-(2n-h)-(j+h)&\\
&&&=k-j-2n;&\\
\bullet\quad &\tau(k)-\sigma(j) &<0<& \tau^{\vee}(k^{\vee})-\sigma^{\wedge}(j^{\wedge})&\\
&=k-h-\sigma(j)&&=k-\sigma(j)+(2n-h)&\\
\bullet\quad &\lambda_k&= &\mu_{k^{\vee}};&\\
\bullet\quad &e_j &=& f_{j^{\wedge}}.&
\end{array}
\end{equation*}

As above, this implies that
\begin{equation*}
\CG_{k,j}(Y,B)=\CG_{k^{\vee},j^{\wedge}}(Y^{\wedge_h},B^{\vee_h}).
\end{equation*}
\\
\textbf{Case 1-$B_2$}
\begin{equation*}
\begin{array}{ll}
\bullet\quad 1 \leq k \leq 2n-h-s,&h+1 \leq k^{\vee}=k+h \leq 2n-s;\\
\bullet\quad \tau(k)=k,&\tau^{\vee}(k^{\vee})=k^{\vee};\\
\bullet\quad 2n-h+1 \leq j \leq 2n,&1 \leq j^{\wedge} \leq h;\\
\bullet\quad 1 \leq \sigma(j) \leq 2n-h, & h+1 \leq \sigma^{\wedge}(j^{\wedge}) \leq 2n;\\
\bullet\quad j^{\wedge}=j-(2n-h),& \sigma^{\wedge}(j^{\wedge})=\sigma(j)+h.
\end{array}
\end{equation*}

Therefore, we have
\begin{equation*}
\begin{array}{lllll}
\bullet\quad &k-j & <0< & k^{\vee}-j^{\wedge}&\\
&&&=k+h-(j-(2n-h))&\\
&&&=k-j+2n;&\\
\bullet\quad &\tau(k)-\sigma(j) &=& \tau^{\vee}(k^{\vee})-\sigma^{\wedge}(j^{\wedge})&\\
&=k-\sigma(j)&&=(k+h)-(\sigma(j)+h)&\\
&&&=k-\sigma(j)&\\
\bullet\quad &\lambda_k-1&= &\mu_{k^{\vee}};&\\
\bullet\quad &e_j &=& f_{j^{\wedge}}.&
\end{array}
\end{equation*}

As above, this implies that
\begin{equation*}
\CG_{k,j}(Y,B)=q^2\CG_{k^{\vee},j^{\wedge}}(Y^{\wedge_h},B^{\vee_h}).
\end{equation*}
\\
\textbf{Case 2-$B_2$}
In this case, we have
\begin{equation*}
\begin{array}{ll}
\bullet\quad  2n-h-s+1 \leq k \leq 2n-h, & h+1 \leq k^{\vee}=k+h \leq 2n-s;\\
\bullet\quad  2n-s+1 \leq \tau(k)=k+h \leq 2n,& \tau^{\vee}(k^{\vee})=k^{\vee}-(2n-h)\\
&=k+2h-2n;\\
\bullet\quad 2n-h+1 \leq j \leq 2n,&1 \leq j^{\wedge} \leq h;\\
\bullet\quad 1 \leq \sigma(j) \leq 2n-h, & h+1 \leq \sigma^{\wedge}(j^{\wedge}) \leq 2n;\\
\bullet\quad j^{\wedge}=j-(2n-h),& \sigma^{\wedge}(j^{\wedge})=\sigma(j)+h.

\end{array}
\end{equation*}

Therefore, we have
\begin{equation*}
\begin{array}{lllll}
\bullet\quad &k-j & <0< & k^{\vee}-j^{\wedge}&\\
&&&=k+h-(j-(2n-h))&\\
&&&=k-j+2n;&\\
\bullet\quad &\tau(k)-\sigma(j) &>0>& \tau^{\vee}(k^{\vee})-\sigma^{\wedge}(j^{\wedge})&\\
&=k+h-\sigma(j)&&=k+2h-2n-(\sigma(j)+h)&\\
&&&=k-\sigma(j)-(2n-h)&\\
\bullet\quad &\lambda_k&= &\mu_{k^{\vee}};&\\
\bullet\quad &e_j &=& f_{j^{\wedge}}.&
\end{array}
\end{equation*}

As above, this implies that
\begin{equation*}
\CG_{k,j}(Y,B)=\CG_{k^{\vee},j^{\wedge}}(Y^{\wedge_h},B^{\vee_h}).
\end{equation*}
\\
\textbf{Case 3-$B_2$}
\begin{equation*}
\begin{array}{ll}
\bullet\quad 2n-h+1 \leq k \leq 2n-s, & 1\leq k^{\vee}=k-(2n-h) \leq h-s;\\
\bullet\quad \tau(k)=k,& \tau^{\vee}(k^{\vee})=k^{\vee};\\
\bullet\quad 2n-h+1 \leq j \leq 2n,&1 \leq j^{\wedge} \leq h;\\
\bullet\quad 1 \leq \sigma(j) \leq 2n-h, & h+1 \leq \sigma^{\wedge}(j^{\wedge}) \leq 2n;\\
\bullet\quad j^{\wedge}=j-(2n-h),& \sigma^{\wedge}(j^{\wedge})=\sigma(j)+h.
\end{array}
\end{equation*}

Therefore, we have
\begin{equation*}
\begin{array}{lllll}
\bullet\quad &k-j & = & k^{\vee}-j^{\wedge}&\\
&&&=k-(2n-h)-(j-(2n-h))&\\
&&&=k-j;&\\
\bullet\quad &\tau(k)-\sigma(j) &>0>& \tau^{\vee}(k^{\vee})-\sigma^{\wedge}(j^{\wedge})&\\
&=k-\sigma(j)&&=k-(2n-h)-(\sigma(j)+h)&\\
&&&=k-\sigma(j)-2n&\\
\bullet\quad &\lambda_k+1&= &\mu_{k^{\vee}};&\\
\bullet\quad &e_j &=& f_{j^{\wedge}}.&
\end{array}
\end{equation*}

As above, this implies that
\begin{equation*}
\CG_{k,j}(Y,B)=q^{-2}\CG_{k^{\vee},j^{\wedge}}(Y^{\wedge_h},B^{\vee_h}).
\end{equation*}
\\
\textbf{Case 4-$B_2$}
In this case, we have
\begin{equation*}
\begin{array}{ll}
\bullet\quad 2n-s+1 \leq k \leq 2n, & h-s+1 \leq k^{\vee}=k-(2n-h) \leq h;\\
\bullet\quad \tau(k)=k-h,& \tau^{\vee}(k^{\vee})=k^{\vee}+(2n-h)=k;\\
\bullet\quad 2n-h+1 \leq j \leq 2n,&1 \leq j^{\wedge} \leq h;\\
\bullet\quad 1 \leq \sigma(j) \leq 2n-h, & h+1 \leq \sigma^{\wedge}(j^{\wedge}) \leq 2n;\\
\bullet\quad j^{\wedge}=j-(2n-h),& \sigma^{\wedge}(j^{\wedge})=\sigma(j)+h.

\end{array}
\end{equation*}

Therefore, we have
\begin{equation*}
\begin{array}{lllll}
\bullet\quad &k-j & = & k^{\vee}-j^{\wedge}&\\
&&&=k-(2n-h)-(j-(2n-h))&\\
&&&=k-j;&\\
\bullet\quad &\tau(k)-\sigma(j) &=& \tau^{\vee}(k^{\vee})-\sigma^{\wedge}(j^{\wedge})&\\
&=k-h-\sigma(j)&&=k-\sigma(j)-h&\\
\bullet\quad &\lambda_k&= &\mu_{k^{\vee}};&\\
\bullet\quad &e_j &=& f_{j^{\wedge}}.&
\end{array}
\end{equation*}

As above, this implies that
\begin{equation*}
\CG_{k,j}(Y,B)=\CG_{k^{\vee},j^{\wedge}}(Y^{\wedge_h},B^{\vee_h}).
\end{equation*}
\\
\textbf{Case 1-C}
\begin{equation*}
\begin{array}{ll}
\bullet\quad 1 \leq k \leq 2n-h-s,&h+1 \leq k^{\vee}=k+h \leq 2n-s;\\
\bullet\quad \tau(k)=k,&\tau^{\vee}(k^{\vee})=k^{\vee};\\
\bullet\quad 2n-h+1 \leq j, \sigma(j) \leq 2n,&1 \leq j^{\wedge}, \sigma^{\wedge}(j^{\wedge}) \leq h;\\
\bullet\quad j^{\wedge}=j-(2n-h),& \sigma^{\wedge}(j^{\wedge})=\sigma(j)-(2n-h).

\end{array}
\end{equation*}
Therefore, we have
\begin{equation*}
\begin{array}{lllll}
\bullet\quad &k-j & <0< & k^{\vee}-j^{\wedge}&\\
&&&=k+h-(j-(2n-h))&\\
&&&=k-j+2n;&\\
\bullet\quad &\tau(k)-\sigma(j) &<0<& \tau^{\vee}(k^{\vee})-\sigma^{\wedge}(j^{\wedge})&\\
&=k-\sigma(j)&&=k+h-\sigma(j)+2n-h&\\
&&&=k-\sigma(j)+2n&\\
\bullet\quad &\lambda_k-1&= &\mu_{k^{\vee}};&\\
\bullet\quad &e_j-1 &=& f_{j^{\wedge}}.&
\end{array}
\end{equation*}

As above, this implies that
\begin{equation*}
\CG_{k,j}(Y,B)=q^4\CG_{k^{\vee},j^{\wedge}}(Y^{\wedge_h},B^{\vee_h}).
\end{equation*}
\\
\textbf{Case 2-C}
In this case, we have
\begin{equation*}
\begin{array}{ll}
\bullet\quad  2n-h-s+1 \leq k \leq 2n-h, & h+1 \leq k^{\vee}=k+h \leq 2n-s;\\
\bullet\quad  2n-s+1 \leq \tau(k)=k+h \leq 2n,& \tau^{\vee}(k^{\vee})=k^{\vee}-(2n-h)\\
&=k+2h-2n;\\
\bullet\quad 2n-h+1 \leq j, \sigma(j) \leq 2n,&1 \leq j^{\wedge}, \sigma^{\wedge}(j^{\wedge}) \leq h;\\
\bullet\quad j^{\wedge}=j-(2n-h),& \sigma^{\wedge}(j^{\wedge})=\sigma(j)-(2n-h).
\end{array}
\end{equation*}

Therefore, we have
\begin{equation*}
\begin{array}{lllll}
\bullet\quad &k-j & <0< & k^{\vee}-j^{\wedge}&\\
&&&=k+h-(j-(2n-h))&\\
&&&=k-j+2n;&\\
\bullet\quad &\tau(k)-\sigma(j) &=& \tau^{\vee}(k^{\vee})-\sigma^{\wedge}(j^{\wedge})&\\
&=k+h-\sigma(j)&&=k+2h-2n-\sigma(j)+2n-h&\\
&&&=k-\sigma(j)+h&\\
\bullet\quad &\lambda_k&= &\mu_{k^{\vee}};&\\
\bullet\quad &e_j-1 &=& f_{j^{\wedge}}.&
\end{array}
\end{equation*}

As above, this implies that
\begin{equation*}
\CG_{k,j}(Y,B)=q^2\CG_{k^{\vee},j^{\wedge}}(Y^{\wedge_h},B^{\vee_h}).
\end{equation*}
\\
\textbf{Case 3-C}
\begin{equation*}
\begin{array}{ll}
\bullet\quad 2n-h+1 \leq k \leq 2n-s, & 1\leq k^{\vee}=k-(2n-h) \leq h-s;\\
\bullet\quad \tau(k)=k,& \tau^{\vee}(k^{\vee})=k^{\vee};\\
\bullet\quad 2n-h+1 \leq j, \sigma(j) \leq 2n,&1 \leq j^{\wedge}, \sigma^{\wedge}(j^{\wedge}) \leq h;\\
\bullet\quad j^{\wedge}=j-(2n-h),& \sigma^{\wedge}(j^{\wedge})=\sigma(j)-(2n-h).
\end{array}
\end{equation*}

Therefore, we have
\begin{equation*}
\begin{array}{lllll}
\bullet\quad &k-j & = & k^{\vee}-j^{\wedge}&\\
&&&=k-(2n-h)-(j-(2n-h))&\\
&&&=k-j;&\\
\bullet\quad &\tau(k)-\sigma(j) &=& \tau^{\vee}(k^{\vee})-\sigma^{\wedge}(j^{\wedge})&\\
&=k-\sigma(j)&&=k-(2n-h)-\sigma(j)+2n-h&\\
&&&=k-\sigma(j)&\\
\bullet\quad &\lambda_k+1&= &\mu_{k^{\vee}};&\\
\bullet\quad &e_j-1 &=& f_{j^{\wedge}}.&
\end{array}
\end{equation*}

As above, this implies that
\begin{equation*}
\CG_{k,j}(Y,B)=\CG_{k^{\vee},j^{\wedge}}(Y^{\wedge_h},B^{\vee_h}).
\end{equation*}
\\
\textbf{Case 4-C}
In this case, we have
\begin{equation*}
\begin{array}{ll}
\bullet\quad 2n-s+1 \leq k \leq 2n, & h-s+1 \leq k^{\vee}=k-(2n-h) \leq h;\\
\bullet\quad \tau(k)=k-h,& \tau^{\vee}(k^{\vee})=k^{\vee}+(2n-h)=k;\\
\bullet\quad 2n-h+1 \leq j, \sigma(j) \leq 2n,&1 \leq j^{\wedge}, \sigma^{\wedge}(j^{\wedge}) \leq h;\\
\bullet\quad j^{\wedge}=j-(2n-h),& \sigma^{\wedge}(j^{\wedge})=\sigma(j)-(2n-h).

\end{array}
\end{equation*}

Therefore, we have
\begin{equation*}
\begin{array}{lllll}
\bullet\quad &k-j & = & k^{\vee}-j^{\wedge}&\\
&&&=k-(2n-h)-(j-(2n-h))&\\
&&&=k-j;&\\
\bullet\quad &\tau(k)-\sigma(j) &<0<& \tau^{\vee}(k^{\vee})-\sigma^{\wedge}(j^{\wedge})&\\
&=k-h-\sigma(j)&&=k-\sigma(j)+2n-h&\\
\bullet\quad &\lambda_k&= &\mu_{k^{\vee}};&\\
\bullet\quad &e_j-1 &=& f_{j^{\wedge}}.&
\end{array}
\end{equation*}

As above, this implies that
\begin{equation*}
\CG_{k,j}(Y,B)=q^2\CG_{k^{\vee},j^{\wedge}}(Y^{\wedge_h},B^{\vee_h}).
\end{equation*}\\
\end{enumerate}

Now, we write $\xi$ for $\vert B_1^h(Y) \vert =\vert B_2^h(Y) \vert$. Then we have
\begin{equation*}
\begin{array}{l}
\vert A_1^h(Y) \cup A_2^h(Y)\vert=2n-h-\xi,\\
\vert C_1^h(Y) \cup C_2^h(Y)\vert=h-\xi.
\end{array}
\end{equation*}

Therefore,
\begin{equation*}
\begin{array}{ll}
\CG(Y,B)&=q^{\lbrace(2n-h-s)\times 0 + s \times (-2)+(h-s) \times (-4) +s \times(-2)\rbrace\times(2n-h-\xi)}\\
&\times q^{\lbrace(2n-h-s)\times 2 + s \times 0+(h-s) \times (-2) +s \times 0\rbrace\times\xi}\\
&\times q^{\lbrace(2n-h-s)\times 2 + s \times 0+(h-s) \times (-2) +s \times0\rbrace\times\xi}\\
&\times q^{\lbrace(2n-h-s)\times 4 + s \times 2+(h-s) \times 0 +s \times 2\rbrace\times(h-\xi)}\\
&\times \CG(Y^{\wedge_h},B^{\vee_h})\\
&= q^{-4h(2n-h-\xi)+(4n-4h)2\xi+(8n-4h)(h-\xi)}\CG(Y^{\wedge_h},B^{\vee_h})\\
&=\CG(Y^{\wedge_h},B^{\vee_h}).
\end{array}
\end{equation*}

This finishes the proof of the lemma.
\end{proof}

	\begin{lemma}\label{lemma3.9}
		For $Y \in \CR_{2n}$,
		\begin{equation*}
		q^{(2n-h)^2-h^2}\alpha(Y;\Gamma_{2n})=\alpha(Y^{\wedge_h};\Gamma_{2n}).
		\end{equation*}
	\end{lemma}
\begin{proof}
	Before we prove this lemma, let us introduce the following notation.
	For
	\begin{equation*}
	\gamma=\left( \begin{array}{ll}
	A &B\\
	C&D
	\end{array}
	\right) \in \Gamma_{2n},
	\end{equation*}
	we define
	\begin{equation*}
	\gamma^{\curlywedge}=\left( \begin{array}{ll}
	D &\pi^{-1}C\\
	\pi B&A
	\end{array}
	\right) \in \Gamma_{2n},
	\end{equation*}
	
	Here $A$, $B$, $C$ and $D$ are $(2n-h) \times (2n-h)$, $(2n-h) \times h$, $h \times (2n-h)$, and $h \times h$ matrices respectively.
	
	Note that $^\curlywedge:\Gamma_{2n} \rightarrow \Gamma_{2n}$ is bijective.
	
	Let
	\begin{equation*}
	Y=\left( \begin{array}{ll}
	E &F\\
	G & H
	\end{array}
	\right),
	\end{equation*}
	where $E$, $F$, $G$ and $H$ are $(2n-h) \times (2n-h)$, $(2n-h) \times h$, $h \times (2n-h)$, and $h \times h$ matrices respectively.
	
	Then, we have
		\begin{equation*}
	Y^{\wedge}=\left( \begin{array}{ll}
	\pi^{-1}H &G\\
	F & \pi E
	\end{array}
	\right).
	\end{equation*}
	
	Note that
	\begin{equation*}
	\begin{array}{ll}
	\gamma \cdot Y&=\gamma Y ^t\gamma^*
	\end{array}
	\end{equation*}
	\begin{equation*}
	=\left( \begin{array}{ll}
	A & B\\
	C & D
	\end{array}\middle)
	\middle( \begin{array}{ll}
	E & F \\
	G & H
	\end{array}\middle ) \middle(
	\begin{array}{ll}
	^tA^* & ^tC^*\\
	^tB^* & ^tD^*
	\end{array}\right)
	\end{equation*}
	
	\begin{equation*}
	=\left( \begin{array}{cc}
	AE(^tA^*)+BG(^tA^*)\quad\quad\quad\quad  & AE(^tC^*)+BG(^tC^*) \\
	+AF(^tB^*)+BH(^tB^*) &+AF(^tD^*)+BH(^tD^*)\\
	&\\
	CE(^tA^*)+DG(^tA^*)\quad\quad\quad\quad  & CE(^tC^*)+DG(^tC^*) \\
+CF(^tB^*)+DH(^tB^*)&+CF(^tD^*)+DH(^tD^*)
	\end{array}\right).
	\end{equation*}

Similarly,
	\begin{equation*}
\begin{array}{ll}
\gamma^{\curlywedge} \cdot Y^{\wedge}&=\gamma^{\curlywedge} Y^{\wedge} (^t(\gamma^{\curlywedge})^*)
\end{array}
\end{equation*}
\begin{equation*}
=\left( \begin{array}{ll}
D & \pi^{-1}C\\
\pi B & A
\end{array}\middle)
\middle( \begin{array}{ll}
\pi^{-1}H & G \\
F & \pi E
\end{array}\middle ) \middle(
\begin{array}{ll}
^tD^* & \pi (^tB^*)\\
\pi^{-1}(^tC^*) & ^tA^*
\end{array}\right)
\end{equation*}
	\begin{equation*}
=\left( \begin{array}{cc}

\pi^{-1}\left(\begin{array}{l}
DH(^tD^*)+CF(^tD^*)\\
+DG(^tC^*)+CE(^tC^*)
\end{array}\right)  & \left(\begin{array}{l}
DH(^tB^*)+CF(^tB^*)\\
+DG(^tA^*)+CE(^tA^*)
\end{array}\right)  \\

\left(\begin{array}{l}
BH(^tD^*)+AF(^tD^*)\\
+BG(^tC^*)+AE(^tC^*)
\end{array}\right)  & \pi\left(\begin{array}{l}
BH(^tB^*)+AF(^tB^*)\\
+BG(^tA^*)+AE(^tA^*)
\end{array}\right)
\end{array}\right).
\end{equation*}

We write $M$ for $M_{2n,2n}(O_{E_v})$, $V$ for $V_{2n}(O_{E_v})$

The above two equations imply that $\gamma \cdot Y \equiv Y (\Mod \pi^d V)$ if and only if $\gamma^{\curlywedge} \cdot Y^{\wedge} \equiv Y^{\wedge} (\Mod \pi^d \widetilde{V})$ where $\widetilde{V}$ is the subset of $V_{2n}(E_v)$ consisting of matrices of the form
\begin{equation*}
\left(\begin{array}{ll}
\pi^{-1}A & B\\
 C & \pi D
\end{array}\right).
\end{equation*}
Here $A$, $B$, $C$, and $D$ are matrices in $M_{h,h}(O_{E_v})$, $M_{h,2n-h}(O_{E_v})$, $M_{2n-h,h}(O_{E_v})$, and $M_{2n-h,2n-h}(O_{E_v})$ respectively.

Also, we write $\widehat{M}$ for the set of matrices of the form
\begin{equation*}
\left(\begin{array}{ll}
A & \pi^{-1} B\\
\pi C & D
\end{array}\right).
\end{equation*}
Here $A$, $B$, $C$ and $D$ are the matrices as above.

Then, we have
\begin{equation*}
\begin{array}{l}
\alpha(Y;\Gamma_{2n})\\
=\lim_{d \rightarrow \infty} q^{-4dn^2}N_d(Y;\Gamma_{2n})\\
=\lim_{d \rightarrow \infty} q^{-4dn^2} \vert \lbrace \gamma \in \Gamma_{2n} (\Mod \pi^dM) \vert \gamma \cdot Y\equiv Y (\Mod \pi^dV) \rbrace \vert\\
=\lim_{d \rightarrow \infty} q^{-4dn^2-8n^2} \vert \lbrace \gamma \in \Gamma_{2n} (\Mod \pi^{(d+1)}M) \vert \gamma \cdot Y\equiv Y (\Mod \pi^dV) \rbrace \vert\\
=\lim_{d \rightarrow \infty} q^{-4dn^2-8n^2} \vert \lbrace \gamma^{\curlywedge} \in \Gamma_{2n} (\Mod \pi^{(d+1)}\widehat{M}) \vert \gamma^{\curlywedge} \cdot Y^{\wedge}\equiv Y^{\wedge} (\Mod \pi^d\widetilde{V}) \rbrace \vert.\\
\end{array}
\end{equation*}

On the other hand, by thinking $\Gamma_{2n}$ as a union of cosets of $\pi^{(d+1)}\widehat{M}$ and $V_{2n}(E_v)$ as a union of cosets of $ \pi^d\widetilde{V}$ for sufficiently large $d$, we have
\begin{equation*}
\begin{array}{l}
\alpha(Y^{\wedge};\Gamma_{2n})\\
=\lim_{d \rightarrow \infty} q^{-\vert M:\pi^{(d+1)}\widehat{M} \vert +\vert V:\pi^d \widetilde{V} \vert} \\
\times \vert \lbrace \gamma^{\curlywedge} \in \Gamma_{2n} (\Mod \pi^{(d+1)}\widehat{M}) \vert \gamma^{\curlywedge} \cdot Y^{\wedge}\equiv Y^{\wedge} (\Mod \pi^d\widetilde{V}) \rbrace \vert\\
=\lim_{d \rightarrow \infty} q^{-2 (d+1) (2n)^2+(d-1)h^2+d(2n-h)h+dh(2n-h)+(d+1)(2n-h)^2} \\
\times \vert \lbrace \gamma^{\curlywedge} \in \Gamma_{2n} (\Mod \pi^{(d+1)}\widehat{M}) \vert \gamma^{\curlywedge} \cdot Y^{\wedge}\equiv Y^{\wedge} (\Mod \pi^d\widetilde{V}) \rbrace \vert\\
=\lim_{d \rightarrow \infty} q^{-4dn^2-8n^2+(2n-h)^2-h^2} \\
\times \vert \lbrace \gamma^{\curlywedge} \in \Gamma_{2n} (\Mod \pi^{(d+1)}\widehat{M}) \vert \gamma^{\curlywedge} \cdot Y^{\wedge}\equiv Y^{\wedge} (\Mod \pi^d\widetilde{V}) \rbrace \vert\\
=q^{(2n-h)^2-h^2}\alpha(Y;\Gamma_{2n}).
\end{array}
\end{equation*}

This finishes the proof of the lemma.

\end{proof}

\begin{definition}
	We can regard $\CF_h(Y,A_t^{[r]})$ and $W_{h,t}(B,r)$ as functions of $x=(-q)^{(-2r)}$. We define $\CF_h'(Y,A_t^{[0]})$ as
	\begin{equation*}
	\CF_h'(Y,A_t^{[0]})=-\dfrac{d}{dx}\CF_h(Y,A_t^{[r]}) \vert_{x=1}.
	\end{equation*}
	Also we define $W_{h,t}'(B,0)$ as
	\begin{equation*}
	W_{h,t}'(B,0)=-\dfrac{d}{dx}W_{h,t}(B,r) \vert_{x=1}.
	\end{equation*}

\end{definition}

\begin{remark}
	One can prove that $W_{h,t}(B,r)$ is a polynomial in $x$ as in \cite[Corollary 4.4]{Hir} with a slight modification.
\end{remark}

\begin{lemma}\label{lemma 3.12}(\cite[Lemma 4.1]{Hir})
	For $e \in \BZ$, we have the following equations.
	\begin{equation*}
	\begin{array}{l}
	\int_{O} \psi_v(\pi^e\Nm(x))dx=(-q)^{\min( 0, e )},\\
	\int_{\pi O} \psi_v(\pi^e\Nm(x))dx=(-q)^{\min( 0, e+2 )-2},\\
	\int_{O \times O}\psi_v(\pi^{e}\Tr(xy))dxdy=(-q)^{2\min( 0,e )},\\
	\int_{\pi O \times O}\psi_v(\pi^{e}\Tr(xy))dxdy=(-q)^{2\min( 0,e+1 )-2},\\
	\int_{\pi O \times \pi O}\psi_v(\pi^{e}\Tr(xy))dxdy=(-q)^{2\min( 0,e+2 )-4}.
	\end{array}
	\end{equation*}
\end{lemma}
\begin{proof}
This follows from \cite[Lemma 4.1]{Hir} and some changes of variables.
\end{proof}

\begin{definition}
	For $e \in \BZ$, we write $I^*(e)$ for the integral
	\begin{equation*}
	\int_{O^{\times}} \psi_v(\pi^e\Nm(x))dx.
	\end{equation*}
	
	Also, we write $J_1(e)$ for the integral
	\begin{equation*}
	\int_{O}\psi_v(\pi^e\Tr(x))dx.
	\end{equation*}
\end{definition}

We will need the following lemma later.

\begin{lemma}(\cite[Lemma 4.1]{Hir})\label{lemma3.14}
	For $e \in \BZ$, we have
	\begin{equation*}
	I^*(e)=\left\lbrace \begin{array}{ll}
	1-q^{-2}& \text{if } e \geq 0\\
	-q^{-1}-q^{-2} & \text{if } e=-1\\
	0 & \text{if } e \leq -2
	\end{array}\right..
	\end{equation*}
	
	Also, we have
	\begin{equation*}
	J_1(e)=\left\lbrace \begin{array}{ll}
	1& \text{if } e \geq 0\\
	0& \text{if } e < 0
	\end{array}\right..
	\end{equation*}
\end{lemma}

\begin{lemma}\label{lemma 3.13}
	We write $\FA(Y)$ and $\FC(Y)$ for
	\begin{equation*}
	\begin{array}{l}
	\FA(Y)=\vert \lbrace j \in A_1^h(Y) \cup A_2^h(Y) \vert \text{ } e_j \leq -1 \rbrace \vert,\\
	\FC(Y)=\vert \lbrace j \in C_1^h(Y) \cup C_2^h(Y) \vert \text{ } e_j \leq 0 \rbrace \vert.
	\end{array}
	\end{equation*}
	Then, we have
	\begin{equation*}
	\dfrac{\CF_h'(Y,A_n^{[0]})}{\alpha(Y;\Gamma_{2n})}-\dfrac{\CF_{2n-h}'(Y^{\wedge},A_n^{[0]})}{\alpha(Y^{\wedge_h};\Gamma_{2n})}=\dfrac{(\FC(Y)-\FA(Y))f_h(Y)(-q)^{-2n(2n-h)}}{\alpha(Y;\Gamma_{2n})},
	\end{equation*}
	\begin{equation*}
	\CF_h(Y,A_t^{[0]})=(-q)^{((n-t)(\FC(Y)-\FA(Y)-2t(2n-h))}f_h(Y),
	\end{equation*}
	and
	\begin{equation*}
	\CF_{2n-h}(Y^{\wedge},A_t^{[0]})=(-q)^{((n-t)(\FA(Y)-\FC(Y))-2th)}f_h(Y),
	\end{equation*}
	where
	\begin{align*}
f_h(Y)&=\mathlarger{\prod}_{j \in A_1^h(Y)} (-q)^{n\min( 0, e_j )+n \min( 0, e_j+1 )}\\
&\times \mathlarger{\prod}_{\substack{j \in A_2^h(Y)\\ j < \sigma(j)}} (-q)^{n \mathsmaller{\times} 2\min( 0, e_j )+n \mathsmaller{\times} 2\min( 0, e_j+1 )}\\
&\times \mathlarger{\prod}_{j \in B_1^h(Y)} (-q)^{n \mathsmaller{\times} 2\min( 0, e_j )+n \mathsmaller{\times} 2\min( 0, e_j )}\\
&\times \mathlarger{\prod}_{j \in C_1^h(Y)} (-q)^{n \mathsmaller{\times} \min( 0, e_j )+n \mathsmaller{\times} \min( 0, e_j-1 )}\\
&\times \mathlarger{\prod}_{\substack{j \in C_2^h(Y)\\ j < \sigma(j)}} (-q)^{n \mathsmaller{\times} 2\min( 0, e_j )+n \mathsmaller{\times} 2\min( 0, e_j-1 )}.
\end{align*}
	
\end{lemma}

\begin{proof}
	Let $A_t^{[r]}=\diag(\pi^{a_{t,1}},\dots,\pi^{a_{t,2n+2r}})$, i.e.,
	\begin{equation*}
	a_{t,k}=\left\lbrace \begin{array}{ll}
		-1, & \text{ if } 2n-t+1 \leq k \leq 2n\\
		0, & \text{ otherwise}.
		\end{array}\right.
	\end{equation*}
	
	Then, we can write
	\begin{equation*}
	\begin{array}{ll}
	\langle Y,A_t^{[r]}[X] \rangle
	&=\mathlarger{\sum}^{2n}_{\substack{j=1\\ \sigma(j)=j}}\mathlarger{\sum}^{2n+2r}_{k=1}\pi^{e_j+a_{t,k}}\Nm(X_{kj})\\
	&+\mathlarger{\sum}^{2n}_{\substack{j=1\\ \sigma(j)>j}}\mathlarger{\sum}^{2n+2r}_{k=1}\pi^{e_j+a_{t,k}}\Tr(X_{kj}X^*_{k\sigma(j)}).\\
	\end{array}
	\end{equation*}
	
	Also, note that if we denote by $\lbrace v_1, \dots, v_{2n} \rbrace$ the basis of $L_t$ such that $A_t$ is the hermitian matrix with respect to this basis, then $L_t^{\vee}$ is the span of $\lbrace v_1, \dots, v_{2n-t}, \pi v_{2n-t+1}, \dots \pi v_{2n} \rbrace$.

Since $1_{h,t}^{[r]}$ is the characteristic function of $(L_t^{\vee})^{2n-h} \times L_t^h \times L_{r,r}$, we can write $\CF_h(Y,A_t^{[r]})$ as

\begin{align*}
\CF_h(Y,A_t^{[r]})&=\mathlarger{\prod}_{\substack{k \leq 2n-t \text{ or } 2n+1 \leq k \leq 2n+2r \\ j \in A_1^h(Y)}} \int_{O} \psi_v(\pi^{e_j}\Nm(x))dx\\
 &\times \mathlarger{\prod}_{\substack{2n-t+1 \leq k \leq 2n\\ j \in A_1^h(Y)}} \int_{\pi O} \psi_v(\pi^{e_j-1}\Nm(x))dx\\
 &\times \mathlarger{\prod}_{\substack{k \leq 2n-t \text{ or } 2n+1 \leq k \leq 2n+2r \\ j \in A_2^h(Y), \text{ } j < \sigma(j)}} \int_{O \times O} \psi_v(\pi^{e_j}\Tr(xy))dxdy\\
 &\times \mathlarger{\prod}_{\substack{2n-t+1 \leq k \leq 2n\\ j \in A_2^h(Y),\text{ } j < \sigma(j)}} \int_{\pi O \times \pi O} \psi_v(\pi^{e_j-1}\Tr(xy))dxdy\\
 &\times \mathlarger{\prod}_{\substack{k \leq 2n-t \text{ or } 2n+1 \leq k \leq 2n+2r \\ j \in B_1^h(Y)}} \int_{O \times O} \psi_v(\pi^{e_j}\Tr(xy))dxdy\\
 &\times \mathlarger{\prod}_{\substack{2n-t+1 \leq k \leq 2n\\ j \in B_1^h(Y)}} \int_{\pi O \times O} \psi_v(\pi^{e_j-1}\Tr(xy))dxdy\\
 &\times \mathlarger{\prod}_{\substack{k \leq 2n-t \text{ or } 2n+1 \leq k \leq 2n+2r \\ j \in C_1^h(Y)}} \int_{O} \psi_v(\pi^{e_j}\Nm(x))dx\\
 &\times \mathlarger{\prod}_{\substack{2n-t+1 \leq k \leq 2n\\ j \in C_1^h(Y)}} \int_{O} \psi_v(\pi^{e_j-1}\Nm(x))dx\\
 &\times \mathlarger{\prod}_{\substack{k \leq 2n-t \text{ or } 2n+1 \leq k \leq 2n+2r \\ j \in C_2^h(Y), \text{ } j< \sigma(j)}} \int_{O \times O} \psi_v(\pi^{e_j}\Tr(xy))dxdy\\
 &\times \mathlarger{\prod}_{\substack{2n-t+1 \leq k \leq 2n\\ j \in C_2^h(Y),\text{ } j < \sigma(j)}} \int_{ O \times  O} \psi_v(\pi^{e_j-1}\Tr(xy))dxdy. 
\end{align*}

Therefore, by Lemma \ref{lemma 3.12}, we have

\begin{align*}
\CF_h(Y,A_t^{[r]})&=\mathlarger{\prod}_{j \in A_1^h(Y)} (-q)^{(2n-t+2r)\min( 0, e_j )+t \min( 0, e_j+1 )-2t}\\
&\times \mathlarger{\prod}_{\substack{j \in A_2^h(Y)\\ j < \sigma(j)}} (-q)^{(2n-t+2r) \mathsmaller{\times} 2\min( 0, e_j )+t \mathsmaller{\times} 2\min( 0, e_j+1 )-4t}\\
&\times \mathlarger{\prod}_{j \in B_1^h(Y)} (-q)^{(2n-t+2r) \mathsmaller{\times} 2\min( 0, e_j )+t \mathsmaller{\times} 2\min( 0, e_j )-2t}\\
&\times \mathlarger{\prod}_{j \in C_1^h(Y)} (-q)^{(2n-t+2r) \mathsmaller{\times} \min( 0, e_j )+t \mathsmaller{\times} \min( 0, e_j-1 )}\\
&\times \mathlarger{\prod}_{\substack{j \in C_2^h(Y)\\ j < \sigma(j)}} (-q)^{(2n-t+2r) \mathsmaller{\times} 2\min( 0, e_j )+t \mathsmaller{\times} 2\min( 0, e_j-1 )}.
\end{align*}

Similarly, we have (here, we use notations in the proof of Lemma \ref{lemma3.8} and suppress $h$ from the notation $^{\wedge_h}$)

\begin{align*}
\CF_{2n-h}(Y^{\wedge},A_t^{[r]})&=\mathlarger{\prod}_{j^{\wedge} \in A_1^{2n-h}(Y^{\wedge})} (-q)^{(2n-t+2r)\min( 0, f_{j^{\wedge}} )+t \min( 0, f_{j^{\wedge}}+1 )-2t}\\
&\times \mathlarger{\prod}_{\substack{j^{\wedge} \in A_2^{2n-h}(Y^{\wedge})\\ j^{\wedge} < \sigma^{\wedge}(j^{\wedge})}} (-q)^{(2n-t+2r) \mathsmaller{\times} 2\min( 0, f_{j^{\wedge}} )+t \mathsmaller{\times} 2\min( 0, f_{j^{\wedge}}+1 )-4t}\\
&\times \mathlarger{\prod}_{j^{\wedge} \in B_1^{2n-h}(Y^{\wedge})} (-q)^{(2n-t+2r) \mathsmaller{\times} 2\min( 0, f_{j^{\wedge}} )+t \mathsmaller{\times} 2\min( 0, f_{j^{\wedge}} )-2t}\\
&\times \mathlarger{\prod}_{j^{\wedge} \in C_1^{2n-h}(Y^{\wedge})} (-q)^{(2n-t+2r) \mathsmaller{\times} \min( 0, f_{j^{\wedge}} )+t \mathsmaller{\times} \min( 0, f_{j^{\wedge}}-1 )}\\
&\times \mathlarger{\prod}_{\substack{j^{\wedge} \in C_2^{2n-h}(Y^{\wedge})\\ j^{\wedge} < \sigma^{\wedge}(j^{\wedge})}} (-q)^{(2n-t+2r) \mathsmaller{\times} 2\min( 0, f_{j^{\wedge}} )+t \mathsmaller{\times} 2\min( 0, f_{j^{\wedge}}-1 )}.
\end{align*}

Note that we have 
\begin{equation*}
\begin{array}{l}
j^{\wedge} \in A_i^{2n-h}(Y^{\wedge}) \text{ if and only if } j \in C_i^{h}(Y)\text{ } (i=1,2),\\
j^{\wedge} \in C_i^{2n-h}(Y^{\wedge}) \text{ if and only if } j \in A_i^{h}(Y)\text{ } (i=1,2),\\
j^{\wedge} \in B_1^{2n-h}(Y^{\wedge}) \text{ if and only if } j \in B_2^{h}(Y),\\
j^{\wedge} \in B_2^{2n-h}(Y^{\wedge}) \text{ if and only if } j \in B_1^{h}(Y).
\end{array}
\end{equation*}

Therefore, we have
\begin{equation*}
\begin{array}{ll}
f_{j^{\wedge}}=e_j+1 & \text{if } j \in A_1^h(Y) \cup A_2^h(Y),\\
f_{j^{\wedge}}=f_{\sigma^{\wedge}(j^{\wedge})}=e_j=e_{\sigma(j)} & \text{if } j \in B_1^h(Y) \cup B_2^h(Y),\\
f_{j^{\wedge}}=e_j-1 & \text{if } j \in C_1^h(Y) \cup C_2^h(Y).
\end{array}
\end{equation*}

This implies that

\begin{align*}
\CF_{2n-h}(Y^{\wedge},A_t^{[r]})&=\mathlarger{\prod}_{j \in A_1^h(Y)} (-q)^{(2n-t+2r)\min( 0, e_j+1 )+t \min( 0, e_j )}\\
&\times \mathlarger{\prod}_{\substack{j \in A_2^h(Y)\\ j < \sigma(j)}} (-q)^{(2n-t+2r) \mathsmaller{\times} 2\min( 0, e_j+1 )+t \mathsmaller{\times} 2\min( 0, e_j )}\\
&\times \mathlarger{\prod}_{j \in B_1^h(Y)} (-q)^{(2n-t+2r) \mathsmaller{\times} 2\min( 0, e_j )+t \mathsmaller{\times} 2\min( 0, e_j )-2t}\\
&\times \mathlarger{\prod}_{j \in C_1^h(Y)} (-q)^{(2n-t+2r) \mathsmaller{\times} \min( 0, e_j-1 )+t \mathsmaller{\times} \min( 0, e_j )-2t}\\
&\times \mathlarger{\prod}_{\substack{j \in C_2^h(Y)\\ j < \sigma(j)}} (-q)^{(2n-t+2r) \mathsmaller{\times} 2\min( 0, e_j-1 )+t \mathsmaller{\times} 2\min( 0, e_j )-4t}.
\end{align*}

Therefore,
\begin{align*}
\CF'_{h}(Y,A_n^{[0]})=&\lbrace \mathlarger{\sum}_{j \in A_1^h(Y)} \min(0,e_j) + \mathlarger{\sum}_{\substack{j \in A_2^h(Y)\\ j < \sigma(j)}} 2\min(0,e_j)\\
&+\mathlarger{\sum}_{j \in B_1^h(Y)} 2\min(0,e_j)+\mathlarger{\sum}_{j \in C_1^h(Y)} \min(0,e_j)+\mathlarger{\sum}_{\substack{j \in C_2^h(Y)\\ j < \sigma(j)}} 2\min(0,e_j) \rbrace\\
& \times f_h(Y)(-q)^{-2n(2n-h)},
\end{align*}
and
\begin{align*}
\CF'_{2n-h}(Y^{\wedge},A_n^{[0]})=&\lbrace \mathlarger{\sum}_{j \in A_1^h(Y)} \min(0,e_j+1) + \mathlarger{\sum}_{\substack{j \in A_2^h(Y)\\ j < \sigma(j)}} 2\min(0,e_j+1)\\&+\mathlarger{\sum}_{j \in B_1^h(Y)} 2\min(0,e_j)+
\mathlarger{\sum}_{j \in C_1^h(Y)} \min(0,e_j-1)\\
&+\mathlarger{\sum}_{\substack{j \in C_2^h(Y)\\ j < \sigma(j)}} 2\min(0,e_j-1) \rbrace \times f_h(Y)(-q)^{-2nh}.
\end{align*}

Note that
\begin{equation*}
\min(0,e_j+1)-\min(0,e_j)=\left\lbrace
\begin{array}{ll}
1 & \text{ if }e_j \leq -1\\
0 & \text{ otherwise}
\end{array}\right.,
\end{equation*}
\begin{equation*}
\min(0,e_j)-\min(0,e_j-1)=\left\lbrace
\begin{array}{ll}
1 & \text{ if }e_j \leq 0\\
0 & \text{ otherwise}
\end{array}\right..
\end{equation*}

This implies that
\begin{equation*}
\begin{array}{l}
\lbrace \mathlarger{\sum}_{j \in A_1^h(Y)} \min(0,e_j)-\min(0,e_j+1) + \mathlarger{\sum}_{\substack{j \in A_2^h(Y)\\ j < \sigma(j)}} 2\min(0,e_j)-2\min(0,e_j+1)\\
+\mathlarger{\sum}_{j \in C_1^h(Y)} \min(0,e_j)-\min(0,e_j-1)+\mathlarger{\sum}_{\substack{j \in C_2^h(Y)\\ j < \sigma(j)}} 2\min(0,e_j)-2\min(0,e_j-1) \rbrace\\
=\FC(Y)-\FA(Y).
\end{array}
\end{equation*}

Therefore, by Lemma \ref{lemma3.9}, we have
\begin{align*}
\dfrac{\CF_h'(Y,A_n^{[0]})}{\alpha(Y;\Gamma_{2n})}-\dfrac{\CF_{2n-h}'(Y^{\wedge},A_n^{[0]})}{\alpha(Y^{\wedge_h};\Gamma_{2n})}&=\dfrac{(\FC(Y)-\FA(Y))f_h(Y)(-q)^{-2n(2n-h)}}{\alpha(Y;\Gamma_{2n})}.
\end{align*}

Also, we can write
\begin{equation*}
\CF_h(Y,A_t^{[0]})=(-q)^{((n-t)(\FC(Y)-\FA(Y)-2t(2n-h))}f_h(Y),
\end{equation*}
and
\begin{equation*}
\CF_{2n-h}(Y^{\wedge},A_t^{[0]})=(-q)^{((n-t)(\FA(Y)-\FC(Y)-2th)}f_h(Y).
\end{equation*}

This finishes the proof of the lemma.
\end{proof}

\begin{theorem}\label{theorem3.14}
	There are the unique constants 
	\begin{equation*}
	\beta_{0}^h, \dots, \beta_{n-1}^h, \beta_0^{2n-h},\dots,\beta_{n-1}^{2n-h},\delta_h,
	\end{equation*}
	in $E_v$ such that
	\begin{equation*}
	\begin{array}{l}
	W_{h,n}'(B,0)-W_{2n-h,n}'(B^{\vee_h},0)\\
	=\mathlarger{\sum}_{0 \leq i \leq n-1} \beta_i^hW_{h,i}(B,0)
	-\mathlarger{\sum}_{0 \leq j \leq n-1}\beta_j^{2n-h}W_{2n-h,j}(B^{\vee_h},0)
	+\delta_hW_{h,n}(B,0).
	\end{array}
	\end{equation*}
\end{theorem}

\begin{proof}
	By Lemma \ref{lemma3.3}, Lemma \ref{lemma3.8}, Lemma \ref{lemma3.9}, and Lemma \ref{lemma 3.13}, we can write
	\begin{equation*}
	\begin{array}{l}
		W_{h,n}'(B,0)-W_{2n-h,n}'(B^{\vee_h},0)\\
		=\mathlarger{\sum}_{i=-(2n-h)}^{h}\mathlarger{\sum}_{\substack{Y\\\FC(Y)-\FA(Y)=i}}  \dfrac{\CG(Y,B)f_h(Y)(-q)^{-2n(2n-h)}}{\alpha(Y;\Gamma_{2n})} \times i.
	\end{array}
	\end{equation*}
	
	On the other hand, we have
	\begin{equation*}
	W_{h,t}(B,0)=\mathlarger{\sum}_{i=-(2n-h)}^{h}\mathlarger{\sum}_{\substack{Y\\\FC(Y)-\FA(Y)=i}}  \dfrac{\CG(Y,B)f_h(Y)}{\alpha(Y;\Gamma_{2n})} \times (-q)^{(n-t)i-2t(2n-h)},
	\end{equation*}	
	
		\begin{equation*}
		\begin{array}{lll}
	W_{2n-h,t}(B^{\vee_h},0)&=\mathlarger{\sum}_{i=-(2n-h)}^{h}\mathlarger{\sum}_{\substack{Y\\\FC(Y)-\FA(Y)=i}}&  \dfrac{\CG(Y,B)f_h(Y)}{\alpha(Y;\Gamma_{2n})}\\
	&& \times q^{-(2n-h)^2+h^2} (-q)^{-(n-t)i-2th}.
	\end{array}
	\end{equation*}	
	
	We denote by $m_{it}, n_{it}$ the coefficients of $W_{h,t}(B,0)$ and $W_{2n-h,t}(B^{\wedge_h},0)$, respectively. More precisely,
	\begin{equation*}
	\begin{array}{l}
	m_{it}=(-q)^{(n-t)(i-(2n-h+1))-2t(2n-h)}\\
	n_{it}=q^{-(2n-h)^2+h^2}(-q)^{-(n-t)(i-(2n-h+1))-2th}.
	\end{array}
	\end{equation*}
	
	Note that $m_{in}=(-q)^{-2n(2n-h)}$ for all $1 \leq i \leq 2n+1$.
	
	We define a $(2n+1) \mathsmaller{\times} (2n+1)$ matrix $\FB$
	\begin{equation*}
	\begin{array}{l}
	\FB=\\\left(\begin{array}{ccccccc}
	m_{10}& \dots& m_{1(n-1)}& n_{10}& \dots& n_{1(n-1)} & m_{1n}\\
	\vdots&\vdots&\vdots&\vdots&\vdots&\vdots&\vdots\\
	m_{(2n+1)0}& \dots &m_{(2n+1)(n-1)} & n_{(2n+1)0} & \dots & n_{(2n+1)(n-1)}&m_{2n+1,n}
	\end{array}\right)
	\end{array}
	\end{equation*}
	
	Then, we should find constants
	\begin{equation*}
	\beta_{0}^h, \dots, \beta_{n-1}^h, \beta_0^{2n-h},\dots,\beta_{n-1}^{2n-h},\delta_h,
	\end{equation*}
	such that
	
	\begin{equation*}
	\FB\left(\begin{array}{l}
	\beta_0^h\\
	\vdots\\
	\beta_{n-1}^h\\
	-\beta_{0}^{2n-h}\\
	\vdots\\
	-\beta_{n-1}^{2n-h}\\
	\delta_h
	\end{array}
	\middle)=(-q)^{-2n(2n-h)}\middle(\begin{array}{c}
	-(2n-h)\\
	\vdots\\
	-1\\
	0\\
	1\\
	\vdots\\
	h
	\end{array}\right)
	\end{equation*}
	
	It is easy to check that $\FB$ is invertible, and this finishes the proof.
\end{proof}

\subsection{Local conjectures on the arithmetic intersection numbers $\CI_t^n(\tb{x},\tb{y})$}\label{section3.2}
\begin{conjecture}\label{conjecture1}
	
	For a basis $\lbrace x_1, \dots,x_n,y_1,\dots,y_n \rbrace$ of $\BV$, we have
	 \begin{equation*}
	 \CI_n^n(\tb{x},\tb{y})=\dfrac{1}{W_{n,n}(A_n,0)}\lbrace W'_{n,n}(B,0)-\mathlarger{\sum}_{0 \leq i \leq n-1} \beta_i^nW_{n,i}(B,0)\rbrace.
	 \end{equation*}
	 Here $B$ is the matrix
	 	\begin{displaymath}
	 B=\left(\begin{array}{cc} 
	 h(x_i,x_j) & h(x_i,y_l)\\
	 h(y_k,x_j)& h(y_k,y_l)
	 \end{array} 
	 \right)_{1\leq i,j,k,l \leq n}.
	 \end{displaymath}
	 
\end{conjecture}

\begin{definition}
	For $B \in X_{2n}(E_v)$, we define the function $\CJ_t^n(B)$ by
	\begin{equation*}
	\CJ_t^n(B)=\dfrac{1}{W_{n,n}(A_n,0)}\lbrace W'_{t,n}(B,0)-\mathlarger{\sum}_{0 \leq i \leq n-1} \beta_i^tW_{t,i}(B,0)\rbrace.
	\end{equation*}
\end{definition}

Now, we need to recall some facts about the usual representation densities. For $A \in V_m(O_{E_v})$ and $B \in V_{2n}(O_{E_v})$, we write the usual representation density $\alpha(A,B)$ by
\begin{equation*}
\alpha(A,B)= \lim_{d \rightarrow \infty} (q^{-d})^{2n(2m-2n)} \vert \CA_d(A,B) \vert,
\end{equation*}
where $\CA_d(A,B)=\lbrace x \in M_{m,2n}(O_{E_v}/\pi^d O_{E_v}) \vert A[x] \equiv B (\Mod \pi^d) \rbrace.$

\begin{theorem}\label{theorem3.16}
	(Compatibility with $\CN^0(1,n-1)$) Let
	\begin{equation*}
	B=\left(\begin{array}{ll}
	B_1 & 0 \\
	0 & \pi^{-1}1_n
	\end{array}\right), \quad B_1\in X_{n}(E_v).
	\end{equation*}
	
	Then
	\begin{equation*}
	\CJ_n^n(B)=\dfrac{1}{W_{n,n}(A_n,0)}\lbrace W'_{n,n}(B,0)-\mathlarger{\sum}_{0 \leq i \leq n-1} \beta_i^nW_{n,i}(B,0)\rbrace=\dfrac{\alpha'(1_n,B_1)}{\alpha(1_n,1_n)}.
	\end{equation*}
\end{theorem}

Here, we follow the proof of \cite[Proposition 9.3]{KR2}.
Let $L$ be an $O_{E_v}$-lattice with basis $\lbrace u_1,\dots, u_{2n} \rbrace$ such that the matrix of inner product $(\langle u_i,u_j \rangle)$ is $\pi B$. Let $L_1$ be the $O_{E_v}$-lattice of rank $n$ with basis $\lbrace u_1, \dots, u_n \rbrace$, and let $L_2$ be the $O_{E_v}$-lattices of rank $n$ with basis $\lbrace u_{n+1},\dots, u_{2n} \rbrace$. Note that $L_2$ is unimodular with respect to  $\langle \cdot, \cdot \rangle$.

Let $M$ be an $O_{E_v}$-lattice with basis $\lbrace v_1, \dots v_m \rbrace$ ($m=2n+2r$) such that the matrix of inner product $(\langle v_i, v_j \rangle)$ is $\pi A_n^{[r]}$. Therefore, the dual $\pi M^{\vee}$ of $M$ with respect to $\langle \cdot, \cdot \rangle$ is the $O_{E_v}$-lattice with basis 
\begin{equation*}
\lbrace v_1, \dots, v_n, \pi v_{n+1},\dots,\pi v_{2n},v_{2n+1},\dots, v_{m} \rbrace.
\end{equation*}

We define the set
\begin{equation*}
\begin{array}{lll}
J_d(L,M)&=\lbrace & \varphi \in \Hom_{O_{E_v}}(L, M/\pi^{d} M) \vert\\
&&\langle \varphi(x), \varphi(y) \rangle\equiv \langle x, y \rangle \Mod \pi^d, \quad \forall x,y \in L,\\
&&\varphi(L_1) \subset \pi M^{\vee} \rbrace.

\end{array}
\end{equation*}
Then,
\begin{equation}\label{eq3.2.1}
W_{n,n}(B,r)= q^{-2dm(2n)}q^{(d-1)(2n)^2} \vert J_d(L,M) \vert,
\end{equation}
for sufficiently large $d$.

We denote by $U(M)$ the group of isometries of $M$.

For a sublattice $N \subset M$ such that $N$ is isometric to $L$, we define the sets,
\begin{equation*}
\begin{array}{ll}
I_d(L,M)=\lbrace & \varphi \in \Hom_{O_{E_v}}(L, M/\pi^{d} M) \vert\\
&\langle \varphi(x), \varphi(y) \rangle\equiv \langle x, y \rangle \Mod \pi^{d}, \quad \forall x,y \in L \rbrace,
\end{array}
\end{equation*}

\begin{equation*}
\begin{array}{ll}
\tilde{I}_d(L,M)=\lbrace & \varphi \in \Hom_{O_{E_v}}(L, M) \vert\\
&\langle \varphi(x), \varphi(y) \rangle\equiv \langle x, y \rangle \Mod \pi^{d}, \quad \forall x,y \in L \rbrace,
\end{array}
\end{equation*}

and
\begin{equation*}
\begin{array}{ll}
I_d(L,M;N)=\lbrace & \varphi \in I_d(L,M) \vert\\
&\exists \eta \in U(M) \text{ such that } \tilde{\varphi}(L)=\eta(N) \rbrace.
\end{array}
\end{equation*}
Here $\tilde{\varphi} \in \tilde{I}_d(L,M)$ is a preimage of $\varphi$. By \cite[Lemma 9.6, Lemma 9.7, Lemma 9.8]{KR2}, these sets are well-defined for sufficiently large $d$.

\begin{lemma}\label{lemma3.17}
	If $\varphi \in U(M)$, then $\varphi(\pi M^{\vee}) \subset \pi M^{\vee}$.
\end{lemma}
\begin{proof}
	Let $\varphi(v_i)=\mathlarger{\sum}_{j=1}^{m} x_{ji}v_j$. Since $\varphi$ is an isometry, we have
	\begin{equation*}
	\begin{array}{l}
	\left( \begin{array}{lll}
	x_{11}^*&\dots&x_{m1}^*\\
	\dots &\dots &\dots\\
	x_{1m}^* & \dots & x_{mm}^*
	\end{array}\middle)\middle(
	\begin{array}{lll}
	\pi 1_n & 0 & 0 \\
	0 & 1_n & 0 \\
	0 & 0 & \pi 1_{2r}
	\end{array} \middle) \middle(
	\begin{array}{lll}
	x_{11} & \dots & x_{1m} \\
	\dots & \dots & \dots \\
	x_{m1} & \dots & x_{mm}
	\end{array}
\right)\\
=\left( \begin{array}{lll}
\pi 1_n & 0 & 0 \\
0 & 1_n & 0 \\
0 & 0 & \pi 1_{2r}
\end{array}
\right).
\end{array}
	\end{equation*}
	
Therefore, by reduction modulo $\pi$, this implies that

	\begin{equation}\label{equation3.2.1}
\left( \begin{array}{l}
\mathlarger{\sum}_{k=n+1}^{2n} x_{ki}^*x_{kj}
\end{array}\middle)_{1\leq i,j \leq m}\equiv \middle(
\begin{array}{lll}
0 & 0 & 0 \\
0 & 1_n & 0 \\
0 & 0 & 0
\end{array}
\right) (\Mod \pi).
\end{equation}

If we write
\begin{equation*}
A=\left(\begin{array}{lll}
x_{n+1,1} & \dots & x_{n+1,n} \\
\dots&\dots&\dots\\
x_{2n,1}&\dots&x_{2n,n}
\end{array}\right),
\end{equation*}

\begin{equation*}
B=\left(\begin{array}{lll}
x_{n+1,n+1} & \dots & x_{n+1,2n} \\
\dots&\dots&\dots\\
x_{2n,n+1}&\dots&x_{2n,2n}
\end{array}\right),
\end{equation*}

\begin{equation*}
C=\left(\begin{array}{lll}
x_{n+1,2n+1} & \dots & x_{n+1,m} \\
\dots&\dots&\dots\\
x_{2n,2n+1}&\dots&x_{2n,m}
\end{array}\right),
\end{equation*}

then \eqref{equation3.2.1} implies that
\begin{equation*}
\left(\begin{array}{lll}
^tA^*A & ^tA^*B & ^tA^*C \\
^tB^*A&^tB^*B&^tB^*C\\
^tC^*A&^tC^*B&^tC^*C
\end{array}\middle) \equiv 
\middle(\begin{array}{lll}
0 & 0 & 0 \\
0 & 1_n & 0 \\
0 & 0 & 0
\end{array}\right) (\Mod \pi).
\end{equation*}

Since $^tB^*B=1_n$, $B$ is invertible. Therefore, $A$ and $C$ are $0$ modulo $\pi$, and this implies that $\varphi(\pi M^{\vee}) \subset \pi M^{\vee}$.
\end{proof}

\begin{proposition}\label{proposition3.18}(cf. \cite[Proposition 9.9]{KR2})
	Let $\lbrace N_i \rbrace$ be a set of representatives for the $U(M)$-orbits in the set of all sublattices $N$ of $M$ such that $N$ is isometric to $L_2$. Then we have
	\begin{equation*}
	\begin{array}{ll}
	\vert J_d(L,M) \vert=&\sum_{i}\vert I_d(L_2,M;N_i) \vert\\
	&\times \vert \lbrace \varphi_1 \in J_d(L_1,M) \vert \langle \varphi(L_1), N_i \rangle \equiv 0 (\Mod \pi^d) \rbrace.
	\end{array}
	\end{equation*}
\end{proposition}

\begin{proof}
	The proof of this proposition is almost identical to the proof of \cite[Proposition 9.9]{KR2}. By \cite[Lemma 9.7]{KR2}, for any $\varphi_2 \in \tilde{I}_d(L_2,M)$, $\varphi_2(L_2)$ is isometric to $L_2$ for sufficiently large $d$. Therefore, for each $\varphi_2 \in \tilde{I}_d(L_2,M)$, we can choose an isometry $\vartheta(\varphi_2) \in U(M)$ such that $N_i=\vartheta(\varphi_2)(\varphi_2(L_2))$ for some $i$. Also, for each $\varphi_2 \in I_d(L_2,M)$, we can choose its preimage $\tilde{\varphi}_2$ in $\tilde{I}_d(L_2,M)$.
	
	Then we have a bijection
	\begin{equation*}
	J_d(L,M) \rightarrow \mathlarger{\mathlarger{\sqcup}}_i \lbrace \varphi_1 \in J_d(L_1,M) \vert \langle \varphi(L_1), N_i \rangle \equiv 0 (\Mod \pi^d) \rbrace \times I_d(L_2,M;N_i)
	\end{equation*}
	given by 
	\begin{equation*}
	\varphi \mapsto (\vartheta(\varphi\vert_{L_2})\circ \varphi\vert_{L_1},\varphi\vert_{L_2}).
	\end{equation*}
	Indeed, the condition $\varphi(L_1) \subset \pi M^{\vee}$ and Lemma \ref{lemma3.17} imply that 
	\begin{equation*}
	\vartheta(\varphi\vert_{L_2})\circ \varphi\vert_{L_1}(L_1) \subset \pi M^{\vee}.
	\end{equation*}
\end{proof}

Note that $L_2$ is unimodular. Therefore, by the arguments in \cite[p.680]{KR2}, all sublattices $N \subset M$ such that $N$ is isometric to $L_2$ are in the same $U(M)$-orbit, and $M=N \perp N^{\perp}$ where $N^{\perp}:=(E_v N)^{\perp} \cap M$. Let us fix $N$ the sublattice of $M$ with the basis $\lbrace v_{n+1}, \dots, v_{2n} \rbrace$ as a representative of the unique $U(M)$-orbit of $L_2$. Then, $N^{\perp}$ is the lattice with basis $\lbrace v_1, \dots, v_n, v_{2n+1}, \dots, v_{m} \rbrace$. In particular, $N^{\perp} \subset \pi M^{\vee}$.

We have the following analogue of \cite[Lemma 9.10]{KR2}
\begin{lemma}\label{lemma3.19}
	\begin{equation}\label{eq3.2.3}
	\begin{array}{l}
	\vert \lbrace \varphi_1 \in J_d(L_1,M) \vert \langle \varphi_1(L_1),N \rangle \equiv 0 \text{ }(\Mod \pi^d) \rbrace \vert=\vert I_d(L_1,N^{\perp})\vert
	\end{array}
	\end{equation}
	for sufficiently large $d$.
\end{lemma}
\begin{proof}
	The proof of this lemma is almost identical to the proof of \cite[Lemma 9.10]{KR2}. First note that
	\begin{equation*}
	\lbrace x \in M \vert \langle x , N \rangle \equiv 0 \text{ } (\Mod \pi^d) \rbrace=\pi^d N^{\vee} \perp N^{\perp},
	\end{equation*}
	for sufficiently large $d$ such that $\pi^d N^{\vee} \subset N$.
	
	Therefore, we have
	\begin{equation*}
	\begin{array}{ll}
		\lbrace \varphi_1 \in J_d(L_1,M) \vert \quad\quad \langle \varphi_1(L_1) , N \rangle \equiv 0 \text{ } &(\Mod \pi^d) \rbrace\\
		=\lbrace \varphi_1:L_1 \rightarrow \pi^d N^{\vee} \perp N^{\perp} \text{ } (\Mod \pi^d) \vert& \langle \varphi_1(x), \varphi_1(y) \rangle \equiv \langle x,y\rangle \text{ } (\Mod \pi^d)\\
		 &\varphi_1(L_1) \subset \pi M^{\vee} \rbrace.
	\end{array}
	\end{equation*}
	Since $\pi^{d}N^{\vee} \subset \pi M^{\vee}$ for sufficiently large $d$ and $N^{\perp} \subset \pi M^{\vee}$, the condition $\varphi_1(L_1) \subset \pi M^{\vee}$ is automatic in this case. Therefore, \eqref{eq3.2.3} is equal to
	\begin{equation*}
	\lbrace \varphi_1:L_1 \rightarrow \pi^d N^{\vee} \perp N^{\perp} \text{ } (\Mod \pi^d) \vert \langle \varphi_1(x), \varphi_1(y) \rangle \equiv \langle x,y\rangle \text{ } (\Mod \pi^d)\rbrace.
	\end{equation*}
	
	Now, fix a positive integer $a$ such that $\pi^a N^{\vee} \subset N$.
	
	Then we have 
	\begin{equation*}
	\begin{array}{lll}
	\quad \vert\lbrace \varphi_1:L_1 \rightarrow \pi^d N^{\vee} \perp N^{\perp} \text{ } & (\Mod \pi^d) \vert \langle \varphi_1(x), \varphi_1(y) \rangle \equiv \langle x,y\rangle \text{ } (\Mod \pi^d)\rbrace\vert&\\
	=\vert\lbrace \varphi_1:L_1 \rightarrow \pi^d N^{\vee} \perp N^{\perp} \text{ }& (\Mod \pi^d (\pi^aN^{\vee} \perp N^{\perp})) \vert &\\
	&\langle \varphi_1(x), \varphi_1(y) \rangle \equiv \langle x,y\rangle \text{ } (\Mod \pi^d)\rbrace\vert\\
	\times \vert M:\pi^a N^{\vee} \perp N^{\perp} \vert^{-n}
	\end{array}
	\end{equation*}
	by replacing $M$ by $\pi^a N^{\vee} \perp N^{\perp}$.
	
	Now, we can write $\varphi_1=\psi_1 + \psi_2$ where
	\begin{equation*}
	\begin{array}{l}
	\psi_1:L_1 \rightarrow \pi^d N^{\vee} / \pi^{d+a}N^{\vee},\\
	\psi_2:L_1 \rightarrow N^{\perp}/\pi^d N^{\perp}.
	\end{array}
	\end{equation*}
	
	Since we assume that $d$ is sufficiently large such that $\pi^d N^{\vee} \subset N$, we have
	\begin{equation*}
	\langle \psi_1(x),\psi_1(y) \rangle \in \langle N, \pi^d N^{\vee} \rangle \subset \pi^dO_{E_v}.
	\end{equation*}
	
	This means that the condition $\langle \varphi_1(x),\varphi_1(y) \rangle \equiv \langle x,y \rangle \text{ }(\Mod \pi^d)$ is the same as the condition $\langle \psi_2(x),\psi_2(y) \rangle \equiv \langle x,y \rangle \text{ } (\Mod \pi^d)$, and we do not need to impose any condition on $\psi_1$.
	
	Therefore, \eqref{eq3.2.3} is equal to
	\begin{equation*}
	\begin{array}{l}
	\vert M:\pi^aN^{\vee} \perp N^{\perp} \vert^{-n}\vert \pi^d N^{\vee} : \pi^{d+a}N^{\vee}\vert^n \vert I_d(L_1,N^{\perp})\vert\\
	=\vert M:N^{\vee}\perp N^{\perp} \vert^{-n} \vert I_d(L_1,N^{\perp}) \vert\\
	=\vert M:N\perp N^{\perp} \vert^{-n}\vert N^{\vee}:N \vert^n \vert I_d(L_1,N^{\perp}) \vert\\
	=\vert I_d(L_1,N^{\perp}) \vert.
	\end{array}
	\end{equation*}
	Here, we used the fact that $M=N \perp N^{\perp}$ and $N=N^{\vee}$ (since $N$ is unimodular).
	
\end{proof}

Now, Proposition \ref{proposition3.18} and Lemma \ref{lemma3.19} imply the following proposition.

\begin{proposition}\label{proposition3.20}
	$\vert J_d(L,M)\vert=\vert I_d(L_2,M)\vert \vert I_d(L_1,N^{\perp}) \vert$.
\end{proposition}

\begin{proof}[Proof of Theorem \ref{theorem3.16}]
	Note that
	\begin{equation*}
	\begin{array}{l}
	\vert I_d(L_2,M) \vert=\vert \CA_d(\pi A_n^{[r]},1_n) \vert,\\
	\vert I_d(L_1,N^{\perp}) \vert=\vert \CA_d(
	\left(\begin{array}{ll}
	\pi 1_n & 0 \\
	0 & \pi 1_{2r}
	\end{array}\right),\pi B_1) \vert.\\
	\end{array}
	\end{equation*}
	
	Therefore, we have
	
	\begin{equation*}
	\begin{array}{l}
	(q^{-d})^{n(2m-n)} \vert I_d(L_2,M) \vert=\alpha(\pi A_n^{[r]},1_n),\\
	
	(q^{-d})^{n(2(m-n)-n)} \vert I_d(L_1,N^{\perp}) \vert=\alpha(\left(\begin{array}{ll}
	\pi 1_n & 0 \\
	0 & \pi 1_{2r}
	\end{array}\right),\pi B_1),
	\end{array}
	\end{equation*}
	for sufficiently large $d$. Note that $\alpha(\pi 1_n,\pi B_1)=0$. This is because $\det(B) \equiv n+1 \text{ }(\Mod 2)$ and hence $ \det(\pi B_1) \neq \det(\pi 1_n) \text{ }(\Mod 2)$. 
	
	Now, let $\Delta$ be the $O_{E_v}$-lattice with basis $\lbrace w_1, \dots, w_{2n} \rbrace$ such that the matrix of inner product $(\langle w_i,w_j \rangle)$ is
	\begin{equation*}
	\pi A_n = \left(\begin{array}{ll}
	\pi 1_n&0\\
	0&1_n
	\end{array}\right).
	\end{equation*}
	
	Also, let $\Delta_1$ be the $O_{E_v}$-lattice with basis $\lbrace w_1, \dots, w_{n} \rbrace$ and let $\Delta_2$ be the $O_{E_v}$-lattice with basis $\lbrace w_{n+1}, \dots, w_{2n} \rbrace$
	
	Then, we have
	\begin{equation*}
	\begin{array}{l}
	(q^{-d})^{n(2m-n)} \vert I_d(\Delta_2,M) \vert=\alpha(\pi A_n^{[r]},1_n),\\
	
	(q^{-d})^{n(2(m-n)-n)} \vert I_d(\Delta_2,N^{\perp}) \vert=\alpha(\left(\begin{array}{ll}
	\pi 1_n & 0 \\
	0 & \pi 1_{2r}
	\end{array}\right),\pi 1_n),
	\end{array}
	\end{equation*}
	
	By Proposition \ref{proposition3.20}, we have
	\begin{equation}\label{eq3.2.4}
	\vert J_d(\Delta,M)\vert=\vert I_d(\Delta_2,M)\vert \vert I_d(\Delta_1,N^{\perp})\vert.
	\end{equation}
	
	Also, by equation \eqref{eq3.2.1}, we have
	\begin{equation*}
	W_{n,n}(A_n,r)=q^{-2dm(2n)}q^{(d-1)(2n)^2} \vert J_d(\Delta,M) \vert.
	\end{equation*}

Combining the above formulas, we have
\begin{equation*}
\dfrac{W_{n,n}'(B,0)}{W_{n,n}(A_n,0)}=\dfrac{\alpha'(\pi 1_n,\pi B_1)}{\alpha(\pi 1_n, \pi 1_n)}=\dfrac{\alpha'( 1_n, B_1)}{\alpha( 1_n,  1_n)}
\end{equation*}

Now, it suffices to show that $W_{n,i}(B,0)=0$ for all $0 \leq i \leq n-1$. Indeed, we have
\begin{equation*}
W_{n,i}(B,0)=q^{-2d(2n)^2}q^{(d-1)(2n)^2} \vert J_d(L,M_i) \vert,
\end{equation*} 
where $M_i$ is the $O_{E_v}$-lattice with basis $\lbrace v_{1,i}, \dots, v_{2n,i} \rbrace$ such that the matrix of inner product $(\langle v_{k,i},v_{l,i}\rangle)$ is $\pi A_i$.

Note that $L$ has the rank $n$ unimodular sublattice $L_2$. However, any unimodular sublattice of $M_i$ has rank $\leq i <n$. Therefore, there is no $\varphi \in \Hom_{O_{E_v}}(L,M_i/\pi^dM_i)$ such that $\langle \varphi(x), \varphi(y) \rangle \equiv \langle x,y \rangle \Mod \pi^d$ for sufficiently large $d$. This implies that $J_d(L,M_i)$ is an empty set. Therefore, $W_{n,i}(B,0)=0$.

\end{proof}

\begin{theorem}\label{theorem3.24}
	\begin{enumerate} $\CJ_n^n(\cdot)$ has the following properties.
		\item (Linear invariance under $GL_n(O_{E_v})\times GL_n(O_{E_v})$-action)\\
		Let $g \in {GL_{n}(O_{E_v}) \mathsmaller{\times} GL_{n}(O_{E_v})} \subset GL_{2n}(O_{E_v})$. Then, 
		\begin{equation*}
		\CJ_n^n(^tg^*Bg)=\CJ_n^n(B).
		\end{equation*}
		 
		\item (Invariance under the isomorphism $\theta: \CN \rightarrow
		 \widehat{\CN}$) We have
		 \begin{equation*}
		 \CJ_n^n(B)=\CJ_n^n(B^{\vee_n}).
		 \end{equation*}
		 \item (Compatibility with $\CN^n(1,n-1)$)
		 Let
		 \begin{equation*}
		 B=\left(\begin{array}{ll}
		 1_n & 0 \\
		 0 & B_2
		 \end{array}\right), \quad B_2\in X_{n}(E_v).
		 \end{equation*}
		 
		 Then
		 \begin{equation*}
		 \CJ_n^n(B)=\dfrac{\alpha'(1_n,\pi B_2)}{\alpha(1_n,1_n)}.
		 \end{equation*}
	\end{enumerate}
	
\end{theorem}
\begin{proof}
	\begin{enumerate}
		\item This follows from Lemma \ref{lemma3.5}.
		\item This follows from Theorem \ref{theorem3.14} and the fact that $W_{n,n}(B,0)=0$.
		\item Since $\CJ_n^n(B)=\CJ_n^n(B^{\vee_n})$, we have
		\begin{equation*}
		\CJ_n^n(B)=\CJ_n^n(\left(\begin{array}{ll}
		\pi B_2 & 0\\
		0&\pi^{-1} 1_n
		\end{array}\right)).
		\end{equation*}
		
		By Theorem \ref{theorem3.16}, this implies that
		\begin{equation*}
		\CJ_n^n(B)=\dfrac{\alpha'(1_n,\pi B_2)}{\alpha(1_n,1_n)}.
		\end{equation*}
		\end{enumerate}
\end{proof}

\begin{remark}\label{remark3.25}
	Note that when we formulate Conjecture \ref{conjecture1}, we only required that $\CI_n^n(\tb{x},\tb{y})$ is invariant under $GL_n(O_{E_v})\times GL_n(O_{E_v})$-action. However, Conjecture \ref{conjecture1} suggests that $\CI_n^n(\tb{x},\tb{y})$ is invariant under a slightly bigger group $GL_{2n}(O_{E_v}) \cap \Gamma_{n,n}$, where
	\begin{equation*}
	\Gamma_{n,n}=\lbrace \left( \begin{array}{ll} A&B\\C&D \end{array} \right) \in M_{2n,2n}(O_{E_v}) \vert A,B,D \in M_{n,n}(O_{E_v}), C \in \pi M_{n,n}(O_{E_v}) \rbrace.
	\end{equation*}
\end{remark}
We can formulate the following conjecture similar to Conjecture \ref{conjecture1}.

\begin{conjecture}\label{conjecture2}
	
	For a basis $\lbrace x_1, \dots,x_{2n-h},y_1,\dots,y_{h} \rbrace$ of $\BV$, we have
	\begin{equation*}
	\CI_h^n(\tb{x},\tb{y})=\dfrac{1}{W_{n,n}(A_n,0)}\lbrace W'_{h,n}(B,0)-\mathlarger{\sum}_{0 \leq i \leq n-1} \beta_i^hW_{h,i}(B,0)\rbrace.
	\end{equation*}
	Here $B$ is the matrix
	\begin{displaymath}
	B=\left(\begin{array}{cc} 
	h(x_i,x_j) & h(x_i,y_l)\\
	h(y_k,x_j)& h(y_k,y_l)
	\end{array} 
	\right)_{1\leq i,j \leq 2n-h, 1\leq k,l \leq h}.
	\end{displaymath}
	
\end{conjecture}

\begin{remark}
	Similar to Remark \ref{remark3.25}, Conjecture \ref{conjecture2} suggests that $\CI_h^n(\tb{x},\tb{y})$ is invariant under $GL_{2n}(O_{E_v}) \cap \Gamma_{2n-h,h}$, where
	\begin{equation*}
	\Gamma_{2n-h,h}=\left\lbrace \begin{array}{l} \left( \begin{array}{ll} A&B\\C&D \end{array} \right) \\\in M_{2n,2n}(O_{E_v}) \end{array} \vert \begin{array}{l} A \in M_{2n-h,2n-h}(O_{E_v}), B \in M_{2n-h,h}(O_{E_v}) \\ C \in \pi M_{h,2n-h}(O_{E_v}), D \in M_{h,h}(O_{E_v})  \end{array}\right\rbrace.
	\end{equation*}
\end{remark}

\bigskip

\section{Arithmetic intersection numbers of special cycles in $\CN^1(1,1)$}\label{section4}
In this section, we will prove that Conjecture \ref{conjecture1} and Conjecture \ref{conjecture2} are true for $n=1$, $0 \leq h \leq 2$.

\subsection{The cases of $n=1, h=0,2$}
In this subsection, we will prove that our conjectures are compatible with the result of \cite{San}.

First, we need the following lemma

\begin{lemma}\label{lemma4.1}
	$W_{1,1}(A_1,0)=q^{-5}(q+1)^2$.
\end{lemma}
\begin{proof}
	By Lemma \ref{lemma3.3}, we have
	\begin{equation*}
	W_{1,1}(A_1,0)=\sum_{Y \in \Gamma_{2} \backslash X_{2}(E_v)} \dfrac{\CG(Y,A_1)\CF_1(Y,A_1)}{\alpha(Y;\Gamma_{2})}.
	\end{equation*}
	
	Also, note that $\CR_2$ forms a complete set of representatives of $\Gamma_{2} \backslash X_{2}(E_v)$. Therefore, we only need to consider the following cases.
	
	\textbf{Case 1.} For $\mu_1, \mu_2 \in \BZ$,
	\begin{equation*}
	Y=\left( \begin{array}{ll}
	\pi^{\mu_1} & \\
	 & \pi^{\mu_2}
	 \end{array}\right).
	\end{equation*}
	
		\textbf{Case 2.} For $e \in \BZ$,
	\begin{equation*}
	Y=\left( \begin{array}{ll}
	&\pi^{e}  \\
	 \pi^{e}&
	\end{array}\right).
	\end{equation*}
	
	Now, we will compute $\CG(Y,A_1)$, $\CF_1(Y,A_1)$, and $\alpha(Y;\Gamma_{2})$ separately. In the following computation, we use Lemma \ref{lemma 3.12} and Lemma \ref{lemma3.14}.
	
	First, let us compute $\CF_1(Y,A_1)$. We have the following two cases.
	
	\textbf{Case 1-F.} For 	
$Y=\left( \begin{array}{ll}
	\pi^{\mu_1} & \\
	& \pi^{\mu_2}
	\end{array}\right)$, $\mu_{1}, \mu_2 \in \BZ$, we have
	\begin{equation*}
	\begin{array}{ll}
	\CF_1(Y,A_1)&=\int_{O}\psi_v(\pi^{\mu_1}\Nm(x))dx \times \int_{\pi O}\psi_v(\pi^{\mu_1-1}\Nm(x))dx\\
	& \times \int_{O}\psi_v(\pi^{\mu_2}\Nm(x))dx \times \int_{O}\psi_v(\pi^{\mu_2-1}\Nm(x))dx\\
	&=(-q)^{\min(0,\mu_1)+\min(0,\mu_1+1)-2+\min(0,\mu_2)+\min(0,\mu_2-1)}.
	\end{array}
	\end{equation*}
	
	\textbf{Case 2-F.} For
	$Y=\left( \begin{array}{ll}
	& \pi^{e}  \\
	 \pi^{e} &
	\end{array}\right)$, $e \in \BZ$, we have
	\begin{equation*}
\begin{array}{ll}
F_1(Y,A_1)&=\int_{O\times O}\psi_v(\pi^{e}\Tr(xy))dxdy \times \int_{\pi O \times O}\psi_v(\pi^{e-1}\Tr(xy))dxdy\\
&=(-q)^{2\min(0,e)+2\min(0,e)-2}\\
&=(-q)^{4\min(0,e)-2}.
\end{array}
\end{equation*}
	
	Now, let us compute $\CG(Y,A_1)$. We have the following two cases.
	
	\textbf{Case 1-G.} For 	
	$Y=\left( \begin{array}{ll}
	\pi^{\mu_1} & \\
	& \pi^{\mu_2}
	\end{array}\right)$, $\mu_{1}, \mu_2 \in \BZ$, we have
	\begin{equation*}
	\begin{array}{ll}
	\CG(Y,A_1)&=\int_{O^{\times}}\psi_v(\pi^{\mu_1}\Nm(x))dx \times \int_{O^{\times}}\psi_v(\pi^{\mu_2-1}\Nm(x))dx\\
	& \times \int_{O}\psi_v(\pi^{\mu_2}\Nm(x))dx \times \int_{\pi O}\psi_v(\pi^{\mu_1-1}\Nm(x))dx\\
	&=I^*(\mu_1)I^*(\mu_2-1)(-q)^{\min(0,\mu_1+1)+\min(0,\mu_2)-2}.
	\end{array}
	\end{equation*}
	Therefore, $\CG(Y,A_1)=0$ if $\mu_1 <-1$ or $\mu_2 <0$.
	
	\textbf{Case 2-G.}  For
	$Y=\left( \begin{array}{ll}
	& \pi^{e}  \\
	\pi^{e} &
	\end{array}\right)$, $e \in \BZ$, we have
	\begin{equation*}
	\begin{array}{ll}
	\CG(Y,A_1)&=\int_{O^{\times}\times O}\psi_v(\pi^{e}\Tr(xy))dxdy \times \int_{\pi O \times O^{\times}}\psi_v(\pi^{e-1}\Tr(xy))dxdy\\
	&=(1-q^{-2})J_1(e)(1-q^{-2})q^{-2}J_1(e)\\
	&=(-q)^{-6}(q^2-1)J_1(e)^2.
	\end{array}
	\end{equation*}
	Therefore, $\CG(Y,A_1)=0$ if $e <0$.
	
	Finally, we need to compute $\alpha(Y;\Gamma_2)$. Here we can use \cite[Theorem 2.5]{Hir}. Since we do not want to introduce more notations, we will write its computation directly. It is easy to compute it by using the above theorem (for the notations in the theorem, we refer to \cite[page 108]{Hir}).
	
	\textbf{Case 1-1-$\alpha$}  For 	
	$Y=\left( \begin{array}{ll}
		\pi^{\mu_1} & \\
		& \pi^{\mu_2}
	\end{array}\right)$, $\mu_{1} \geq \mu_2 \in \BZ$, we have
	\begin{equation*}
	\alpha(Y;\Gamma_2)=(q+1)^2q^{-4+\mu_1+3\mu_2}.
	\end{equation*}
	
	\textbf{Case 1-2-$\alpha$}  For 	
	$Y=\left( \begin{array}{ll}
	\pi^{\mu_1} & \\
	& \pi^{\mu_2}
	\end{array}\right)$, $\mu_{1}< \mu_2 \in \BZ$, we have
	\begin{equation*}
	\alpha(Y;\Gamma_2)=(q+1)^2q^{-2+3\mu_1+\mu_2}.
	\end{equation*}
	
	\textbf{Case 2-$\alpha$}  For
	$Y=\left( \begin{array}{ll}
	& \pi^{e}  \\
	\pi^{e} &
	\end{array}\right)$, $e \in \BZ$, we have
	\begin{equation*}
	\alpha(Y;\Gamma_2)=q(q^2-1)q^{-4+4e}.
	\end{equation*}
	
	Now, we have all ingredients to compute $W_{1,1}(A_1,0)$.
	Note that in \textbf{Case 1-G}, $\CG(Y,A_1)=0$ if $\mu_1 <-1$ or $\mu_2 <0$. Therefore we can assume that $\mu_{1} \geq -1, \mu_2 \geq 0$. Similarly, we can assume that $e \geq 0$ in \textbf{Case 2-G}.
	
	We will compute $\sum_{Y \in \Gamma_{2} \backslash X_{2}(E_v)} \dfrac{\CG(Y,A_1)\CF_1(Y,A_1)}{\alpha(Y;\Gamma_{2})}$ by dividing into the following 6 cases.
	
	\textbf{Case 1-1-1.} In this case we assume that $\mu_1 \geq \mu_2 \geq 1$.
	Therefore,
	\begin{equation*}
	\begin{array}{l}
	\CF_1(Y,A_1)=q^{-2}\\
	\CG(Y,A_1)=q^{-6}(q^2-1)^2\\
	\alpha(Y;\Gamma_2)=(q+1)^2q^{-4+\mu_1+3\mu_2}.
	\end{array}
	\end{equation*}
	
	This implies that
	\begin{equation*}
	\begin{array}{l}
	\sum_{\text{case 1-1-1}} \CF_1(Y,A_1)\CG(Y,A_1)\alpha(Y;\Gamma_2)^{-1}\\
	=\sum_{\mu_2=1}^{\infty}\sum_{\mu_1=\mu_2}^{\infty}q^{-8}(q^2-1)^2q^4(q+1)^{-2}q^{-\mu_1-3\mu_2}\\
	=q^{-3}(q^3+q^2+q+1)^{-1}.
	\end{array}
	\end{equation*}
	
		\textbf{Case 1-1-2.} In this case we assume that $\mu_1 \geq \mu_2=0$.
	Therefore,
	\begin{equation*}
	\begin{array}{l}
	\CF_1(Y,A_1)=-q^{-3}\\
	\CG(Y,A_1)=-q^{-6}(q^2-1)(q+1)\\
	\alpha(Y;\Gamma_2)=(q+1)^2q^{-4+\mu_1}.
	\end{array}
	\end{equation*}
	
	This implies that
	\begin{equation*}
	\begin{array}{l}
	\sum_{\text{case 1-1-2}} \CF_1(Y,A_1)\CG(Y,A_1)\alpha(Y;\Gamma_2)^{-1}\\
	=\sum_{\mu_1=0}^{\infty}q^{-5}(q-1)q^{-\mu_1}\\
	=q^{-4}.
	\end{array}
	\end{equation*}
	
		\textbf{Case 1-2-1.} In this case we assume that $\mu_2 > \mu_1 \geq 0$.
	Therefore,
	\begin{equation*}
	\begin{array}{l}
	\CF_1(Y,A_1)=q^{-2}\\
	\CG(Y,A_1)=q^{-6}(q^2-1)^2\\
	\alpha(Y;\Gamma_2)=(q+1)^2q^{-2+3\mu_1+\mu_2}.
	\end{array}
	\end{equation*}
	
	This implies that
	\begin{equation*}
	\begin{array}{l}
	\sum_{\text{case 1-2-1}} \CF_1(Y,A_1)\CG(Y,A_1)\alpha(Y;\Gamma_2)^{-1}\\
	=\sum_{\mu_1=0}^{\infty}\sum_{\mu_2=\mu_1+1}^{\infty}q^{-6}(q-1)^2q^{-3\mu_1-\mu_2}\\
	=q^{-2}(q^3+q^2+q+1)^{-1}.
	\end{array}
	\end{equation*}
	
		\textbf{Case 1-2-2.} In this case we assume that $\mu_2>0, \mu_1=-1$.
	Therefore,
	\begin{equation*}
	\begin{array}{l}
	\CF_1(Y,A_1)=-q^{-3}\\
	\CG(Y,A_1)=-q^{-6}(q^2-1)(q+1)\\
	\alpha(Y;\Gamma_2)=(q+1)^2q^{-5+\mu_2}.
	\end{array}
	\end{equation*}
	
	This implies that
	\begin{equation*}
	\begin{array}{l}
	\sum_{\text{case 1-2-2}} \CF_1(Y,A_1)\CG(Y,A_1)\alpha(Y;\Gamma_2)^{-1}\\
	=\sum_{\mu_2=1}^{\infty}q^{-4}(q-1)q^{-\mu_2}\\
	=q^{-4}.
	\end{array}
	\end{equation*}
	
		\textbf{Case 1-2-3.} In this case we assume that $\mu_2=0, \mu_1=-1$.
	Therefore,
	\begin{equation*}
	\begin{array}{l}
	\CF_1(Y,A_1)=q^{-4}\\
	\CG(Y,A_1)=-q^{-6}(q+1)^2\\
	\alpha(Y;\Gamma_2)=(q+1)^2q^{-5}.
	\end{array}
	\end{equation*}
	
	This implies that
	\begin{equation*}
	\begin{array}{l}
	\CF_1(Y,A_1)\CG(Y,A_1)\alpha(Y;\Gamma_2)^{-1}=q^{-5}.
	\end{array}
	\end{equation*}
	
		\textbf{Case 2.} In this case we assume that $e \geq 0$.
	Therefore,
	\begin{equation*}
	\begin{array}{l}
	\CF_1(Y,A_1)=q^{-2}\\
	\CG(Y,A_1)=q^{-6}(q^2-1)^2\\
	\alpha(Y;\Gamma_2)=q^{-3}(q^2-1)^2q^{4e}.
	\end{array}
	\end{equation*}
	
	This implies that
	\begin{equation*}
	\begin{array}{l}
	\sum_{\text{case 2}} \CF_1(Y,A_1)\CG(Y,A_1)\alpha(Y;\Gamma_2)^{-1}\\
	=\sum_{e=0}^{\infty}q^{-5}(q^2-1)q^{-4e}\\
	=q^{-1}(q^2+1)^{-1}.
	\end{array}
	\end{equation*}

	These 6 cases imply that
	\begin{equation*}
	\sum_{Y \in \CR_2} \CF_1(Y,A_1)\CG(Y,A_1)\alpha(Y;\Gamma_2)^{-1}=q^{-5}(q+1)^2.
	\end{equation*}
	This finishes the proof of the lemma.	
\end{proof}
\begin{theorem}\label{theorem4.2}
	Conjecture \ref{conjecture2} is true when $n=1, h=0,2$.
\end{theorem}
\begin{proof}
	First, assume that $h=0$. We need to compute the constant $\beta_0^0$ in Theorem \ref{theorem3.14}, therefore we will use the notation in the proof of the theorem. Note that
	
	\begin{equation*}
	\FB\left(\begin{array}{l}
	\beta_0^0\\
	-\beta_{0}^{2}\\
	\delta_0
	\end{array}
	\middle)=(-q)^{-4}\middle(\begin{array}{c}
	-2\\
	-1\\
	0
	\end{array}\right),
	\end{equation*}
	and
	\begin{equation*}
	\FB=\left(\begin{array}{lll}
	(-q)^{-2} & (-q)^{-2} &(-q)^{-4}\\
	(-q)^{-1} & (-q)^{-3} &(-q)^{-4}\\
	(-q)^0 &(-q)^{-4} & (-q)^{-4}
	\end{array}\right).
	\end{equation*}
	
	These two imply that $\beta_0^0=q^{-2}(q^2-1)^{-1}$.
	
	Now, we will express $W'_{0,1}(B,0)$ and $W_{0,0}(B,0)$ as usual representation densities.
	
	Recall that $1_{h,t}$ is the characteristic function of $(L_t^{\vee})^{2n-h} \times L_t^h$. Also $(L_t)^{2n}$ has measure $1$ with respect to our Haar measure. Therefore, in the case of $n=1, h=0$, $((L_1)^{\vee})^2$ has measure $q^{-4}$. This implies that
	
	\begin{equation*}
	\begin{array}{l}
	W_{0,1}'(B,0)=q^{-4}\alpha'(\left(\begin{array}{ll} \pi & \\ &1 \end{array}\right),B)\\
	W_{0,0}(B,0)=\alpha(\left(\begin{array}{ll} 1 & \\ &1\end{array}\right),B).
	\end{array}
	\end{equation*}
	
	Therefore, our conjecture is
	\begin{equation*}\begin{array}{ll}
	\CI_0^1(\tb{x})&=\dfrac{1}{W_{1,1}(A_1,0)}\lbrace W'_{0,1}(B,0)- \beta_0^0W_{0,0}(B,0)\rbrace\\
	&=q^5(q+1)^{-2}\lbrace q^{-4}\alpha'(\left(\begin{array}{ll} \pi & \\ &1 \end{array}\right),B)-q^{-2}(q^2-1)^{-1}\alpha(\left(\begin{array}{ll} 1 & \\ &1\end{array}\right),B)\rbrace\\
	&=q(q+1)^{-2}\lbrace\alpha'(\left(\begin{array}{ll} \pi & \\ &1 \end{array}\right),B)-q^2(q^2-1)^{-1}\alpha(\left(\begin{array}{ll} 1 & \\ &1\end{array}\right),B)\rbrace.
	\end{array}
	\end{equation*}
	
	This coincides with \cite[Corollary 3.6]{San}, therefore our conjecture is true for $n=1, h=0$. The case of $n=1, h=2$ is similar by using \cite[Corollary 3.6]{San} and Theorem \ref{theorem3.24} (2).
	
\end{proof}

\subsection{The case of $n=1, h=1$}
In this subsection, we will prove Conjecture \ref{conjecture1} when $n=1$.
For this, we need to compute the intersection number of special cycles $\CZ(x), \CY(y)$, where $x,y \in \BV$. Let
\begin{equation*}
B=\left( \begin{array}{ll}
h(x,x) & h(x,y)\\
h(y,x) & h(y,y)
\end{array}\right).
\end{equation*}

In the case $\val(h(y,y))=-1$, we already computed the intersection number of $\CZ(x), \CY(y)$ in \cite{Cho}. This computation and Theorem \ref{theorem3.16} imply that Conjecture \ref{conjecture1} holds in this case. Therefore, we may assume that $\val(h(x,x)) \geq 0$, $\val(h(y,y)) \geq 0$, and $\val(h(x,y)) \geq 0$.

First, let us compute the coefficient $\beta^1_0$.

\begin{lemma}\label{lemma4.3}
	$\beta^1_0=-\dfrac{1}{q(q^2-1)}$.
\end{lemma}
\begin{proof}
	In the proof of Theorem \ref{theorem3.14}, we have
	\begin{equation*}
	\CB \left(\begin{array}{c}
	\beta^1_0\\
	-\beta^1_0\\
	\delta_1
	\end{array} \middle) =(-q)^{-2} \middle( \begin{array}{c}
	-1\\0\\1
	\end{array} \right),
	\end{equation*}
	where
	\begin{equation*}
	\CB=\left( \begin{array}{ccc}
	(-q)^{-1} & -q & (-q)^{-2}\\
	1 & 1 & (-q)^{-2}\\
	-q & (-q)^{-1}&(-q)^{-2}
	\end{array}\right).
	\end{equation*}
	This implies that $\beta^1_0=-\dfrac{1}{q(q^2-1)}$.
\end{proof}

\begin{lemma}\label{lemma4.4}
	Assume that $\val(h(x,x)) \geq 0$, $\val(h(y,y)) \geq 0$, and $\val(h(x,y)) \geq 0$. Then we have
	\begin{equation*}
	\dfrac{1}{W_{1,1}(A_1,0)}(W'_{1,1}(B,0)-\beta^1_0 W_{1,0}(B,0))=\dfrac{1}{2}\det B+1.
	\end{equation*}
\end{lemma}

\begin{proof}
	First, we will prove that if $\val(h(x,x)) \geq 0$, $\val(h(y,y)) \geq 0$, and $\val(h(x,y)) \geq 0$, then
	\begin{equation}\label{eq4.2.1}
	W_{1,1}'(B,0)=q^{-4}\alpha'(\left(\begin{array}{ll}1&\\&\pi \end{array}\right),B).
	\end{equation}
	
	Let $M$ be an $O_{E_v}$-lattice with basis $\lbrace v_1, \dots, v_m \rbrace$ $(m=2+2r)$ such that the matrix of inner product $(\langle v_i,v_j \rangle)$ is $\pi A_1^{[r]}$. Let $L$ be an $O_{E_v}$-lattice with basis $\lbrace u_1,u_2 \rbrace$ such that the matrix of inner product $(\langle u_i,u_j \rangle)$ is $\pi B$.
	
	Then, for sufficiently large d, we have
	\begin{equation*}
	W_{1,1}(B,r)=q^{-4dm+4(d-1)} \vert J_d(L,M) \vert,
	\end{equation*}
	where
	\begin{equation*}
	\begin{array}{lll}
	J_d(L,M)&=\lbrace & \varphi \in \Hom_{O_{E_v}}(L, M/\pi^{d} M) \vert\\
	&&\langle \varphi(x), \varphi(y) \rangle\equiv \langle x, y \rangle \Mod \pi^d, \quad \forall x,y \in L,\\
	&&\varphi(u_1) \subset \pi M^{\vee} \rbrace.
	\end{array}
	\end{equation*}
	
	Now, assume that $\varphi(u_2) \in M \backslash \pi M^{\vee}$. Then we can write
	\begin{equation*}
	\varphi(u_2)=a_1 v_1 + a_2v_2+ \dots +a_mv_m,
	\end{equation*}
	where $a_2 \in O_{E_v}^{\times}$,$ a_1,a_3, \dots, a_m \in O_{E_v}$.
	Therefore,
	\begin{equation*}
	\langle\varphi(u_2), \varphi(u_2) \rangle=a_1a_1^*\pi+a_2a_2^*+ \dots+a_ma_m^*\pi,
	\end{equation*}
	and hence $\val(\langle \varphi(u_2),\varphi(u_2)\rangle)=0$.
	
	On the other hand, $\langle u_2,u_2 \rangle=(\pi B)_{22}$. By the assumption $\val(h(x,x)) \geq 0$, $\val(h(y,y)) \geq 0$, and $\val(h(x,y)) \geq 0$, this implies that $\val( \langle u_2,u_2 \rangle) \geq 1$.
	
	From these, we can conclude that if $\varphi(u_2) \in M \backslash \pi M^{\vee}$, then we have $\langle \varphi(u_2),\varphi(u_2)\rangle \not\equiv \langle u_2, u_2 \rangle$, and hence $\varphi \not\in J_d(L,M)$.
	
	Therefore, we have
	\begin{equation*}
\begin{array}{ll}
\vert J_d(L,M)\vert&\\
=\quad \vert\lbrace \varphi:L \rightarrow M  \text{ }(\Mod \pi^d) &\vert \langle \varphi(x), \varphi(y) \rangle \equiv \langle x,y\rangle \text{ } (\Mod \pi^d)\\
& \varphi(L) \subset \pi M^{\vee} \rbrace\vert\\
=\quad \vert\lbrace \varphi:L \rightarrow \pi M^{\vee}  \text{ }(\Mod \pi^d (\pi M^{\vee})) &\vert \langle \varphi(x), \varphi(y) \rangle \equiv \langle x,y\rangle \text{ } (\Mod \pi^d)\rbrace\vert\\
\times \vert M : \pi M^{\vee} \vert^{-2}&\\
=q^{-4}\vert I_d(L,\pi M^{\vee}) \vert&.
\end{array}
\end{equation*}

This implies that
\begin{equation*}
W_{1,1}(B,r)=q^{-4dm+4(d-1)-4} \vert I_d(L,\pi M^{\vee}) \vert.
\end{equation*}

Also, we have
\begin{equation*}\
\begin{array}{ll}
(q^{-d})^{2(2m-2)} \vert I_d(L, \pi M^{\vee}) \vert&= \alpha ( \left( \begin{array}{lll}
\pi &&\\
& \pi^2 &\\
& & \pi 1_{2r} \end{array} \right),\pi B )\\
&\\
&=q^4 \alpha(\left( \begin{array}{lll} 1 & &\\ &\pi&\\&&1_{2r} \end{array}\right),B).
\end{array}
\end{equation*}

These two imply that
\begin{equation*}
W_{1,1}(B,r)=q^{-4}\alpha(\left( \begin{array}{lll} 1 & &\\ &\pi&\\&&1_{2r} \end{array}\right),B).
\end{equation*}

This proves the claim \eqref{eq4.2.1}
\begin{equation*}
W_{1,1}'(B,0)=q^{-4}\alpha'(\left( \begin{array}{ll}1&\\&\pi\end{array}\right),B).
\end{equation*}

Also, it is easy to see that $W_{1,0}(B,0)=\alpha(1_2,B)$.
	
Therefore, by Lemma \ref{lemma4.1} and Lemma \ref{lemma4.3}, we have
\begin{equation*}
\begin{array}{l}
\dfrac{1}{W_{1,1}(A_1,0)}(W'_{1,1}(B,0)-\beta^1_0 W_{1,0}(B,0))\\
=\dfrac{q^5}{(q+1)^2}(\dfrac{1}{q^4} \alpha'(\left( \begin{array}{ll}1&\\&\pi\end{array}\right),B)+\dfrac{1}{q(q^2-1)}\alpha(1_2,B))\\
\\
=\dfrac{q}{(q+1)^2}(\alpha'(\left( \begin{array}{ll}1&\\&\pi\end{array}\right),B)+\dfrac{q^3}{q^2-1}\alpha(1_2,B)).
\end{array}
\end{equation*}

From \cite[Corollary 3.6]{San}, we know that if $B$ is $GL_2(O_{E_v})$-equivalent to $\diag(\pi^a,\pi^b)$, where $a \geq b \geq 0$ and $a+b$ is even, then
\begin{equation*}
\alpha'(\left( \begin{array}{ll}1&\\&\pi\end{array}\right),B)=\dfrac{1}{2}(a+b)\dfrac{(q+1)^2}{q}-\dfrac{(q+1)^2}{q(q^2-1)}(q^{b+2}-q^2-q+1),
\end{equation*}
and
\begin{equation*}
\alpha(1_2,B)=\dfrac{(q+1)^2}{q^3}(q^{b+1}-1).
\end{equation*}

Therefore,
\begin{equation*}
\begin{array}{l}
\dfrac{1}{W_{1,1}(A_1,0)}(W'_{1,1}(B,0)-\beta^1_0 W_{1,0}(B,0))\\
=\dfrac{q}{(q+1)^2}(\alpha'(\left( \begin{array}{ll}1&\\&\pi\end{array}\right),B)+\dfrac{q^3}{q^2-1}\alpha(1_2,B))\\
\\
=\dfrac{1}{2}(a+b)+1\\\\
=\dfrac{1}{2}\det B+1.
\end{array}
\end{equation*}
This finishes the proof of the lemma.
\end{proof}

Now, it remains to compute the intersection numbers of special cycles $\CZ(x)$ and $\CY(y)$. For this, we need to use the results in \cite{San}, and hence we will use the notation in loc. cit.

\begin{definition}(\cite[Definition 2.8]{San})
	
	\begin{enumerate}
		\item For an $O_{E_v}$-lattice $\Lambda$ in $\BV$, we call $\Lambda$ a vertex lattice if $\pi \Lambda^{\vee} \subset \Lambda \subset \Lambda^{\vee}$.
		
		\item Let $x \in \BV$ with $m_x:=\val(h(x,x)) \geq -1$, and let $t_x:=[\dfrac{m_x+1}{2}]$. Then $\val(h(\pi^{-t_x}x,\pi^{-t_x}x))$ is either $0$ or $-1$. By \cite[Lemma3.8]{San2}, there is a unique vertex lattice $\Lambda_x$ such that $x \in \Lambda_x \backslash \pi \Lambda_x$ and $\Lambda_x=\Lambda_x^{\vee}$ (resp. $\Lambda_x=\pi \Lambda_x^{\vee}$) if $m_x$ is even (resp. odd). We call this lattice $\Lambda_x$ the central lattice for $x$.
	\end{enumerate}
	
\end{definition}

\begin{definition}(cf. \cite[remark 2.10]{San})
	\begin{enumerate}
		\item For two vertex lattices $\Lambda, \Lambda'$, we denote by $d(\Lambda, \Lambda')$ the distance between two vertex lattices in the Bruhat-Tits tree.
		
		\item For $m \geq 0$ and a vertex lattice $\Lambda$, we define
		\begin{equation*}
		\BB_m(\Lambda):=\lbrace \Lambda' \text{ vertex lattice } \vert d(\Lambda,\Lambda')\leq m \rbrace.
		\end{equation*}
		
		\item For $x \in \BV$, we define
		\begin{equation*}
		m(x,\Lambda)=\dfrac{1}{2} \left\lbrace \begin{array}{ll}
		m_x-d(\Lambda,\Lambda_x) &\text{ if }m_x \equiv d(\Lambda,\Lambda_x) \Mod 2\\
		m_x+1-d(\Lambda,\Lambda_x)&\text{ if }m_x \not\equiv d(\Lambda,\Lambda_x) \Mod 2.
		\end{array}\right.
		\end{equation*}
		
		\item For $y \in \BV$, we define
		\begin{equation*}
		m^{\vee}(y,\Lambda)=\dfrac{1}{2} \left\lbrace \begin{array}{ll}
		m_y+1-d(\Lambda,\Lambda_y) &\text{ if }m_y+1 \equiv d(\Lambda,\Lambda_y) \Mod 2\\
		m_y+2-d(\Lambda,\Lambda_y)&\text{ if }m_y+1 \not\equiv d(\Lambda,\Lambda_y) \Mod 2.
		\end{array}\right.
		\end{equation*}

	\end{enumerate}
	
\end{definition}

\begin{lemma}
	(cf. \cite[Lemma 2.12]{San}) Let $x,y \in \BV$ and assume that $\BB_{m_x}(\Lambda_x) \cap \BB_{m_y+1}(\Lambda_y) \neq \emptyset$. Let $\FG$ be the unique shortest path between $\Lambda_x$ and $\Lambda_y$. Let
	\begin{equation*}
	r:=\min(\dfrac{m_x+m_y-d(\Lambda_x,\Lambda_y)}{2},m_x,m_y+1).
	\end{equation*}
	
	\begin{enumerate}
		\item Assume that $m_x \leq m_y+1$ and $\BB_{m_x}(\Lambda_x) \not\subset \BB_{m_y+1}(\Lambda_y)$. Then the intersection $\BB_{m_x}(\Lambda_x) \cap \BB_{m_y+1}(\Lambda_y)$ is equal to $\BB_r(\Gamma_x) \cup \BB_r(\Gamma_y)$ where $\Gamma_x$ (resp. $\Gamma_y$) is the unique lattice in $\FG$ such that
	\begin{equation*}
	\begin{array}{l}
	d(\Lambda_x,\Gamma_x)=m_x-r-1 \text{ (resp. } d(\Lambda_x,\Gamma_y)=m_x-r \text{)}.
	\end{array}
	\end{equation*}
		
		\item Assume that $m_x \geq m_y+1$ and $\BB_{m_x}(\Lambda_x) \not\supset \BB_{m_y+1}(\Lambda_y)$. Then the intersection $\BB_{m_x}(\Lambda_x) \cap \BB_{m_y+1}(\Lambda_y)$ is equal to $\BB_r(\Gamma_x) \cup \BB_r(\Gamma_y)$ where $\Gamma_x$ (resp. $\Gamma_y$) is the unique lattice in $\FG$ such that
		\begin{equation*}
		\begin{array}{l}
		d(\Lambda_y,\Gamma_x)=m_y+1-r \text{ (resp. } d(\Lambda_y,\Gamma_y)=m_y-r \text{)}.
		\end{array}
		\end{equation*}
	\end{enumerate}
\end{lemma}
\begin{proof}
	As in \cite[Lemma 2.12]{San}, this follows from the fact that the vertex lattices can be identified with the vertices in a Bruhat-Tits tree.
\end{proof}

Now, we are ready to prove the following theorem.

\begin{theorem}\label{theorem4.8}
	Conjecture \ref{conjecture1} is true for $n=1$.
\end{theorem}
\begin{proof}
As we noted above, we may assume that $\val(h(x,x))$, $\val(h(x,y))$, and $\val(h(y,y)) \geq 0$. We have the following three cases.

\tb{Case 1.} $\BB_{m_y+1}(\Lambda_y) \subset \BB_{m_x}(\Lambda_x)$.

\tb{Case 2.} $\BB_{m_y+1}(\Lambda_y) \supset \BB_{m_x}(\Lambda_x)$.

\tb{Case 3.} $\BB_{m_y+1}(\Lambda_y) \not\subset \BB_{m_x}(\Lambda_x)$ and $\BB_{m_y+1}(\Lambda_y) \not\supset \BB_{m_x}(\Lambda_x)$.\\

For simplicity, we denote by $\langle \CZ(x), \CY(y) \rangle$ the intersection number of $\CZ(x)$ and $\CY(y)$.

\tb{Case 1.} From \cite[Theorem 3.14]{San2}, we have
\begin{equation*}
\begin{array}{l}
\CZ(y)=\CZ(y)^h+\mathlarger{\sum}_{\BB_{m_y}(\Lambda_x)}m(y,\Lambda) \BP_{\Lambda}=:\CZ(y)^h+\CZ(y)^v,\\
\CY(y)=\CY(y)^h+\mathlarger{\sum}_{\BB_{m_y+1}(\Lambda_y)}m^{\vee}(y,\Lambda)\BP_{\Lambda}=:\CY(y)^h+\CY(y)^v,
\end{array}
\end{equation*}
where $\CZ(y)^h$ (resp. $\CY(y)^h$) is a horizontal divisor of $\CZ(y)$ (resp. $\CY(y)$) and $\CZ(y)^v$ (resp. $\CY(y)^v)$ is a vertical divisor of $\CZ(y)$ (resp. $\CY(y)$). Here, the formula for $\CY(y)$ is from \cite[Theorem 3.14]{San2} and the fact that $\theta:\CN=\CN^1(1,1) \rightarrow \widehat{\CN}=\CN^1(1,1)$ gives $\theta(\CY(y))=\CZ(\lambda \circ y)$ and $\val(h'(\lambda \circ y, \lambda \circ y))=\val(h(y,y))+1$.

Also, note that by the proof of \cite[Proposition 2.15]{San} (and \cite[Proposition 3.10]{San2}) we have $\CZ(x)^h=\CZ(\pi x)^h$. This implies that $\CZ(x)^h=\CY(y)^h$ since $\CZ(x) \subset \CY(y) \subset \CZ(\pi x)$ from the definition of $\CZ$ and $\CY$.

Therefore, we have
\begin{equation}\label{eq.4.2.2}
\CY(y)-\CZ(y)=\mathlarger{\sum}_{\substack{\Lambda \in \BB_{m_y}(\Lambda_y)\\ m_y \equiv d(\Lambda,\Lambda_y) \Mod 2}}\BP_{\Lambda}.
\end{equation}

By \cite[Lemma 2.11]{San}, we have
\begin{equation}\label{eq.4.2.3}
\langle \BP_{\Lambda}, \CZ(x) \rangle=\left\lbrace \begin{array}{ll}
	1, & \text{ if } \Lambda \in \BB_{m_x}(\Lambda_x),\text{ } d(\Lambda,\Lambda_x) \equiv m_x \mod2\\
	-q & \text{ if } \Lambda \in \BB_{m_x}(\Lambda_x),\text{ } d(\Lambda,\Lambda_x) \not\equiv m_x \mod2\\
	0 & \text{ if } \Lambda \not\in \BB_{m_x}(\Lambda_x).
	\end{array}\right.
\end{equation}

Therefore, we have
\begin{equation*}
\begin{array}{ll}
\langle \CZ(x), \CY(y)-\CZ(y) \rangle&=\mathlarger{\sum}_{\substack{\Lambda \in \BB_{m_y}(\Lambda_y)\\ m_y \equiv d(\Lambda,\Lambda_y) \Mod 2}} \langle\CZ(x),\BP_{\Lambda} \rangle\\
&=1+q(q+1)+ \dots +q^{m_y-1}(p+1)\\
&=1+\dfrac{q(q^{m_y}-1)}{q-1}.
\end{array}
\end{equation*}

Here we used the fact that $m_x \equiv d(\Lambda, \Lambda_x) (\Mod 2)$ if and only if $m_y \equiv d(\Lambda, \Lambda_y) (\Mod 2)$.

Now, by \cite[Corollary 2.17]{San} and \cite[Lemma 2.16]{San}, we have
\begin{equation*}
\langle \CZ(x),\CZ(y) \rangle=\dfrac{1}{2}\det B-\dfrac{q(q^{m_y}-1)}{q-1}.
\end{equation*}

Combining the above two equations, we have
\begin{equation*}
\langle \CZ(x),\CY(y) \rangle=\dfrac{1}{2}\det B+1.
\end{equation*}

Therefore, by Lemma \ref{lemma4.4}, Conjecture \ref{conjecture1} holds in the Case 1.\\

\tb{Case 2.} This case is similar to the Case 1. From \eqref{eq.4.2.2}, we have
\begin{equation*}
\CY(y)-\CZ(y)=\mathlarger{\sum}_{\substack{\Lambda \in \BB_{m_y}(\Lambda_y)\\ m_y \equiv d(\Lambda,\Lambda_y) \Mod 2}}\BP_{\Lambda}.
\end{equation*}
Also, by \eqref{eq.4.2.3}, $\langle \BP_{\Lambda}, \CZ(x) \rangle=0$ if $\Lambda \not\in \BB_{m_x}(\Lambda_x)$. Therefore,

\begin{equation*}
\begin{array}{ll}
\langle \CZ(x),\CY(y)-\CZ(y) \rangle&=\mathlarger{\sum}_{\substack{\Lambda \in \BB_{m_y}(\Lambda_y)\\ m_y \equiv d(\Lambda,\Lambda_y) \Mod 2}}\langle\CZ(x),\BP_{\Lambda} \rangle\\
&=\mathlarger{\sum}_{\substack{\Lambda \in \BB_{m_x}(\Lambda_x) \cap \BB_{m_y}(\Lambda_y)\\ m_x \equiv d(\Lambda,\Lambda_x) \Mod 2}}\langle\CZ(x),\BP_{\Lambda} \rangle
\end{array}
\end{equation*}

Note that $\BB_{m_y+1}(\Lambda_y) \supset \BB_{m_x}(\Lambda_x)$ implies that
\begin{equation*}
m_y+1 \geq m_x+d(\Lambda_x,\Lambda_y).
\end{equation*}

The equality cannot hold since $m_y \equiv m_x+d(\Lambda_x,\Lambda_y) \Mod 2$. Therefore, $m_y \geq m_x+d(\Lambda_x,\Lambda_y)$ and hence $\BB_{m_x}(\Lambda_x) \subset \BB_{m_y}(\Lambda_y)$.

Therefore, we have
\begin{equation*}
\begin{array}{ll}
\langle \CZ(x),\CY(y)-\CZ(y) \rangle&=\mathlarger{\sum}_{\substack{\Lambda \in \BB_{m_x}(\Lambda_x) \\ m_x \equiv d(\Lambda,\Lambda_x) \Mod 2}}\langle\CZ(x),\BP_{\Lambda} \rangle\\
&=1+q(q+1)+ \dots +q^{m_x-1}(p+1)\\
&=1+\dfrac{q(q^{m_x}-1)}{q-1}.
\end{array}
\end{equation*}

Now, by \cite[Corollary 2.17]{San} and \cite[Lemma 2.16]{San}, we have
\begin{equation*}
\langle \CZ(x),\CZ(y) \rangle=\dfrac{1}{2}\det B-\dfrac{q(q^{m_x}-1)}{q-1}.
\end{equation*}

Combining the above two equations, we have
\begin{equation*}
\langle \CZ(x),\CY(y) \rangle=\dfrac{1}{2}\det B+1.
\end{equation*}

Therefore, by Lemma \ref{lemma4.4}, Conjecture \ref{conjecture1} holds in the Case 2.\\

\tb{Case 3.} We can divide the Case 3 into the following 4 subcases.

\tb{Case 3-1}. $m_y+1 < m_x$ and $r$ is even.

\tb{Case 3-2}. $m_y+1 < m_x$ and $r$ is odd.

\tb{Case 3-3}. $m_y+1 \geq m_x$ and $r$ is even.

\tb{Case 3-4}. $m_y+1 \geq m_x$ and $r$ is odd.

We will only prove the Case 3-1 since all other cases are similar.

\tb{Case 3-1} Note that
\begin{equation*}
\langle \CZ(x),\CY(y) \rangle=\langle \CZ(x)^h,\CY(y)^h \rangle+\langle \CZ(x),\CY(y)^v \rangle+\langle \CZ(x)^v,\CY(y)^h \rangle.
\end{equation*}
Therefore, it suffices to compute each term separately.

In Case 3, $\Lambda_x \neq \Lambda_y$. Therefore, by \cite[Proposition 2.15]{San}, we have $\langle \CZ(x)^h,\CY(y)^h \rangle=0$.

Also, by the proof of \cite[Lemma 2.11]{San}, it is easy to check that
\begin{equation*}
\langle \CZ(x)^v,\CY(y)^h \rangle=m(x,\Lambda_y).
\end{equation*}

Now, it remains to compute $\langle \CZ(x), \CY(y)^v \rangle$. As in the proof of \cite[Theorem 2.14]{San}, we denote by $\Gamma_y^{(k)}$ the vertex in $\FG$ such that $d(\Lambda_y,\Gamma_y^{(k)})=d(\Lambda_y,\Gamma_y)-k$ and $d(\Gamma_y,\Gamma_y^{(k)})=k$.

For $k \geq 0$, we define $\FF(k)$ as the set of lattices $\Lambda$ in $\BB_{m_x}(\Lambda_x) \cap \BB_{m_y+1}(\Lambda_y)$ such that the unique shortest path from $\Lambda$ to $\Lambda_y$ first meets $\lbrace \Gamma_x, \Gamma_y,\Gamma_y^{(1)}, \dots \rbrace$ at $\Gamma_y^{(k)}$.

Also, we define $\FF(-1)$ as the set of lattices $\Lambda$ in $\BB_{m_x}(\Lambda_x) \cap \BB_{m_y+1}(\Lambda_y)$ such that the unique shortest path from $\Lambda$ to $\Lambda_y$ first meets $\lbrace \Gamma_x, \Gamma_y,\Gamma_y^{(1)}, \dots \rbrace$ at $\Gamma_x$.

Let $\ScF(k)=\mathlarger{\sum}_{\Lambda \in \FF(k)}m^{\vee}(y,\Lambda) \langle \BP_{\Lambda}, \CZ(x) \rangle$.

Then $\langle \CZ(x),\CY(y)^v \rangle=\mathlarger{\sum}_{-1 \leq k} \ScF(k)$. Therefore, it suffices to compute $\ScF(k)$ for $-1 \leq k$.

Now, we have to divide Case 3-1 into two more subcases.

\tb{Case 3-1-1} $m_y-r \leq r$.

\tb{Case 3-1-2} $m_y-r > r$.\\

In \tb{Case 3-1-1}, $\ScF(m_y-r)$ is the last term in the above sum and hence
\begin{equation*}
\langle \CZ(x),\CY(y)^v \rangle=\mathlarger{\sum}_{-1 \leq k \leq m_y-r} \ScF(k).
\end{equation*}

$\bullet$ For $\Lambda \in \FF(-1)$, we have
\begin{equation*}
\begin{array}{ll}
d(\Lambda,\Lambda_y)&=d(\Lambda,\Gamma_x)+d(\Gamma_x,\Lambda_y)\\
&=d(\Lambda,\Gamma_x)+m_y+1-r.
\end{array}
\end{equation*}

Let $t:=d(\Lambda,\Gamma_x)$.

If $t$ is even, then $d(\Lambda,\Lambda_y) \equiv m_y+1$ $(\Mod 2)$, and
if $t$ is odd, then $d(\Lambda,\Lambda_y) \not\equiv m_y+1$ $(\Mod 2)$.

Therefore, we have
\begin{equation*}
m^{\vee}(y,\Lambda)=\dfrac{1}{2}\left\lbrace\begin{array}{ll}
r-t & \text{ if } t \text{ is even,}\\
r+1-t & \text{ if } t \text{ is odd.}\\
\end{array}\right.
\end{equation*}

This implies that
\begin{equation*}
\begin{array}{ll}
\ScF(-1)&=\dfrac{r}{2} \sum_{\substack{\Lambda \in \FF(-1)\\d(\Lambda, \Gamma_x)=0}}\langle \BP_{\Lambda},\CZ(x) \rangle+ \dfrac{r}{2} \sum_{\substack{\Lambda \in \FF(-1)\\d(\Lambda, \Gamma_x)=1}}\langle \BP_{\Lambda},\CZ(x) \rangle\\
&+(\dfrac{r}{2}-1) \sum_{\substack{\Lambda \in \FF(-1)\\d(\Lambda, \Gamma_x)=2}}\langle \BP_{\Lambda},\CZ(x) \rangle+\dots+0\cdot \sum_{\substack{\Lambda \in \FF(-1)\\d(\Lambda, \Gamma_x)=r}}\langle \BP_{\Lambda},\CZ(x) \rangle.
\end{array}
\end{equation*}

Note that \begin{equation*}
\begin{array}{ll}
d(\Lambda,\Lambda_x)&\equiv d(\Lambda,\Gamma_x)+d(\Lambda_x,\Gamma_x) (\Mod 2)\\
&\equiv t+d(\Lambda_x,\Gamma_x) (\Mod 2),
\end{array}
\end{equation*}
and
\begin{equation*}
\begin{array}{ll}
d(\Lambda_x,\Gamma_x)&\equiv d(\Lambda_y,\Lambda_x)-d(\Lambda_y,\Gamma_x) (\Mod 2)\\
&\equiv d(\Lambda_y,\Lambda_x)-(m_y+1-r) (\Mod 2)\\
&\equiv d(\Lambda_y,\Lambda_x)-(m_y+1) (\Mod 2)\\
&\equiv m_x-1 (\Mod 2).
\end{array}
\end{equation*}

Now, by \cite[Lemma 2.11]{San},
\begin{equation*}
\langle \BP_{\Lambda}, \CZ(x) \rangle=\left\lbrace \begin{array}{ll}
	-p & \text{ if } \Lambda \in \FF(-1),t=\text{even},\\
	1 & \text{ if } \Lambda \in \FF(-1),t=\text{odd},
	\end{array}\right.
\end{equation*}

Therefore,

\begin{equation*}
\begin{array}{ll}
\ScF(-1)&=\dfrac{r}{2} (-p)+p\cdot\dfrac{r}{2}\cdot 1+p^2(\dfrac{r}{2}-1)(-p)+\dots\\&\\
&=0.
\end{array}
\end{equation*}

$\bullet$ For $\Lambda \in \FF(0)$, we have

\begin{equation*}
\begin{array}{ll}
d(\Lambda,\Lambda_y)&=d(\Lambda,\Gamma_y)+d(\Gamma_y,\Lambda_y)\\
&=d(\Lambda,\Gamma_y)+m_y-r.
\end{array}
\end{equation*}

Let $t:=d(\Lambda,\Gamma_y)$.

If $t$ is even, then $d(\Lambda,\Lambda_y) \not\equiv m_y+1$ $(\Mod 2)$, and
if $t$ is odd, then $d(\Lambda,\Lambda_y) \equiv m_y+1$ $(\Mod 2)$.

Therefore, we have
\begin{equation*}
m^{\vee}(y,\Lambda)=\dfrac{1}{2}\left\lbrace\begin{array}{ll}
r+2-t & \text{ if } t \text{ is even,}\\
r+1-t & \text{ if } t \text{ is odd.}\\
\end{array}\right.
\end{equation*}

This implies that
\begin{equation*}
\begin{array}{ll}
\ScF(0)&=(\dfrac{r}{2}+1) \sum_{\substack{\Lambda \in \FF(0)\\d(\Lambda, \Gamma_y)=0}}\langle \BP_{\Lambda},\CZ(x) \rangle+ \dfrac{r}{2} \sum_{\substack{\Lambda \in \FF(0)\\d(\Lambda, \Gamma_y)=1}}\langle \BP_{\Lambda},\CZ(x) \rangle\\
&+\dfrac{r}{2} \sum_{\substack{\Lambda \in \FF(0)\\d(\Lambda, \Gamma_y)=2}}\langle \BP_{\Lambda},\CZ(x) \rangle+\dots+1\cdot \sum_{\substack{\Lambda \in \FF(0)\\d(\Lambda, \Gamma_y)=r}}\langle \BP_{\Lambda},\CZ(x) \rangle.
\end{array}
\end{equation*}

Note that \begin{equation*}
\begin{array}{ll}
d(\Lambda,\Lambda_x)&\equiv d(\Lambda,\Gamma_y)+d(\Lambda_x,\Gamma_y) (\Mod 2)\\
&\equiv t+d(\Lambda_x,\Gamma_y) (\Mod 2),
\end{array}
\end{equation*}
and
\begin{equation*}
\begin{array}{ll}
d(\Lambda_x,\Gamma_y)&\equiv d(\Lambda_y,\Lambda_x)-d(\Lambda_y,\Gamma_y) (\Mod 2)\\
&\equiv d(\Lambda_y,\Lambda_x)-(m_y-r) (\Mod 2)\\
&\equiv d(\Lambda_y,\Lambda_x)-m_y (\Mod 2)\\
&\equiv m_x (\Mod 2).
\end{array}
\end{equation*}

Now, by \cite[Lemma 2.11]{San},
\begin{equation*}
\langle \BP_{\Lambda}, \CZ(x) \rangle=\left\lbrace \begin{array}{ll}
1 & \text{ if } \Lambda \in \FF(-1),t=\text{even},\\
-p & \text{ if } \Lambda \in \FF(-1),t=\text{odd},
\end{array}\right.
\end{equation*}

Therefore,

\begin{equation*}
\begin{array}{ll}
\ScF(0)&=(\dfrac{r}{2}+1)+(p-1)\cdot\dfrac{r}{2}\cdot (-p)+p(p-1)(\dfrac{r}{2})\cdot 1+\dots\\&\\
&=\dfrac{r}{2}+1.
\end{array}
\end{equation*}

$\bullet$ For $\Lambda \in \FF(k)$, $k=$odd, $k\neq m_y-r$, we have

\begin{equation*}
\begin{array}{ll}
d(\Lambda,\Lambda_y)&=d(\Lambda,\Gamma_y^{(k)})+d(\Gamma_y^{(k)},\Lambda_y)\\
&=d(\Lambda,\Gamma_y^{(k)})+m_y-r-k.
\end{array}
\end{equation*}

Let $t:=d(\Lambda,\Gamma_y^{(k)})$.

If $t$ is even, then $d(\Lambda,\Lambda_y) \equiv m_y+1$ $(\Mod 2)$, and
if $t$ is odd, then $d(\Lambda,\Lambda_y) \not\equiv m_y+1$ $(\Mod 2)$.

Therefore, we have
\begin{equation*}
m^{\vee}(y,\Lambda)=\dfrac{1}{2}\left\lbrace\begin{array}{ll}
r+k+1-t & \text{ if } t \text{ is even,}\\
r+k+2-t & \text{ if } t \text{ is odd.}\\
\end{array}\right.
\end{equation*}

This implies that
\begin{equation*}
\begin{array}{ll}
\ScF(k)&=(\dfrac{r+k+1}{2}) \sum_{\substack{\Lambda \in \FF(k)\\d(\Lambda, \Gamma_y^{(k)})=0}}\langle \BP_{\Lambda},\CZ(x) \rangle\\
&+ \dfrac{r+k+1}{2} \sum_{\substack{\Lambda \in \FF(k)\\d(\Lambda, \Gamma_y^{(k)})=1}}\langle \BP_{\Lambda},\CZ(x) \rangle\\
&+\dfrac{r+k-1}{2} \sum_{\substack{\Lambda \in \FF(k)\\d(\Lambda, \Gamma_y^{(k)})=2}}\langle \BP_{\Lambda},\CZ(x) \rangle\\
&+\dots+\dfrac{2k+2}{2} \sum_{\substack{\Lambda \in \FF(k)\\d(\Lambda, \Gamma_y^{(k)})=r-k}}\langle \BP_{\Lambda},\CZ(x) \rangle.
\end{array}
\end{equation*}

Note that \begin{equation*}
\begin{array}{ll}
d(\Lambda,\Lambda_x)&\equiv d(\Lambda,\Gamma_y^{(k)})+d(\Lambda_x,\Gamma_y^{(k)}) (\Mod 2)\\
&\equiv t+d(\Lambda_x,\Gamma_y^{(k)}) (\Mod 2),
\end{array}
\end{equation*}
and
\begin{equation*}
\begin{array}{ll}
d(\Lambda_x,\Gamma_y^{(k)})&\equiv d(\Lambda_y,\Lambda_x)-d(\Lambda_y,\Gamma_y^{(k)}) (\Mod 2)\\
&\equiv d(\Lambda_y,\Lambda_x)-(m_y-r-k) (\Mod 2)\\
&\equiv d(\Lambda_y,\Lambda_x)-m_y+1 (\Mod 2)\\
&\equiv m_x+1 (\Mod 2).
\end{array}
\end{equation*}

Now, by \cite[Lemma 2.11]{San},
\begin{equation*}
\langle \BP_{\Lambda}, \CZ(x) \rangle=\left\lbrace \begin{array}{ll}
-p & \text{ if } \Lambda \in \FF(k),t=\text{even},\\
1 & \text{ if } \Lambda \in \FF(k),t=\text{odd},
\end{array}\right.
\end{equation*}

Therefore,

\begin{equation*}
\begin{array}{ll}
\ScF(k)&=(\dfrac{r+k+1}{2})(-p)+(p-1)\cdot\dfrac{r+k+1}{2}\cdot 1\\&+p(p-1)(\dfrac{r+k-1}{2})\cdot (-p)+\dots\\
&=-\dfrac{r+k+1}{2}.
\end{array}
\end{equation*}

$\bullet$ For $\Lambda \in \FF(k)$, $k=$even, $k\neq m_y-r$, we have

\begin{equation*}
\begin{array}{ll}
d(\Lambda,\Lambda_y)&=d(\Lambda,\Gamma_y^{(k)})+d(\Gamma_y^{(k)},\Lambda_y)\\
&=d(\Lambda,\Gamma_y^{(k)})+m_y-r-k.
\end{array}
\end{equation*}

Let $t:=d(\Lambda,\Gamma_y^{(k)})$.

If $t$ is even, then $d(\Lambda,\Lambda_y) \not\equiv m_y+1$ $(\Mod 2)$, and
if $t$ is odd, then $d(\Lambda,\Lambda_y) \equiv m_y+1$ $(\Mod 2)$.

Therefore, we have
\begin{equation*}
m^{\vee}(y,\Lambda)=\dfrac{1}{2}\left\lbrace\begin{array}{ll}
r+k+2-t & \text{ if } t \text{ is even,}\\
r+k+1-t & \text{ if } t \text{ is odd.}\\
\end{array}\right.
\end{equation*}

This implies that
\begin{equation*}
\begin{array}{ll}
\ScF(k)&=(\dfrac{r+k+2}{2}) \sum_{\substack{\Lambda \in \FF(k)\\d(\Lambda, \Gamma_y^{(k)})=0}}\langle \BP_{\Lambda},\CZ(x) \rangle\\
&+ \dfrac{r+k}{2} \sum_{\substack{\Lambda \in \FF(k)\\d(\Lambda, \Gamma_y^{(k)})=1}}\langle \BP_{\Lambda},\CZ(x) \rangle\\
&+\dfrac{r+k}{2} \sum_{\substack{\Lambda \in \FF(k)\\d(\Lambda, \Gamma_y^{(k)})=2}}\langle \BP_{\Lambda},\CZ(x) \rangle\\
&+\dots+\dfrac{2k+2}{2} \sum_{\substack{\Lambda \in \FF(k)\\d(\Lambda, \Gamma_y^{(k)})=r-k}}\langle \BP_{\Lambda},\CZ(x) \rangle.
\end{array}
\end{equation*}

Note that \begin{equation*}
\begin{array}{ll}
d(\Lambda,\Lambda_x)&\equiv d(\Lambda,\Gamma_y^{(k)})+d(\Lambda_x,\Gamma_y^{(k)}) (\Mod 2)\\
&\equiv  t+d(\Lambda_x,\Gamma_y^{(k)}) (\Mod 2),
\end{array}
\end{equation*}
and
\begin{equation*}
\begin{array}{ll}
d(\Lambda_x,\Gamma_y^{(k)})&\equiv d(\Lambda_y,\Lambda_x)-d(\Lambda_y,\Gamma_y^{(k)}) (\Mod 2)\\
&\equiv d(\Lambda_y,\Lambda_x)-(m_y-r-k) (\Mod 2)\\
&\equiv d(\Lambda_y,\Lambda_x)-m_y (\Mod 2)\\
&\equiv m_x (\Mod 2).
\end{array}
\end{equation*}

Now, by \cite[Lemma 2.11]{San},
\begin{equation*}
\langle \BP_{\Lambda}, \CZ(x) \rangle=\left\lbrace \begin{array}{ll}
1 & \text{ if } \Lambda \in \FF(k),t=\text{even},\\
-p & \text{ if } \Lambda \in \FF(k),t=\text{odd},
\end{array}\right.
\end{equation*}

Therefore,

\begin{equation*}
\begin{array}{ll}
\ScF(k)&=(\dfrac{r+k+2}{2})+(p-1)\cdot\dfrac{r+k}{2}\cdot (-p)\\&+p(p-1)(\dfrac{r+k}{2})\cdot 1+\dots\\
&=\dfrac{r+k+2}{2}.
\end{array}
\end{equation*}

$\bullet$ Finally, we need to compute $\ScF(m_y-r)$.

Note that $\Gamma_y^{(m_y-r)}=\Lambda_y$. For $\Lambda \in \FF(m_y-r)$, let $t=d(\Lambda,\Lambda_y)$.

Assume further that $m_y$ is even. This implies that if $t$ is even, then $t=d(\Lambda,\Lambda_y) \not\equiv m_y+1$ $(\Mod 2)$, and
if $t$ is odd, then $t=d(\Lambda,\Lambda_y) \equiv m_y+1$ $(\Mod 2)$.

Therefore, we have
\begin{equation*}
m^{\vee}(y,\Lambda)=\dfrac{1}{2}\left\lbrace\begin{array}{ll}
m_y+2-t & \text{ if } t \text{ is even,}\\
m_y+1-t & \text{ if } t \text{ is odd.}\\
\end{array}\right.
\end{equation*}

This implies that
\begin{equation*}
\begin{array}{ll}
\ScF(m_y-r)&=(\dfrac{m_y+2}{2}) \sum_{\substack{\Lambda \in \FF(m_y-r)\\d(\Lambda, \Gamma_y^{(k)})=0}}\langle \BP_{\Lambda},\CZ(x) \rangle\\
&+ \dfrac{m_y}{2} \sum_{\substack{\Lambda \in \FF(m_y-r)\\d(\Lambda, \Gamma_y^{(k)})=1}}\langle \BP_{\Lambda},\CZ(x) \rangle\\
&+\dfrac{m_y}{2} \sum_{\substack{\Lambda \in \FF(m_y-r)\\d(\Lambda, \Gamma_y^{(k)})=2}}\langle \BP_{\Lambda},\CZ(x) \rangle\\
&+\dots+\dfrac{2m_y-2r+2}{2} \sum_{\substack{\Lambda \in \FF(m_y-r)\\d(\Lambda, \Gamma_y^{(k)})=2r-m_y}}\langle \BP_{\Lambda},\CZ(x) \rangle.
\end{array}
\end{equation*}

Note that
\begin{equation*}
d(\Lambda_x,\Lambda_y)\equiv m_x+m_y \equiv m_x (\Mod 2).
\end{equation*}

Now, by \cite[Lemma 2.11]{San},
\begin{equation*}
\langle \BP_{\Lambda}, \CZ(x) \rangle=\left\lbrace \begin{array}{ll}
1 & \text{ if } \Lambda \in \FF(m_y-r),t=\text{even},\\
-p & \text{ if } \Lambda \in \FF(m_y-r),t=\text{odd},
\end{array}\right.
\end{equation*}

Therefore,

\begin{equation*}
\begin{array}{ll}
\ScF(m_y-r)&=(\dfrac{m_y+2}{2})+p\cdot\dfrac{m_y}{2}\cdot (-p)
+p^2(\dfrac{m_y}{2})\cdot 1+\dots\\&\\
&=\dfrac{m_y+2}{2}.
\end{array}
\end{equation*}

Now, assume that $m_y$ is odd. This implies that
if $t$ is even, then $t=d(\Lambda,\Lambda_y) \equiv m_y+1$ $(\Mod 2)$, and
if $t$ is odd, then $t=d(\Lambda,\Lambda_y) \not\equiv m_y+1$ $(\Mod 2)$.

Therefore, we have
\begin{equation*}
m^{\vee}(y,\Lambda)=\dfrac{1}{2}\left\lbrace\begin{array}{ll}
m_y+1-t & \text{ if } t \text{ is even,}\\
m_y+2-t & \text{ if } t \text{ is odd.}\\
\end{array}\right.
\end{equation*}

This implies that
\begin{equation*}
\begin{array}{ll}
\ScF(m_y-r)&=(\dfrac{m_y+1}{2}) \sum_{\substack{\Lambda \in \FF(m_y-r)\\d(\Lambda, \Gamma_y^{(k)})=0}}\langle \BP_{\Lambda},\CZ(x) \rangle\\
&+ \dfrac{m_y+1}{2} \sum_{\substack{\Lambda \in \FF(m_y-r)\\d(\Lambda, \Gamma_y^{(k)})=1}}\langle \BP_{\Lambda},\CZ(x) \rangle\\
&+\dfrac{m_y-1}{2} \sum_{\substack{\Lambda \in \FF(m_y-r)\\d(\Lambda, \Gamma_y^{(k)})=2}}\langle \BP_{\Lambda},\CZ(x) \rangle\\
&+\dots+\dfrac{2m_y-2r+2}{2} \sum_{\substack{\Lambda \in \FF(m_y-r)\\d(\Lambda, \Gamma_y^{(k)})=2r-m_y}}\langle \BP_{\Lambda},\CZ(x) \rangle.
\end{array}
\end{equation*}

Note that
\begin{equation*}
d(\Lambda_x,\Lambda_y)\equiv m_x+m_y \not\equiv m_x (\Mod 2).
\end{equation*}

Now, by \cite[Lemma 2.11]{San},
\begin{equation*}
\langle \BP_{\Lambda}, \CZ(x) \rangle=\left\lbrace \begin{array}{ll}
-p & \text{ if } \Lambda \in \FF(m_y-r),t=\text{even},\\
1 & \text{ if } \Lambda \in \FF(m_y-r),t=\text{odd},
\end{array}\right.
\end{equation*}

Therefore,

\begin{equation*}
\begin{array}{ll}
\ScF(m_y-r)&=(\dfrac{m_y+1}{2})(-p)+p\cdot\dfrac{m_y}{2}\cdot 1
+p^2(\dfrac{m_y}{2})\cdot (-p)+\dots\\&\\
&=0.
\end{array}
\end{equation*}
\\\\

By combining above computations, we have
\begin{equation*}
\begin{array}{l}
\ScF(-1)+ \ScF(0)+\dots+\ScF(m_y-r)\\\\
=\left\lbrace \begin{array}{ll}
0+(\dfrac{r}{2}+1)-\dfrac{r+2}{2}+\dfrac{r+4}{2}+\dots+\dfrac{m_y+2}{2} & \text{ if } m_y \text{ is even}\\&\\
0+(\dfrac{r}{2}+1)-\dfrac{r+2}{2}+\dfrac{r+4}{2}+\dots+\dfrac{m_y+1}{2}+0 & \text{ if } m_y \text{ is odd}
\end{array}\right.\\\\
=\left\lbrace \begin{array}{ll}
\dfrac{m_y+2}{2} & \text{ if } m_y \text{ is even}\\&\\
\dfrac{m_y+1}{2} & \text{ if } m_y \text{ is odd}.
\end{array}\right.\\\\
=[\dfrac{m_y+2}{2}].

\end{array}
\end{equation*}

Therefore,
\begin{equation*}
\begin{array}{ll}
\langle \CZ(x),\CY(y)\rangle&=\langle \CZ^h(x),\CY(y)^h\rangle+\langle \CZ(x),\CY(y)^v \rangle+\langle \CZ(x)^v,\CY(y)^h\rangle\\
&=0+[\dfrac{m_y+2}{2}]+m(x,\Lambda_y).

\end{array}
\end{equation*}

Recall that

\begin{equation*}
m(x,\Lambda_y)=\dfrac{1}{2} \left\lbrace \begin{array}{ll}
m_x-d(\Lambda_x,\Lambda_y) &\text{ if }m_x \equiv d(\Lambda_x,\Lambda_y) \Mod 2\\
m_x+1-d(\Lambda_x,\Lambda_y)&\text{ if }m_x \not\equiv d(\Lambda_x,\Lambda_y) \Mod 2.
\end{array}\right.
\end{equation*}

Since $d(\Lambda_x,\Lambda_y) \equiv m_x+m_y (\Mod 2)$, we have
\begin{equation*}
m(x,\Lambda_y)=\dfrac{1}{2} \left\lbrace \begin{array}{ll}
m_x-d(\Lambda_x,\Lambda_y) &\text{ if }m_y \text{ is even} \\
m_x+1-d(\Lambda_x,\Lambda_y)&\text{ if }m_x \text{ is odd} .
\end{array}\right.
\end{equation*}

Also, by the condition that $\BB_{m_y+1}(\Lambda_y) \not\subset \BB_{m_x}(\Lambda_x)$, we have
\begin{equation*}
r=\dfrac{1}{2}(m_x+m_y-d(\Lambda_x,\Lambda_y)).
\end{equation*}

Therefore,
\begin{equation*}
\langle \CZ(x),\CY(y)\rangle=\dfrac{1}{2}(m_x+m_y+2-d(\Lambda_x,\Lambda_y))=r+1.
\end{equation*}

By \cite[Lemma 2.16]{San},
\begin{equation*}
\langle \CZ(x),\CY(y)\rangle=r+1=\dfrac{1}{2}\det B+1,
\end{equation*}
and hence by Lemma \ref{lemma4.4}, Conjecture \ref{conjecture1} holds in the Case 3-1-1.\\\\

In \tb{Case 3-1-2}, $\ScF(r)$ is the last term in the above sum and hence
\begin{equation*}
\langle \CZ(x),\CY(y)^v \rangle=\mathlarger{\sum}_{-1 \leq k \leq r} \ScF(k).
\end{equation*}

Also, it is easy to check that $\ScF(r)=r+1$.

As in the Case 3-1-1, we can show that
\begin{equation*}
\langle \CZ(x),\CY(y)^v \rangle=r+1.
\end{equation*}

Also, $m_y-r > r$ implies that $\Lambda_y \not\in \BB_{m_x}(\Lambda_x)$. Therefore,
\begin{equation*}
\langle \CZ(x)^v,\CY(y)^h \rangle=0.
\end{equation*}

This implies that
\begin{equation*}
\begin{array}{ll}
\langle \CZ(x),\CY(y)\rangle&=\langle \CZ^h(x),\CY(y)^h\rangle+\langle \CZ(x),\CY(y)^v \rangle+\langle \CZ(x)^v,\CY(y)^h\rangle\\
&=0+(r+1)+0=r+1.
\end{array}
\end{equation*}

By \cite[Lemma 2.16]{San},
\begin{equation*}
\langle \CZ(x),\CY(y)\rangle=r+1=\dfrac{1}{2}\det B+1,
\end{equation*}
and hence by Lemma \ref{lemma4.4}, Conjecture \ref{conjecture1} holds in the Case 3-1-2.

\end{proof}

\appendix
\section{Arithmetic intersection numbers of special cycles $\CZ(x_1), \dots, \CZ(x_{n+1})$ in $\CN^1(1,n)$}\label{appendix}

In this appendix, we will prove that our conjectures are compatible with \cite[Theorem 10.4.4]{LZ}. This theorem relates the intersection number of $\CZ(x_1), \dots, \CZ(x_{n+1})$ in $\CN^1(1,n)$ to representation densities. This intersection number can be regarded as the intersection number of
\begin{equation*}
\CZ(x_1), \dots, \CZ(x_n),\CZ(x_{n+1}),\CY(y_1),\dots,\CY(y_{n-1})
\end{equation*} 
in $\CN^n(1,2n-1)$, where
	\begin{displaymath}
B=\left(\begin{array}{cc} 
h(x_i,x_j) & h(x_i,y_l)\\
h(y_k,x_j)& h(y_k,y_l)
\end{array} 
\right)_{1\leq i,j \leq n+1, 1\leq k,l \leq n-1}=\left(\begin{array}{cc} 
B_1 &  \\
  & \dfrac{1}{\pi} 1_{n-1}
\end{array} 
\right),
\end{displaymath}
and $B_1 \in X_{n+1}(E_v)$.

Therefore, the Conjecture \ref{conjecture2} should be compatible with \cite[Theorem 10.4.4]{LZ}. To prove this compatibility, we will use the notation in Section \ref{section3.2}.

Let $L$ be an $O_{E_v}$-lattice with basis $\lbrace u_1, \dots, u_{2n} \rbrace$ such that the matrix of inner product $(\langle u_i, u_j \rangle)$ is $\pi B$. Let $L_1$ be the $O_{E_v}$-lattice of rank $n+1$ with basis $\lbrace u_1, \dots, u_{n+1} \rbrace$, $L_2$ be the $O_{E_v}$-lattice of rank $n-1$ with basis $\lbrace u_{n+1}, \dots, u_{2n} \rbrace$. Note that $L_2$ is unimodular with respect to $\langle \cdot, \cdot \rangle$.

Let $M$ be an $O_{E_v}$-lattice with basis $\lbrace v_1, \dots, v_m \rbrace$ ($m=2n+2r$) such that the matrix of inner product $(\langle v_i,v_j \rangle)$ is $\pi A^{[r]}_{n}.$

For $0 \leq i \leq n-1$, let $M_i$ be an $O_{E_v}$-lattice with basis $\lbrace v_{1,i}, \dots, v_{2n,i} \rbrace$ such that the matrix of inner product $(\langle v_{k,i},v_{l,i} \rangle)$ is $\pi A_{i}$.

We define the sets

\begin{equation*}
\begin{array}{lll}
J_d^1(L,M)&=\lbrace & \varphi \in \Hom_{O_{E_v}}(L, M/\pi^{d} M) \vert\\
&&\langle \varphi(x), \varphi(y) \rangle\equiv \langle x, y \rangle \Mod \pi^d, \quad \forall x,y \in L,\\
&&\varphi(L_1) \subset \pi M^{\vee} \rbrace.

\end{array}
\end{equation*}

and for $0 \leq i \leq n-1$,

\begin{equation*}
\begin{array}{lll}
J_d^1(L,M_i)&=\lbrace & \varphi \in \Hom_{O_{E_v}}(L, M_i/\pi^{d} M_i) \vert\\
&&\langle \varphi(x), \varphi(y) \rangle\equiv \langle x, y \rangle \Mod \pi^d, \quad \forall x,y \in L,\\
&&\varphi(L_1) \subset \pi M_i^{\vee} \rbrace.

\end{array}
\end{equation*}

Then we have

\begin{equation*}
W_{n-1,n}(B,r)=q^{-2dm(2n)}q^{(d-1)(2n)^2} \vert J_d^1(L,M) \vert,
\end{equation*}

and

\begin{equation*}
W_{n-1,i}(B,0)=q^{-2d(2n)^2}q^{(d-1)(2n)^2} \vert J_d^1(L,M_i) \vert,
\end{equation*}
for sufficiently large $d$.

Note that $L$ has the rank $n-1$ unimodular sublattice $L_2$. Also, any unimodular sublattice in $M_i$ has rank $\leq i$. Therefore, if $i <n-1$, there is no $\varphi \in \Hom_{O_{E_v}}(L,M_i/\pi^dM_i)$ such that $\langle \varphi(x), \varphi(y) \rangle \equiv \langle x,y \rangle \Mod \pi^d$ for sufficiently large $d$. This implies that $J_d(L,M_i)$ is an empty set for $i <n-1$. Therefore, $W_{n-1,i}(B,0)=0$ for $i <n-1$.

This implies that we only need to compute the constant $\beta_{n-1}^{n-1}$ in Conjecture \ref{conjecture2}. We have the following proposition.

\begin{proposition}\label{propositionA}
	$\beta_{n-1}^{n-1}=\dfrac{(1-(-q)^n)}{(-q)^{3n+1}(1-(-q))(1-(-q)^{-(n+1)})}$.
\end{proposition}
\begin{proof}
	By the proof of theorem \ref{theorem3.14}, we have
		\begin{equation*}
	\FB\left(\begin{array}{l}
	\beta_0^h\\
	\vdots\\
	\beta_{n-1}^h\\
	-\beta_{0}^{2n-h}\\
	\vdots\\
	-\beta_{n-1}^{2n-h}\\
	\delta_h
	\end{array}
	\middle)=(-q)^{-2n(2n-h)}\middle(\begin{array}{c}
	-(2n-h)\\
	\vdots\\
	-1\\
	0\\
	1\\
	\vdots\\
	h
	\end{array}\right),
	\end{equation*}
	where $0 \leq h \leq 2n$ (in this proposition, we only need $h=n-1$, but we will use this general notation for later use).
	Therefore, we need to find the inverse matrix of $(-q)^{2n(2n-h)}\FB$.
	Note that
		\begin{equation*}
	\begin{array}{l}
	\FB=\\\left(\begin{array}{ccccccc}
	m_{10}& \dots& m_{1(n-1)}& n_{10}& \dots& n_{1(n-1)} & m_{1n}\\
	\vdots&\vdots&\vdots&\vdots&\vdots&\vdots&\vdots\\
	m_{(2n+1)0}& \dots &m_{(2n+1)(n-1)} & n_{(2n+1)0} & \dots & n_{(2n+1)(n-1)}&m_{2n+1,n}
	\end{array}\right)
	\end{array}
	\end{equation*}
	and
		\begin{equation*}
	\begin{array}{l}
	m_{it}=(-q)^{(n-t)(i-(2n-h+1))-2t(2n-h)},\\
	n_{it}=q^{-(2n-h)^2+h^2}(-q)^{-(n-t)(i-(2n-h+1))-2th}.
	\end{array}
	\end{equation*}
	
	Therefore, we have
	\begin{equation}\label{A1}
	\begin{array}{l}
	(-q)^{2n(2n-h)}m_{it}=(-q)^{(n-t)(i+2n-h-1)},\\
	(-q)^{2n(2n-h)}n_{it}=(-q)^{-(n-t)(i-(2n+h+1))}.
	\end{array}
	\end{equation}
	 
	 We define
	 \begin{equation*}
	 \begin{array}{ll}
	 \alpha_{i,h}=(-q)^{(n+1-i)(2n-h)}, & 1 \leq i \leq n;\\
	 \alpha_{i,h}=(-q)^{(2n+1-i)(2n+h)}, & n+1 \leq i \leq 2n;\\
	 \alpha_{2n+1,h}=1,
	\end{array}
	 \end{equation*}
	 
	and
	
	\begin{equation*}
	\begin{array}{ll}
	x_i=(-q)^{n+1-i}, & 1 \leq i \leq n; \\
	x_i=(-q)^{i-2n-1}, & n+1 \leq i \leq 2n;\\
	x_{2n+1}=1.
	\end{array}
	\end{equation*}
	
	Also, we define a $2n+1$ $\times$ $2n+1$ matrix $\Fa_h=\diag(\alpha_{1,h}, \dots, \alpha_{2n+1,h})$ and a Vandermonde matrix 
	\begin{equation*}
	\FX=\left( \begin{array}{llll}
	1 & 1& \dots  & 1\\
	x_1 &x_2 &    &x_{2n+1}\\
	x_1^2&x_2^2 &  & x_{2n+1}^2\\
	\vdots & & \ddots & \vdots\\
	x_1^{2n} & x_2^{2n} & \dots & x_{2n+1}^{2n}
	\end{array}\right).
	\end{equation*}
	
	Then, we have $(-q)^{2n(2n-h)}\FB=\FX \Fa_h$
	
	Therefore, we have
	\begin{equation*}
	\left( \begin{array}{l}
	\beta_0^h\\
	\vdots\\
	\beta_{n-1}^h\\
	-\beta_{0}^{2n-h}\\
	\vdots\\
	-\beta_{n-1}^{2n-h}\\
	\delta_h
	\end{array}
	\middle) = \Fa_h^{-1}\FX^{-1} \middle(\begin{array}{c}
	-(2n-h)\\
	\vdots\\
	-1\\
	0\\
	1\\
	\vdots\\
	h
	\end{array}\right)
	\end{equation*}
	
	Let $\FX^{-1}=(y_{ij})_{1 \leq i \leq 2n+1}$. It is well known that the $(i,j)$-th entry of the inverse matrix of a Vandermonde matrix is
	\begin{equation*}
	y_{ij}=\left\lbrace \begin{array}{cl}
	\dfrac{(-1)^{j-1}\mathlarger{\sum}_{\substack{1 \leq m_1 < \dots <m_{2n+1-j} \leq 2n+1\\ m_1, \dots,m_{2n+1-j} \neq i}}x_{m_1} \dots x_{m_{2n+1-j}}}{\mathlarger{\prod}_{1 \leq m \leq 2n+1, m \neq i}(x_m-x_i)} &, 1 \leq j <2n+1\\
	&\\
	\dfrac{1}{\mathlarger{\prod}_{1 \leq m \leq 2n+1, m \neq i}(x_m-x_i)}& ,j=2n+1.
	\end{array}\right..
	\end{equation*}
	
	We can think $y_{ij}$ as the $z^{2n+1-j}$-coefficient of
	\begin{equation*}
	\mathlarger{\prod}_{1 \leq m \leq 2n+1, m \neq i} \dfrac{(1-x_mz)}{(x_m-x_i)}.
	\end{equation*}
	
	Since $\beta_{n-1}^{n-1}=\alpha_{n,n-1}^{-1}\mathlarger{\sum}_{1 \leq j \leq 2n+1} y_{nj}(j-(n+2))$ we have
	
	\begin{equation*}
	\beta_{n-1}^{n-1}=-\alpha_{n,n-1}^{-1}\dfrac{d}{dz}\mathlarger{\mathlarger{\Biggl(}}\dfrac{\mathlarger{\prod}_{1 \leq m \leq 2n+1, m \neq n} (1-x_mz)}{z^{n-1}\mathlarger{\prod}_{1 \leq m \leq 2n+1, m \neq n}(x_m-x_n)}\mathlarger{\mathlarger{\Biggl)}}\mathlarger{\mathlarger{\Biggl|}}_{z=1}.
	\end{equation*}
	
	Note that $(1-x_{2n+1})=0$. Therefore we have,
	\begin{equation*}
	\beta_{n-1}^{n-1}=\alpha_{n,n-1}^{-1}\mathlarger{\mathlarger{\Biggl(}}\dfrac{\mathlarger{\prod}_{1 \leq m \leq 2n, m \neq n} (1-x_m)}{\mathlarger{\prod}_{1 \leq m \leq 2n+1, m \neq n}(x_m-x_n)}\mathlarger{\mathlarger{\Biggl)}}.
 	\end{equation*}
 	
 	Since $x_n=(-q)$ we have
 	\begin{equation*}
 	x_m-x_n=\left\lbrace\begin{array}{ll}
 	((-q)^{n-m}-1)(-q)& 1 \leq m \leq n-1\\
 	((-q)^{m-2n-2}-1)(-q)&n+1 \leq m \leq 2n+1,
 	\end{array}\right.,
 	\end{equation*}
 	
 	and 
 		\begin{equation*}
 	1-x_m=\left\lbrace\begin{array}{ll}
 	1-(-q)^{n+1-m}& 1 \leq m \leq n-1\\
 	1-(-q)^{m-2n-1}&n+1 \leq m \leq 2n
 	\end{array}\right..
 	\end{equation*}
 	
 	Combining these two and $\alpha_{n,n-1}=(-q)^{n+1}$, we have
 	\begin{equation*}
 	\beta_{n-1}^{n-1}=\dfrac{(1-(-q)^n)}{(-q)^{3n+1}(1-(-q))(1-(-q)^{-(n+1)})}.
 	\end{equation*}
\end{proof}

Now, we need the following lemma.

\begin{lemma}\label{lemmaA1}(cf. Lemma \ref{lemma3.17})
	Let $0\leq i\leq n-1$. If $\varphi \in U(M_i)$, then $\varphi(\pi M_i^{\vee}) \subset \pi M_i^{\vee}$.
\end{lemma}
\begin{proof}
	Let $\varphi(v_{l,i})=\mathlarger{\sum}_{k=1}^{2n} x_{kl}v_{k,i}$. Since $\varphi$ is an isometry, we have
	\begin{equation*}
	\begin{array}{l}
	\left( \begin{array}{lll}
	x_{11}^*&\dots&x_{2n1}^*\\
	\dots &\dots &\dots\\
	x_{12n}^* & \dots & x_{2n2n}^*
	\end{array}\middle)\middle(
	\begin{array}{ll}
	\pi 1_{2n-i} & 0 \\
	0 & 1_i  \\
	\end{array} \middle) \middle(
	\begin{array}{lll}
	x_{11} & \dots & x_{12n} \\
	\dots & \dots & \dots \\
	x_{2n1} & \dots & x_{2n2n}
	\end{array}
	\right)\\
	=\left( \begin{array}{ll}
	\pi 1_{2n-i} & 0 \\
	0 & 1_i
	\end{array}
	\right).
	\end{array}
	\end{equation*}
	
	Therefore, by reduction modulo $\pi$, this implies that
	
	\begin{equation}\label{equation A1}
	\left( \begin{array}{l}
	\mathlarger{\sum}_{m=2n-i+1}^{2n} x_{mk}^*x_{ml}
	\end{array}\middle)_{1\leq k,l \leq 2n}\equiv \middle(
	\begin{array}{ll}
	0 & 0  \\
	0 & 1_i 
	\end{array}
	\right) (\Mod \pi).
	\end{equation}
	
If we write
\begin{equation*}
A=\left(\begin{array}{lll}
x_{2n-i+1,1} & \dots & x_{2n-i+1,2n-i} \\
\dots&\dots&\dots\\
x_{2n,1}&\dots&x_{2n,2n-i}
\end{array}\right),
\end{equation*}

\begin{equation*}
B=\left(\begin{array}{lll}
x_{2n-i+1,2n-i+1} & \dots & x_{2n-i+1,2n} \\
\dots&\dots&\dots\\
x_{2n,2n-i+1}&\dots&x_{2n,2n}
\end{array}\right),
\end{equation*}

then \eqref{equation A1} implies that
\begin{equation*}
\left(\begin{array}{ll}
^tA^*A & ^tA^*B \\
^tB^*A&^tB^*B
\end{array}\middle) \equiv 
\middle(\begin{array}{ll}
0 & 0  \\
0 & 1_i 

\end{array}\right) (\Mod \pi).
\end{equation*}

Since $^tB^*B=1_i$, $B$ is invertible. Therefore $A$ is $0$ modulo $\pi$, and this implies that $\varphi(\pi M_i^{\vee}) \subset \pi M_i^{\vee}$.
\end{proof}

\begin{proposition}\label{propositionA1}(cf. \cite[Proposition 9.9]{KR2} and Proposition \ref{proposition3.18})
	\begin{enumerate}
		\item Let $\lbrace O_j \rbrace$ be a set of representatives for the $U(M_i)$-orbits in the set of all sublattices $O$ of $M_i$ such that $O$ is isometric to $L_2$. Then we have
		\begin{equation*}
		\begin{array}{ll}
		\vert J_d^1(L,M_i) \vert=&\sum_{j}\vert I_d(L_2,M_i;O_j) \vert\\
		&\times \vert \lbrace \varphi_1 \in J_d^1(L_1,M_i) \vert \langle \varphi(L_1), O_j \rangle \equiv 0 (\Mod \pi^d) \rbrace.
		\end{array}
		\end{equation*}
		
		\item Let $\lbrace N_j \rbrace$ be a set of representatives for the $U(M)$-orbits in the set of all sublattices $N$ of $M$ such that $N$ is isometric to $L_2$. Then we have
		\begin{equation*}
		\begin{array}{ll}
		\vert J_d^1(L,M) \vert=&\sum_{j}\vert I_d(L_2,M;N_j) \vert\\
		&\times \vert \lbrace \varphi_1 \in J_d^1(L_1,M) \vert \langle \varphi(L_1), N_j \rangle \equiv 0 (\Mod \pi^d) \rbrace.
		\end{array}
		\end{equation*}
	\end{enumerate}
	
\end{proposition}

\begin{proof}
	The proof of this proposition is identical to the proof of \cite[Proposition 9.9]{KR2} (and Proposition \ref{proposition3.18}).
\end{proof}

From now on, we fix $i=n-1$. Note that $L_2$ is unimodular. Therefore, by the arguments in \cite[p.680]{KR2}, all sublattices $O \subset M_{n-1}$ such that $O$ is isometric to $L_2$ are in the same $U(M_{n-1})$-orbit, and $M_{n-1}=O \perp O^{\perp}$ where $O^{\perp}:=(E_v O)^{\perp} \cap M_{n-1}$. Let us fix $O$ the sublattice of $M_{n-1}$ with the basis $\lbrace v_{n+2,n-1}, \dots, v_{2n,n-1} \rbrace$ as a representative of the unique $U(M_{n-1})$-orbit of $L_2$. Then, $O^{\perp}$ is the lattice with basis $\lbrace v_{1,n-1}, \dots, v_{n+1,n-1} \rbrace$. In particular, $O^{\perp} \subset \pi M_{n-1}^{\vee}$.

Also, let us fix $N$ the sublattice of $M$ with the basis $\lbrace v_{n+2}, \dots, v_{2n} \rbrace$. Then $N$ is a representative of the unique $U(M)$-orbit of $L_2$ and $M=N \perp N^{\perp}$.
We have the following analogue of \cite[Lemma 9.10]{KR2} (and Lemma \ref{lemma3.19}).
\begin{lemma}\label{lemmaA2}
	\begin{enumerate}
		\item For sufficiently large $d$, we have
		\begin{equation*}
		\begin{array}{l}
		\vert \lbrace \varphi_1 \in J_d^1(L_1,M_{n-1}) \vert \langle \varphi_1(L_1),O \rangle \equiv 0 \text{ }(\Mod \pi^d) \rbrace \vert=\vert I_d(L_1,O^{\perp})\vert
		\end{array}.
		\end{equation*}

		\item For sufficiently large $d$, we have
		\begin{equation}\label{eqA.0.3}
		\begin{array}{l}
		\vert \lbrace \varphi_1 \in J_d^1(L_1,M) \vert \langle \varphi_1(L_1),N \rangle \equiv 0 \text{ }(\Mod \pi^d) \rbrace \vert\\
		=q^{-2(n+1)}\vert I_d(L_1,N^{\perp} \cap \pi M^{\vee})\vert.
		\end{array}
		\end{equation}
		\end{enumerate}

\end{lemma}
\begin{proof}
	The proof of this lemma is almost identical to the proof of \cite[Lemma 9.10]{KR2} and Lemma \ref{lemma3.19}. Therefore, we will only prove (2). 
	First note that
	\begin{equation*}
	\lbrace x \in M \vert \langle x , N \rangle \equiv 0 \text{ } (\Mod \pi^d) \rbrace=\pi^d N^{\vee} \perp N^{\perp},
	\end{equation*}
	for sufficiently large $d$ such that $\pi^d N^{\vee} \subset N$.
	
	Therefore, we have
	\begin{equation*}
	\begin{array}{ll}
	\lbrace \varphi_1 \in J_d(L_1,M) \vert \quad\quad \langle \varphi_1(L_1) , N \rangle \equiv 0 \text{ } &(\Mod \pi^d) \rbrace\\
	=\lbrace \varphi_1:L_1 \rightarrow \pi^d N^{\vee} \perp N^{\perp} \text{ } (\Mod \pi^d) \vert& \langle \varphi_1(x), \varphi_1(y) \rangle \equiv \langle x,y\rangle \text{ } (\Mod \pi^d)\\
	&\varphi_1(L_1) \subset \pi M^{\vee} \rbrace.
	\end{array}
	\end{equation*}
	Since $\pi^{d}N^{\vee} \subset \pi M^{\vee}$ for sufficiently large $d$, we can write \eqref{eqA.0.3} as
	\begin{equation*}
	\lbrace \varphi_1:L_1 \rightarrow \pi^d N^{\vee} \perp (\pi M^{\vee} \cap N^{\perp}) \text{ } (\Mod \pi^d) \vert \langle \varphi_1(x), \varphi_1(y) \rangle \equiv \langle x,y\rangle \text{ } (\Mod \pi^d)\rbrace.
	\end{equation*}
	
	Now, fix a positive integer $a$ such that $\pi^a N^{\vee} \subset N$.
	
	Then we have 
	\begin{equation*}
	\begin{array}{lll}
	\quad \vert\lbrace \varphi_1:L_1 \rightarrow \pi^d N^{\vee} \perp (\pi M^{\vee} \cap N^{\perp}) \text{ } & (\Mod \pi^d) \vert \langle \varphi_1(x), \varphi_1(y) \rangle \equiv \langle x,y\rangle \text{ } (\Mod \pi^d)\rbrace\vert&\\
	=\vert\lbrace \varphi_1:L_1 \rightarrow \pi^d N^{\vee} \perp (\pi M^{\vee} \cap N^{\perp}) \text{ }& (\Mod \pi^d (\pi^aN^{\vee} \perp (\pi M^{\vee} \cap N^{\perp}))) \vert &\\
	&\langle \varphi_1(x), \varphi_1(y) \rangle \equiv \langle x,y\rangle \text{ } (\Mod \pi^d)\rbrace\vert\\
	\times \vert M:\pi^a N^{\vee} \perp (\pi M^{\vee} \cap N^{\perp}) \vert^{-(n+1)}
	\end{array}
	\end{equation*}
	by replacing $M$ by $\pi^a N^{\vee} \perp (\pi M^{\vee} \cap N^{\perp})$.
	
	Now, we can write $\varphi_1=\psi_1 + \psi_2$ where
	\begin{equation*}
	\begin{array}{l}
	\psi_1:L_1 \rightarrow \pi^d N^{\vee} / \pi^{d+a}N^{\vee},\\
	\psi_2:L_1 \rightarrow (\pi M^{\vee} \cap N^{\perp})/\pi^d (\pi M^{\vee} \cap N^{\perp}).
	\end{array}
	\end{equation*}
	
	Since we assume that $d$ is sufficiently large such that $\pi^d N^{\vee} \subset N$, we have
	\begin{equation*}
	\langle \psi_1(x),\psi_1(y) \rangle \in \langle N, \pi^d N^{\vee} \rangle \subset \pi^dO_{E_v}.
	\end{equation*}
	
	This means that the condition $\langle \varphi_1(x),\varphi_1(y) \rangle \equiv \langle x,y \rangle \text{ }(\Mod \pi^d)$ is the same as the condition $\langle \psi_2(x),\psi_2(y) \rangle \equiv \langle x,y \rangle \text{ } (\Mod \pi^d)$, and we do not need to impose any condition on $\psi_1$.
	
	Therefore, \eqref{eqA.0.3} is equal to
	\begin{equation*}
	\begin{array}{l}
	\vert M:\pi^aN^{\vee} \perp (\pi M^{\vee} \cap N^{\perp}) \vert^{-(n+1)}\vert \pi^d N^{\vee} : \pi^{d+a}N^{\vee}\vert^{n+1} \vert I_d(L_1,(\pi M^{\vee} \cap N^{\perp}))\vert\\
	=\vert M:N^{\vee}\perp (\pi M^{\vee} \cap N^{\perp}) \vert^{-(n+1)} \vert I_d(L_1,(\pi M^{\vee} \cap N^{\perp})) \vert\\
	=\vert M:N\perp (\pi M^{\vee} \cap N^{\perp}) \vert^{-(n+1)}\vert N^{\vee}:N \vert^{n+1} \vert I_d(L_1,(\pi M^{\vee} \cap N^{\perp})) \vert\\
	=q^{-2(n+1)}\vert I_d(L_1,(\pi M^{\vee} \cap N^{\perp})) \vert.
	\end{array}
	\end{equation*}
	Here, we used the fact that $\vert M:N \perp (\pi M^{\vee} \cap N^{\perp})\vert=q^2$ and $N=N^{\vee}$ (since $N$ is unimodular).
	
\end{proof}
	Proposition \ref{propositionA1} and Lemma \ref{lemmaA2} imply that
	\begin{equation}\label{eqA3}
	\begin{array}{l}
	\vert J_d^1(L,M_{n-1}) \vert=\vert I_d(L_2,M_{n-1})\vert \times \vert I_d(L_1,O^{\perp})\vert, \\
	\vert J_d^1(L,M) \vert=q^{-2(n+1)}\vert I_d(L_2,M) \vert \times\vert I_d(L_1,N^{\perp} \cap \pi M^{\vee})\vert.
	\end{array}
	\end{equation}
	
	Note that
	\begin{equation}\label{eqA4}
	\begin{array}{l}
	W_{n-1,n}(B,r)=q^{-2dm(2n)}q^{(d-1)(2n)^2}\vert J_d^1(L,M)\vert;\\
	W_{n-1,n-1}(B,0)=q^{-2d(2n)^2}q^{(d-1)(2n)^2}\vert J_d^1(L,M_{n-1})\vert;\\
	W_{n,n}(A_n,r)=q^{-2dm(2n)}q^{(d-1)(2n)^2}\vert J_d(\Delta,M)\vert,
	\end{array}
	\end{equation}
	for sufficiently large $d$. Here $\Delta$ is the lattice defined in the proof of the theorem \ref{theorem3.16}.

	Also, we have
	\begin{equation}\label{eqA5}
	\begin{array}{c}
	(q^{-d})^{(n-1)(2m-(n-1))} \vert I_d(L_2,M)\vert=\alpha(\pi A_n^{[r]},1_{n-1});\\
	\\
	(q^{-d})^{(n+1)(2(m-(n-1))-(n+1))} \vert I_d(L_1,N^{\perp} \cap \pi M^{\vee})\vert\\
	=\alpha(\left(\begin{array}{ccc}
	\pi 1_n & &\\
	& \pi^2 &\\
	& & \pi 1_{2r}
	\end{array}\right),\pi B_1);\\
	\\
	(q^{-d})^{(n-1)(2(2n)-(n-1))} \vert I_d(L_2,M_{n-1})\vert=\alpha(\pi A_{n-1},1_{n-1});\\
	\\
	(q^{-d})^{(n+1)(2(2n-(n-1))-(n+1))} \vert I_d(L_1,O^{\perp})\vert=\alpha(\pi 1_{n+1},\pi B_1);\\	\\
	
		(q^{-d})^{n(2m-n)} \vert I_d(\Delta_2,M)\vert=\alpha(\pi A_n^{[r]},1_{n});\\\\
		
		(q^{-d})^{n(2(m-n)-n)} \vert I_d(\Delta_1,N^{\perp})\vert=\alpha(\pi 1_{n+2r},\pi 1_{n});	
	\end{array}
	\end{equation}
	for sufficiently large $d$.
	
	Now, we need to compute these representation densities. We have the following proposition.
	
	\begin{proposition}\label{propositionA5} Let $k,m,n$ be integers such that $k,n \leq m$ and let $r=m-n$. Then we have
		\begin{equation*}
		\alpha(\left(\begin{array}{cc}
		\pi 1_k & \\
		& 1_{m-k}
		\end{array}\right), 1_n)=\prod_{l=1}^{n}(1-(-q)^{-l+k-r}).\\
		\end{equation*}
	\end{proposition}
	
	This formula is just a special case of \cite[Theorem II]{Hir2}.
	In the proof of this proposition, we will use the notation in \cite{Hir2} (and \cite[Page 921]{San}).
	
	Let $\Lambda_n^+:=\lbrace \lambda=(\lambda_1, \dots, \lambda_n) \in \BZ^n \vert \lambda_1 \geq \lambda_2 \geq \dots \geq \lambda_n \geq 0 \rbrace$.
	
	For each $\mu \in (\mu_1, \dots, \mu_n) \in \Lambda_n^+$, we define
	\begin{equation*}
	\begin{array}{l}
	\tilde{\mu}:=(\mu_{1}+1, \dots, \mu_n+1);\\
	\vert \mu \vert:=\sum_{i=1}^{n} \mu_n;\\
	n(\mu):=\sum_{i=1}^{n}(i-1)\mu_i;\\
	\mu_i':=\vert \lbrace j \vert \mu_j \geq i \rbrace \vert, \text{ } (i \geq 1);\\
	\pi^{\mu}:=\diag(\pi^{\mu_1}, \dots, \pi^{\mu_n}).
	\end{array}
	\end{equation*}
	
	Furthermore, for $\mu \in \Lambda_n^+$, $\xi \in \Lambda_m^+$, we define
	\begin{equation*}
	\langle \xi', \mu' \rangle:=\sum_{i \geq 1} \xi_i'\mu_i'.
	\end{equation*}
	
	For $\mu, \lambda \in \Lambda_n^+$, we denote by $\mu \leq \lambda$ if $\mu_i \leq \lambda_i$ for all $1 \leq i \leq n$.
	
	For $u \geq v \geq 0$, we define
	\begin{equation*}
	\left[ \begin{array}{l}
	u\\
	v
	\end{array}\right]:=\dfrac{\prod_{i=1}^u (1-(-q)^{-i})}{\prod_{i=1}^v (1-(-q)^{-i}) \prod_{i=1}^{u-v} (1-(-q)^{-i})}
	\end{equation*}
	
	Finally, for $\mu, \lambda \in \Lambda_n^+$, we define
	\begin{equation*}
	I_j(\mu, \lambda):=\sum_{i=\mu_{j+1}'}^{\min((\tilde{\lambda})'_{j+1}, \mu_j')}(-q)^{i(2(\tilde{\lambda})'_{j+1}+1-i)/2}\left[ \begin{array}{l}
	(\tilde{\lambda})_{j+1}'-\mu_{j+1}'\\
	(\tilde{\lambda})_{j+1}'-i
	\end{array}\middle] \middle[ \begin{array}{l} (\tilde{\lambda})_{j}'-i\\ (\tilde{\lambda})_{j}'-\mu_{j}' \end{array}\right].
	\end{equation*}
	
	Now, we can state the following proposition.
	
	\begin{proposition}
		(\cite[Theorem II]{Hir2}) For $\lambda \in \Lambda_n^+$ and $\xi \in \Lambda_m^+$ with $m \geq n$, we have
		\begin{equation*}
		\alpha(\pi^{\xi},\pi^{\mu})=\sum_{\substack{\mu \in \Lambda_n^+\\ \mu \leq \tilde{\lambda}}}(-1)^{\vert \mu \vert}(-q)^{-n(\mu)+(n-m-1)\vert \mu \vert+\langle \xi',\mu' \rangle} \prod_{j \geq 1} I_j(\mu,\lambda).
		\end{equation*}
	\end{proposition}

\begin{proof}[Proof of Proposition \ref{propositionA5}]
In our case, $\lambda=(0,0, \dots, 0) \in \Lambda_n^+$ and $\xi=(1^k,0^{m-k}) \in \Lambda_m^+$. Therefore, $\tilde{\lambda}=(1^n) \in \Lambda_n^+$ and hence $\mu$ runs over $(1^l,0^{n-l}) \in \Lambda_n^+$. We write $^l\mu$ for $(1^l,0^{n-l})$.

First, we need to compute $\prod_{j \geq 1} I_j(^l\mu,\lambda)$.

For $j=1$, we have
\begin{equation*}
I_1(^l\mu, \lambda):=\sum_{i=^l\mu_{2}'}^{\min((\tilde{\lambda})'_{2}, ^l\mu_1')}(-q)^{i(2(\tilde{\lambda})'_{2}+1-i)/2}\left[ \begin{array}{l}
(\tilde{\lambda})_{2}'- (^l\mu)_{2}'\\
(\tilde{\lambda})_{2}'-i
\end{array}\middle] \middle[ \begin{array}{l} (\tilde{\lambda})_{1}'-i\\ (\tilde{\lambda})_{1}'- (^l\mu)_{1}' \end{array}\right].
\end{equation*}

Since $(\tilde{\lambda})'_2=0$, $i$ should be $0$. Therefore, we have
\begin{equation*}
\begin{array}{ll}
I_1(^l\mu,\lambda)&=\left[ \begin{array}{c} 0\\0 \end{array} \middle] \middle[ \begin{array}{c} n \\ n-l \end{array} \right]\\
&=\left[ \begin{array}{c} n \\ n-l \end{array} \right]\\
&=\dfrac{\prod_{i=1}^n (1-(-q)^{-i})}{\prod_{i=1}^{n-l} (1-(-q)^{-i}) \prod_{i=1}^{l} (1-(-q)^{-i})}.

\end{array}
\end{equation*}

For $j>1$, it is easy to see that $I_j(^l\mu,\lambda)=1$.

Therefore,

\begin{equation*}
\begin{array}{ll}
\alpha(\pi^{\xi},\pi^{\mu})&=\sum_{l=0}^{n}(-1)^{\vert ^l\mu \vert}(-q)^{-n(^l\mu)+(n-m-1)\vert ^l\mu \vert+\langle \xi', ^l\mu' \rangle} \prod_{j \geq 1} I_j( ^l\mu,\lambda)\\
&=\sum_{l=0}^{n}(-1)^{\vert ^l\mu \vert}(-q)^{-n(^l\mu)+(n-m-1)\vert ^l\mu \vert+\langle \xi', ^l\mu' \rangle} \left[ \begin{array}{c} n \\ n-l \end{array}\right].

\end{array}
\end{equation*}

Since $\vert ^l\mu \vert=l$, $n( ^l\mu)=\dfrac{l(l-1)}{2}$, and $\langle \xi', ^l\mu' \rangle=kl$, we have

\begin{equation}\label{eqA.0.4}
\alpha(\pi^{\xi},\pi^{\mu})=\sum_{l=0}^{n}(-1)^{l}(-q)^{-l(l-1)/2-(r+1)l+kl} \left[ \begin{array}{c} n \\ n-l \end{array}\right].
\end{equation}

If $k=0$, it is well-known that
\begin{equation*}
\begin{array}{l}
\alpha(\pi^{\xi},\pi^{\mu})=\prod_{l=1}^{n} (1-(-q)^{-l}X).
\end{array}
\end{equation*}
where $X=(-q)^{-r}$ (see \cite[page 677]{KR2}).

Therefore, by changing $X$ to $(-q)^kX$, we have that \eqref{eqA.0.4} is equal to
\begin{equation*}
\begin{array}{ll}
\alpha(\pi^{\xi},\pi^{\mu})&=\prod_{l=1}^{n} (1-(-q)^{-l}(-q)^kX)\\
&=\prod_{l=1}^{n} (1-(-q)^{-l}(-q)^{k-r}).
\end{array}
\end{equation*}

\end{proof}

Now, we have the following proposition.

\begin{proposition}
	\begin{equation*}
	\begin{array}{ll}
	\dfrac{W_{n-1,n}'(B,0)}{W_{n,n}(A_n,0)}-\beta^{n-1}_{n-1}\dfrac{W_{n-1,n-1}(B,0)}{W_{n,n}(A_n,0)}=\\
	\dfrac{1}{q+1}\lbrace\dfrac{\alpha(\left(\begin{array}{cc} 1_n & \\
	& \pi 
	\end{array}\right),B_1)}{\alpha(1_n,1_n)}-\dfrac{\alpha(1_{n+1},B_1)}{\alpha(1_{n+1},1_{n+1})}\rbrace.
	\end{array}
	\end{equation*}
\end{proposition}
\begin{proof}
From \eqref{eq3.2.4}, \eqref{eqA3}, \eqref{eqA4}, \eqref{eqA5}, we have
	\begin{equation*}
	\begin{array}{l}
	W_{n-1,n}(B,r)=q^{-4n^2}q^{-2(n+1)}\alpha(\pi A_n^{[r]},1_{n-1})\alpha(\left(\begin{array}{ccc}
	\pi 1_n & &\\
	& \pi^2 &\\
	& & \pi 1_{2r}
	\end{array}\right),\pi B_1);\\
	W_{n-1,n-1}(B,0)=q^{-4n^2}\alpha(\pi A_{n-1},1_{n-1})\alpha(\pi 1_{n+1},\pi B_1);\\
	W_{n,n}(A_n,r)=q^{-4n^2}\alpha(\pi A_n^{[r]},1_n)\alpha(\pi 1_{n+2r},\pi 1_n).
	\end{array}
	\end{equation*}
	
	Note that 
	\begin{equation*}
	\val(\det(B)) \equiv n+1 (\Mod 2)
	\end{equation*} and hence $\val(\det(B_1))\equiv 0 (\Mod 2)$. Therefore $\alpha(\left(\begin{array}{cc}
	\pi 1_n & \\
	& \pi^2 \end{array}\right), \pi B_1)=0$.
	
	This implies that
	\begin{equation*}
	W_{n-1,n}'(B,0)=q^{-4n^2}q^{-2(n+1)}\alpha(\pi A_n,1_{n-1})\alpha'(\left(\begin{array}{ccc}
	\pi 1_n & \\
	& \pi^2
	\end{array}\right), \pi B_1).
	\end{equation*}
	
	By proposition \ref{propositionA5}, we have
	\begin{equation*}
	\begin{array}{l}
	\alpha(\pi A_n,1_{n-1})=\prod_{l=1}^{n-1}(1-(-q)^{-(l+1)})\\
	\alpha(\pi A_n,1_n)=\prod_{l=1}^{n}(1-(-q)^{-l})\\
	\alpha(\pi A_{n-1},1_{n-1})=\prod_{l=1}^{n-1}(1-(-q)^{-l}).
	\end{array}
	\end{equation*}
	
	Therefore,
	\begin{equation*}
	\begin{array}{ll}
	\dfrac{W_{n-1,n}'(B,0)}{W_{n,n}(A_n,0)}&=\dfrac{q^{-2(n+1)}\alpha(\pi A_n,1_{n-1})\alpha'(\left(\begin{array}{ccc}
			\pi 1_n & \\
			& \pi^2
			\end{array}\right), \pi B_1)}{\alpha(\pi A_n,1_n)\alpha(\pi 1_n,\pi 1_n)}\\
		&=\dfrac{q^{-2(n+1)}\prod_{l=1}^{n-1}(1-(-q)^{-(l+1)})q^{(n+1)^2}\alpha'(\left(\begin{array}{ccc}
			 1_n & \\
			& \pi
			\end{array}\right), B_1)}{\prod_{l=1}^{n}(1-(-q)^{-l})q^{n^2}\alpha(1_n,1_n)}\\
		
		&=\dfrac{1}{q+1} \dfrac{\alpha'(\left(\begin{array}{ccc}
			1_n & \\
			& \pi
			\end{array}\right), B_1)}{\alpha(1_n,1_n)}.
	\end{array}
	\end{equation*}
	
	Here, we used the fact that $\alpha(\pi C, \pi D)=q^{n^2}\alpha(C,D)$ where $D$ is a $n \times n$ hermitian matrix.
	
	Similarly,
	
	\begin{equation*}
	\begin{array}{ll}
	\beta^{n-1}_{n-1}\dfrac{W_{n-1,n-1}(B,0)}{W_{n,n}(A_n,0)}&=\beta^{n-1}_{n-1}\dfrac{\alpha(\pi A_{n-1},1_{n-1})\alpha(\pi 1_{n+1}, \pi B_1)}{\alpha(\pi A_n,1_n)\alpha(\pi 1_n,\pi 1_n)}\\\\
	&=\beta^{n-1}_{n-1}\dfrac{\prod_{l=1}^{n-1}(1-(-q)^{-l})q^{(n+1)^2}\alpha(1_{n+1}, B_1)}{\prod_{l=1}^{n}(1-(-q)^{-l})q^{n^2}\alpha(1_n,1_n)}\\\\
	
	&=\dfrac{1}{q+1} \dfrac{\alpha(1_{n+1}, B_1)}{\alpha(1_{n+1},1_{n+1})}.
	\end{array}
	\end{equation*}
	
	Here, we used Proposition \ref{propositionA} and 
	\begin{equation*}
	\dfrac{\alpha(1_{n+1},1_{n+1})}{\alpha(1_n,1_n)}=(1-(-q)^{-(n+1)}).
	\end{equation*}
	\end{proof}

	Note that $\dfrac{\alpha(\left(\begin{array}{cc} 1_n & \\
		& \pi 
	\end{array}\right),B_1)}{\alpha(1_n,1_n)}$
is equal to $\partial\text{Den}_{\Lambda}(L)$ and
$\dfrac{\alpha(1_{n+1},B_1)}{\alpha(1_{n+1},1_{n+1})}$
	is equal to $\text{Den}(L)$ in \cite[Theorem 10.4.4]{LZ} (but, our $n$ is $n-1$ in loc. cit.). Therefore, Conjecture \ref{conjecture2} is compatible with their result.

\end{document}